\documentclass[notitlepage,twoside,a4paper]{amsart}

\usepackage{mathrsfs}
\usepackage{amsmath,amssymb,enumerate}
\usepackage{epsfig,fancyhdr,color}
\usepackage{epstopdf}
\usepackage{amssymb}
\usepackage{amsmath,amsthm}
\usepackage{latexsym}
\usepackage{amscd}
\usepackage{psfrag}
\usepackage{graphicx}
\usepackage{epsf}
\usepackage[latin1]{inputenc}
\usepackage[all]{xy}
\usepackage{tikz}
\usetikzlibrary{decorations.pathreplacing}
\usepackage{prettyref}
\usepackage{subcaption}
\usepackage{float}
\usepackage{color}
\usepackage{array}
\usepackage{cite}
\usepackage[totalwidth=17cm,totalheight=24cm]{geometry}

\newcommand{\N}{{\mathbb N}}
\newcommand{\R}{{\mathbb R}}
\newcommand{\RP}{\mathbb{RP}}
\newcommand{\Z}{{\mathbb Z}}

\newcommand{\cT}{{\mathcal T}}

\newcommand{\AdS}{\mathbb{A}\rm{d}\mathbb{S}^3}

\newcommand{\HS}{\mathbb{HS}^3_1}
\newcommand{\HSS}{\mathbb{HS}^2}

\newcommand{\bAdS}{\mathbb{A}\rm{d}\mathbb{S}}

\newcommand{\bHS}{\mathbb{HS}}
\newcommand{\bH}{\mathbb{H}}

\newcommand{\PO}{\rm{PO}}
\newcommand{\PSL}{\rm{PSL}}

\newcommand{\Isom}{\rm{Isom}}
\newcommand{\cP}{\mathcal{P}}
\newcommand{\cA}{\mathcal{A}}
\newcommand{\E}{\mathbb{E}}
\newcommand{\cG}{\mathcal{G}}
\newcommand{\HH}{\mathbb{H}}
\newcommand{\cV}{\mathcal{V}}

\newcommand{\rO}{\rm{O}}
\newcommand{\Stab}{\rm{Stab}}
\newcommand{\bigzero}{\mbox{\normalfont\large\bfseries 0}}

\newcommand{\sgn}{\rm{sgn}}

\DeclareMathOperator{\area}{Area}

\DeclareMathOperator{\Graph}{\textsf{Graph}}
\DeclareMathOperator{\polyg}{\textsf{polyg}}

\newtheorem{theorem}{\rm\bf Theorem}[section]
\newtheorem{proposition}[theorem]{\rm\bf Proposition}
\newtheorem{lemma}[theorem]{\rm\bf Lemma}
\newtheorem{corollary}[theorem]{\rm\bf Corollary}
\newtheorem{definition}[theorem]{\rm\bf Definition}
\newtheorem{remark}[theorem]{\rm\bf Remark}

\newtheorem{observation}[theorem]{\rm\bf Observation}
\newtheorem{claim}[theorem]{\rm \bf Claim}
\newtheorem {Claim}{\rm\bf Claim}
\newtheorem*{condition1}{\bf{\textit{Proof for Condition (i)}}}
\newtheorem*{condition2}{\bf{\textit{Proof for Condition (ii)}}}
\newtheorem*{condition3}{\bf{\textit{Proof for Condition (iii)}}}
\newtheorem*{condition4}{\bf{\textit{Proof for Condition (iv)}}}
\newtheorem{Step}{\rm\bf Step}
\newtheorem{Step_metrics}{\rm\bf Step}

\def\interieur#1{\mathord{\mathop{\kern 0pt #1}\limits^\circ}}

\begin{document}

\title[Hyperideal anti-de Sitter polyhedra]{Hyperideal polyhedra in the 3-dimensional anti-de Sitter space}

\author{Qiyu Chen}
\address{Qiyu Chen:
School of Mathematics, South China University of Technology,
510641, Guangzhou, P. R. China.}
\email{chenqy0121@gmail.com}
\thanks{Q. Chen was partially supported by Shanghai Postdoctoral Excellence Program (No. 2018018) and China Postdoctoral
Science Foundation Grant (No. 2019M661326).}

\author{Jean-Marc Schlenker}
\address{Jean-Marc Schlenker:
University of Luxembourg, Mathematics Research Unit,
Maison du nombre, 6 avenue de la Fonte
L-4364 Esch-sur-Alzette, Luxembourg}
\email{jean-marc.schlenker@uni.lu}
\thanks{JMS was partially supported by FNR grants INTER/ANR/15/11211745, OPEN/16/11405402 and O20/14766753. The second author also acknowledge support from U.S. National Science Foundation grants DMS-1107452, 1107263, 1107367 ``RNMS: GEometric structures And Representation varieties'' (the GEAR Network).}

\date{v2, \today}

\begin{abstract}
  We study hyperideal polyhedra in the 3-dimensional anti-de Sitter space $\AdS$, which are defined as the intersection of the projective model of $\AdS$ with a convex polyhedron in $\mathbb{RP}^3$ whose vertices are all outside of $\AdS$ and whose edges all meet $\AdS$.
  We show that hyperideal polyhedra in $\AdS$ are uniquely determined by their combinatorics and dihedral angles, as well as by the induced metric on their boundary together with an additional combinatorial data, and describe the possible dihedral angles and the possible induced metrics on the boundary.
  \bigskip


\end{abstract}

	\maketitle
	
    \section{Introduction}

\subsection{Hyperbolic polyhedra}

The main motivation here is the beautiful theory of polyhedra in the 3-dimensional hyperbolic space $\HH^3$. Using the Klein (projective) model of $\HH^3$ as a ball in $\R^3$, it is possible to consider ``compact'' hyperbolic polyhedra that are fully contained in $\HH^3$, but also ``ideal'' polyhedra --- those with their vertices on the boundary of $\HH^3$ --- and ``hyperideal'' polyhedra (as an extension of ideal polyhedra) --- with all their vertices outside of $\HH^3$, but all edges intersecting $\HH^3$.

Those hyperbolic polyhedra can be described in terms of two types of boundary quantities.
\begin{itemize}
\item The induced metric on the part of the boundary contained in $\HH^3$. For compact polyhedra, this is a hyperbolic metric with cone singularities of angle less than $2\pi$ on the sphere, and Alexandrov \cite{alex} proved that each such metric is obtained on a unique compact polyhedron (up to isometries). 
    For hyperideal polyhedra, the induced metrics are complete hyperbolic metrics on punctured spheres, possibly of infinite area, and each such metric is obtained on a unique hyperideal polyhedron (up to isometries), see \cite{rivin-comp,shu}, and the metric has finite area if and only if the polyhedron is ideal.

  \item Ideal and hyperideal polyhedra are characterized by their combinatorics and dihedral angles, and the possible dihedral angles are described by a simple set of linear equations and inequalities \cite{Andreev-ideal,rivin-annals,bao-bonahon,rousset1}. It is not yet known whether compact hyperbolic polyhedra are uniquely determined by their combinatorics and dihedral angles --- this statement is known as the ``hyperbolic Stoker conjecture'', see \cite{stoker} --- but this holds locally \cite{mazzeo-montcouquiol}. However those polyhedra are uniquely determined by the {\em dual metric}, the induced metric on the dual polyhedron in the de Sitter space \cite{HR}.
\end{itemize}
Understanding hyperbolic polyhedra in terms of their dihedral angles is relevant in a number of areas of hyperbolic geometry, where polyhedra are used as ``building blocks'' to construct hyperbolic manifolds or orbifolds.

We are particularly interested here in hyperideal polyhedra, so we recall the result of Bao and Bonahon \cite{bao-bonahon} on their dihedral angles. A graph is said to be \emph{3-connected} if it cannot be disconnected or reduced to a single point by removing 0, 1 or 2 vertices and their incident edges. Let $\Gamma$ be a 3-connected planar graph and let $\Gamma^*$ denote the dual graph of $\Gamma$. For each edge $e\in E(\Gamma)$, we denote by $e^{*}\in E(\Gamma^{*})$ the dual edge of $e$. Let $\theta: E(\Gamma)\rightarrow \mathbb{R}$ be a weight function on $E(\Gamma)$. It was shown in \cite{bao-bonahon} that $\Gamma$ is isomorphic to the 1-skeleton of a hyperideal polyhedron in $\mathbb{H}^3$ with exterior dihedral angle $\theta(e)$ at the edge $e\in E(\Gamma)$ if and only if the weight function $\theta: E(\Gamma)\rightarrow \mathbb{R}$ satisfies the following four conditions:
\begin{enumerate}[(1)]
\item $0<\theta(e)<\pi$ for each edge of $\Gamma$.
\item If $e^*_1, \ldots, e^*_k$ bound a face of $\Gamma^*$, then $\theta(e_1)+\cdots+\theta(e_k)\geq 2\pi$.
\item If $e^*_1, \ldots, e^*_k$ form a simple circuit which does not bound a face of $\Gamma^*$, then $\theta(e_1)+\cdots+\theta(e_k)> 2\pi$.
\item If $e^*_1, \ldots, e^*_k$ form a simple path starting and ending on the boundary of the same face of $\Gamma^*$, but not contained in the boundary of that face, then $\theta(e_1)+\cdots+\theta(e_k)> \pi$.
\end{enumerate}

This implies that $\Gamma$ is isomorphic to the 1-skeleton of a hyperideal polyhedron in $\mathbb{H}^3$ if and only if there is a weight function $\theta: E(\Gamma)\rightarrow \mathbb{R}$ satisfying Conditions (1)-(4). Note that the equality in Condition (2) holds if and only if the vertex of $\Gamma$, dual to the face of $\Gamma^*$ bounded by the edges $e^*_1, e^*_2, \ldots, e^*_k$, is ideal (see e.g. \cite{bao-bonahon}).

\subsection{Anti-de Sitter geometry}\label{subsec:AdS geodesics and planes}

The anti-de Sitter (AdS) space can be considered as a Lorentzian cousin of hyperbolic space. We provide a more detailed description in Section \ref{sect:AdS} and just recall some key properties here. We refer the reader to e.g. \cite{bonsante-seppi:anti} for more details and analysis.

The (projective model of) 3-dimensional AdS space is defined as the projectivization in $\RP^3$ of the set of timelike vectors in $\R^{2,2}$, which is $\R^4$ equipped with a bilinear symmetric form $\langle \cdot,\cdot \rangle_{2,2}$ of signature $(2,2)$:
$$ {\AdS} = \{[x]\in \RP^{3}~|~ \langle x,x\rangle _{2,2}<0\} ~. $$

It is a Lorentzian space of constant curvature $-1$. $\AdS$ is geodesically complete but not simply connected, its fundamental group is $\Z$. There is a deep analogy between quasifuchsian hyperbolic manifolds and the so-called globally hyperbolic maximal compact (GHMC) AdS spacetimes, see \cite{mess,mess-notes,bonsante-seppi:anti}.

We consider $\AdS$ as an oriented and time-oriented space, so that there is at each point a notion of future and past time directions.

\subsection{Anti-de Sitter polyhedra}

We are interested in the notion of hyperideal polyhedra, given by the following definition, which extends the notion of ideal polyhedra in $\AdS$.

Note that the polyhedra we consider throughout the paper are \emph{convex} in $\RP^3$ (i.e. they are contained in an affine chart of $\RP^3$ and are convex in that affine chart). A polyhedron in $\RP^3$ is \emph{non-degenerate} if its boundary bounds a non-empty (3-dimensional) domain in $\RP^3$, otherwise, it is \emph{degenerate} (in this case the polyhedron is a two-sided polygon and its boundary is identified with the two copies of the polygon glued along the union of its edges).

\begin{definition}
A (convex) hyperideal polyhedron in $\AdS$ is the intersection of $\AdS$ with a convex polyhedron $P'$ in $\mathbb{RP}^3$ whose vertices are all outside of $\AdS$ and whose edges all intersect $\AdS$.
\end{definition}

It is clear that ideal polyhedra are special cases of hyperideal polyhedra. Without causing confusion, when we talk about the vertices, edges, faces and 1-skeleton of a hyperideal polyhedron $P$, we always mean the corresponding vertices, edges, faces and 1-skeleton of the polyhedron $P'$ in $\mathbb{RP}^3$ with $P'\cap\AdS=\mathnormal{P}$, and we will sometimes consider $P'$ as a ``hyperideal AdS polyhedron''.

Let $v$ be a vertex of a hyperideal AdS polyhedron. We say that $v$ is \emph{ideal} if $v$ lies on the boundary of $\AdS$. We say that $v$ is \emph{strictly hyperideal} if $v$ lies outside of the closure of $\AdS$ in $\mathbb{RP}^3$. 
A hyperideal AdS polyhedron is said to be \emph{ideal} (resp. \emph{strictly hyperideal}) if all its vertices are ideal (resp. strictly hyperideal).

\subsection{Hyperideal polyhedra in $\AdS$}\label{subsec:polyhedra}

Let $P$ be a hyperideal AdS polyhedron. Note that all the edges and faces of $P$ are space-like (see Lemma \ref{lm:face_spacelike} for more details). A face of $P$ is called \emph{future-directed} (resp. \emph{past-directed}), if the outward-pointing unit normal vector is future-pointing (resp. past-pointing) at each point of the interior of that face. Therefore, the faces of $P$ are sorted into two types (i.e. future-directed and past-directed), delimited on $\partial P$ by a \emph{Hamiltonian cycle}, which is a closed path following edges of $P$ and visiting each vertex of $P$ exactly once. We call this Hamiltonian cycle the \emph{equator} of $P$, see e.g. \cite{DMS}. 

Each hyperideal polyhedron $P$ in $\AdS$ with $N$ ($N\geq 3$) vertices is associated to a \emph{marking}, which is an identification, up to isotopy, of the equator of $P$ with the oriented $N$-cycle graph (whose vertices are labelled consecutively) so that the induced ordering of the vertices is positive with respect to the orientation and time-orientation of $\AdS$. 
Let $\tilde{\mathcal{P}}_N$ ($N\geq 4$) be the space of all marked non-degenerate hyperideal polyhedra in $\AdS$ with $N$ vertices.  Note that the group $\Isom_0(\AdS)$ of orientation and time-orientation preserving isometries of $\AdS$ has a natural action on $\tilde{\mathcal{P}}_N$. We denote by $\mathcal{P}_{N}$ the quotient space of $\tilde{\mathcal{P}}_N$ by this action of $\Isom_0(\AdS)$. For simplicity, we call both an element of $\mathcal{P}_N$ and its representative a hyperideal polyhedron in $\AdS$ henceforth.

Fix an orientation on $\Sigma_{0,N}$ (where $\Sigma_{0,N}$ is a 2-sphere with $N\geq 3$ marked points) and an oriented simple closed curve $\gamma$ on $\Sigma_{0,N}$ which visits each marked point exactly once, and then label the marked points in order along the path. Following the same notations as in \cite{DMS}, we call the component of $\Sigma_{0,N}\setminus\gamma$ on the left (resp. right) side of $\gamma$ the \emph{top} (resp. \emph{bottom}).
It is natural to identify each hyperideal polyhedron in $\AdS$ with $\Sigma_{0,N}$ via the isotopy class of the map taking each vertex to the corresponding marked point, the equator to $\gamma$, and the union of future-directed (resp. past-directed) faces to the top (resp. bottom) of $\Sigma_{0,N}$.

It is a classical theorem of Steinitz that a graph $\Gamma$ is isomorphic to the 1-skeleton of a (non-degenerate convex) polyhedron in $\R^3$ if and only if $\Gamma$ is planar and 3-connected, see e.g. \cite[Chapter 4]{ziegler:polytopes}.
Let $\Graph(\Sigma_{0,N},\gamma)$ ($N\geq 4$) be the set of 3-connected graphs embedded in $\Sigma_{0,N}$, such that the vertices are located at the marked points and the edge set contains the edges of $\gamma$, up to a homeomorphism isotopic to the identity (fixing each marked point). Via the marking, any hyperideal polyhedron in $\tilde{\mathcal{P}}_N$ has its 1-skeleton identified to a unique graph in $\Graph(\Sigma_{0,N},\gamma)$.
If two hyperideal polyhedra $P_1, P_2\in\tilde{\mathcal{P}}_N$ are equivalent, then their 1-skeletons are isomorphic to the same graph in $\Graph(\Sigma_{0,N},\gamma)$. Therefore, the 1-skeleton of an element $P$ of $\mathcal{P}_N$ is always well-defined and is identified to a graph  $\Gamma$ in $\Graph(\Sigma_{0,N},\gamma)$. We say that $\Gamma$ is \emph{realized} as the 1-skeleton of $P$, or directly that $\Gamma$ is the 1-skeleton of $P$.

Let $\Sigma$ be an oriented piecewise totally geodesic space-like surface in $\AdS$. Let $H$ and $H'$ be two (totally geodesic) faces of $\Sigma$ meeting along a common space-like geodesic segment $l$ and let $U$ be a small convex neighborhood in $\RP^3$ of an interior point $x$ of $l$, which intersects $\Sigma$ at only two faces $H$, $H'$. We say that $\Sigma$ is \emph{convex} (resp. \emph{concave}) at $l$ if the component of $U\setminus \Sigma$ lying on the other side of the time-like normal vectors (determined by the orientation of $\Sigma$) is
convex (resp. concave) in $\RP^3$ (see e.g. Figure \ref{fig:dihedral angles}).

We define the \emph{exterior dihedral angle} at $l$ of $\Sigma$ as follows. Consider the isometry of $\AdS$ that fixes the complete space-like geodesic $L\supset l$ point-wise and maps the plane of $H'$ to the plane of $H$, it is a hyperbolic rotation in the group $\rm{O}(1,1)$ of a time-like plane orthogonal to $L$ with the rotation amount, say $|\theta|$, such that
\begin{equation}\label{def:angle amount}
\cosh|\theta|=|\langle n,n'\rangle_{2,1}|~,
\end{equation}
where $n$ (resp. $n'$) is the oriented unit normal vector to $H$ (resp. $H'$) at a point $x\in L$ (note that the set of unit timelike tangent vectors at $x$ is isometric to $\{x\in\R^{2,1}: \langle x,x\rangle_{2,1}=-1\}$, two copies  of the hyperboloid model of hyperbolic 2-space). $|\theta|$ is thus valued in $\R_{\geq 0}$ rather than in the circle $S^1$ (see e.g. \cite{DMS}).
The \emph{sign} of $\theta$ is defined as follows. If $H$ and $H'$ lie in the opposite space-like quadrants (locally divided by the light-cone of $l$ in $\AdS$), we take $\theta$ to be non-negative (resp. positive) if the surface $\Sigma$ is convex (resp. strictly convex) at $l$, and negative if $\Sigma$ is strictly concave at $l$. If $H$ and $H'$ lie in the same space-like quadrants, we take $\theta$ to be non-positive (resp. negative) if the surface $\Sigma$ is is convex (resp. strictly convex) at $l$, and positive if $\Sigma$ is strictly concave at $l$ (see Figure \ref{fig:dihedral angles} for instance).

\begin{figure}
\begin{subfigure}[b]{0.46\textwidth}
\centering
\begin{tikzpicture}[scale=1.1]
\draw[white][fill=gray,opacity=0.4] (-1,3.5) .. controls (-1.35,3.4) and (-1.35,3.4) .. (-1.35,3.4) .. controls (-1.25,3.05) and (-0.7,3.1) .. (-0.65,3.4) .. controls (-0.65,3.4) and (-1,3.5) .. (-1,3.5);
\draw [dashed](-2.5,2) -- (0.5,5);
\draw [dashed](0.5,2) -- (-2.5,5);
\draw (-1,3.5) .. controls (-1.15,3.15) and (-1.15,3.15) .. (-1.15,3.15) .. controls (-3.9,2.4) and (-3.9,2.4) .. (-3.9,2.4) .. controls (-3.7,3.1) and (-3.7,3.1) .. (-3.7,3.1) .. controls (-2.9,3.3) and (-2.9,3.3) .. (-2.9,3.3);
\draw (-1,3.5) .. controls (-1.15,3.15) and (-1.15,3.15) .. (-1.15,3.15) .. controls (1.6,2.3) and (1.6,2.3) .. (1.6,2.3) .. controls (1.9,2.95) and (1.9,2.95) .. (1.9,2.95) .. controls (0.9,3.25) and (0.9,3.25) .. (0.9,3.25);
\draw [thick](-1,3.5) .. controls (-2.8,3) and (-2.8,3) .. (-2.8,3);
\draw [thick](-1,3.5) .. controls (0.8,2.95) and (0.8,2.95) .. (0.8,2.95);
\node at (-3.35,2.85) {$H$};
\node at (1.4,2.8) {$H'$};
\node at (1.2,5.1) {$F$};
\node at (-1.15,2.9) {$l$};
\node at (-0.75,3.55) {$x$};
\draw[-latex](-1,3.5)--(-0.6,4.2);
\draw[-latex](-1,3.5)--(-1.1,4.15);
\node at (-1.25,4.35) {$n'$};
\node at (-0.5,4.35) {$n$};
\draw (-2.9,5.4) .. controls (-2.9,5.4) and (0.9,5.4) .. (0.9,5.4) .. controls (0.9,5.4) and (0.9,2.9) .. (0.9,2.9) .. controls (0.9,2.9) and (0.9,2.9) .. (0.9,2.9);
\draw (0.9,2.5) .. controls (0.9,1.6) and (0.9,1.6) .. (0.9,1.6) .. controls (-2.9,1.6) and (-2.9,1.6) .. (-2.9,1.6) .. controls (-2.9,2.65) and (-2.9,2.65) .. (-2.9,2.65);
\draw (-2.9,2.95) .. controls (-2.9,2.95) and (-2.9,2.95) .. (-2.9,5.4);
\draw[densely dotted] (0.9,2.9) .. controls (0.9,2.5) and (0.9,2.5) .. (0.9,2.5);
\draw[densely dotted] (-2.9,2.95) .. controls (-2.9,2.65) and (-2.9,2.65) .. (-2.9,2.65);
\draw[densely dotted] (-2.9,3.3) .. controls (-0.9,3.8) and (-0.9,3.8) .. (-0.9,3.8) .. controls (0.9,3.25) and (0.9,3.25) .. (0.9,3.25);
\draw [densely dotted](-0.9,3.8) .. controls (-1,3.5) and (-1,3.5) .. (-1,3.5);
\end{tikzpicture}
\end{subfigure}
\begin{subfigure}[b]{0.46\textwidth}
\centering
\begin{tikzpicture}[scale=1.1]
\draw[white][fill=gray,opacity=0.4] (2,-0.5) .. controls (2.35,-0.25) and (2.35,-0.25) .. (2.35,-0.25) .. controls (2.5,-0.35) and (2.5,-0.5) .. (2.4,-0.65) .. controls (2.4,-0.65) and (2.4,-0.65) .. (2,-0.5);
\draw[dashed] (0.5,-2) -- (3.5,1) -- cycle;
\draw[dashed] (0.5,1) -- (3.5,-2);
\draw (1.35,-0.65)        (1.35,-0.65);
\draw (2,-0.5)     .. controls (1.35,-0.65) and (1.35,-0.65) .. (1.35,-0.65);
\draw [thick](2,-0.5) .. controls (3.8,0.85) and (3.8,0.85) .. (3.8,0.85);
\draw[thick] (2,-0.5) .. controls (3.75,-1.25) and (3.75,-1.25) .. (3.75,-1.25);
\node at (4.4,1.3) {$H$};
\node at (4.35,-1.4) {$H'$};
\node at (1.1,-0.7) {$l$};
\node at (1.95,-0.7) {$x$};
\node at (2.05,1.8) {$F$};
\draw (0.1,1.5) .. controls (3.9,1.5) and (3.9,1.5) .. (3.9,1.5) .. controls (3.9,0.95) and (3.9,0.95) .. (3.9,0.95);
\draw (3.9,0.55) .. controls (3.9,-1.3) and (3.9,-1.3) .. (3.9,-1.3);
\draw(3.9,-1.7) .. controls (3.9,-2.3) and (3.9,-2.3) .. (3.9,-2.3) .. controls (3.9,-2.3) and (0.1,-2.3) .. (0.1,-2.3) .. controls (0.1,1.5) and (0.1,1.5) .. (0.1,1.5);
\draw[densely dotted] (3.9,0.95) .. controls (3.9,0.55) and (3.9,0.55) .. (3.9,0.55);
\draw[densely dotted] (3.9,-1.3) .. controls (3.9,-1.7) and (3.9,-1.7) .. (3.9,-1.7);
\draw [densely dotted](2.65,-0.35) .. controls (3.9,0.55) and (3.9,0.55) .. (3.9,0.55);
\draw [densely dotted](2,-0.5) .. controls (2.65,-0.35) and (2.65,-0.35) .. (2.65,-0.35);
\draw (2,-0.5) .. controls (1.35,-0.65) and (1.35,-0.65) .. (1.35,-0.65);
\draw (1.35,-0.65)     .. controls (1.75,-0.35) and (1.75,-0.35) .. (1.75,-0.35);
\draw [densely dotted](2.65,-0.35) .. controls (3.9,-0.85) and (3.9,-0.85) .. (3.9,-0.85);
\draw (3.9,-0.85) .. controls (5.55,-1.55) and (5.55,-1.55) .. (5.55,-1.55) .. controls (4.3,-1.9) and (4.3,-1.9) .. (4.3,-1.9) .. controls (1.35,-0.65) and (1.35,-0.65) .. (1.35,-0.65);
\draw(3.9,1.3) .. controls (3.9,1.3) and (3.9,1.3) .. (3.9,1.3) .. controls (3.9,1.3) and (4.35,1.65) .. (4.35,1.65) .. controls (4.35,1.65) and (5.55,1.85) .. (5.55,1.85) .. controls (5.55,1.85) and (3.9,0.55) .. (3.9,0.55);
\draw [densely dotted](1.75,-0.35) .. controls (3.9,1.3) and (3.9,1.3) .. (3.9,1.3);
\draw[-latex](2,-0.5)--(2.45,0.15);
\draw[-latex](2.2,-1)--(2.3,-1.25);
\draw [densely dotted](2,-0.5) .. controls (2.2,-1) and (2.2,-1) .. (2.2,-1);
\node at (2.3,0.3) {$n$};
\node at (2.35,-1.4) {$n'$};
\end{tikzpicture}
\end{subfigure}
\caption{\small{Two examples of $\Sigma$ which is convex at $l=H\cap H'$ (with normal vectors $n$ and $n'$ to the spacelike planes of $H$ and $H'$ at $x\in l$). Here $F$ is the timelike plane orthogonal to $l$ at $x$, and the shaded region is the component of $(U\setminus\Sigma)\cap F$ lying on 
the other side of 
$n$ and $n'$ for a convex neighborhood $U$ of $x$. In the left (resp. right) picture, the intersection lines $F\cap H$ and $F\cap H'$ (shown in bold lines) lie in the opposite (resp. the same) spacelike quadrant separated by the lightlike geodesics (shown in dashed lines) in $F$ and the exterior dihedral angle at $l$ is positive (resp. negative).}}
\label{fig:dihedral angles}
\end{figure}
Let $P$ be a convex polyhedron in $\AdS$. The boundary $\partial P$ of $P$ in $\AdS$ equipped with outward-pointing normal vectors is an oriented piecewise totally geodesic space-like surface in $\AdS$. The \emph{exterior dihedral angle} of $P$ at an edge $e$ is defined as the exterior dihedral angle at the edge $e$ of the oriented surface $\partial P$.

If $\Sigma$ is an oriented piecewise totally geodesic surface in $\AdS$ and $H$, $H'$ are totally geodesic space-like and time-like faces of $\Sigma$ meeting orthogonally along 
a common space-like geodesic $l$, we define the \emph{exterior dihedral angle} $\theta$ at $l$ between $H$ and $H'$ to be zero (see e.g. Case (b) in Definition \ref{def:angles}).

Unless otherwise stated, the dihedral angles we consider throughout this paper are exterior dihedral angles.

\subsection{Main results}
\label{ssc:main}

We can now state the main results, describing the dihedral angles and induced metrics on hyperideal polyhedra in $\AdS$.

\begin{definition}\label{def:admissible angles}
Let $\Gamma\in \Graph(\Sigma_{0,N},\gamma)$. We denote by $\Gamma^*$ the dual graph of $\Gamma$ and by $e^*\in E(\Gamma^*)$ the dual edge of $e$. We say that a function $\theta: E(\Gamma)\rightarrow \R$ is $\gamma$-admissible if it satisfies the following four conditions:
\begin{enumerate}[(i)]
\item $\theta(e)<0$ if $e$ is an edge of the equator $\gamma$, and $\theta(e)>0$ otherwise.
\item If $e^*_1, \ldots, e^*_k$ bound a face of $\Gamma^*$, then $\theta(e_1)+\cdots+\theta(e_k)\geq 0$.
\item If $e^*_1, \ldots, e^*_k$ form a simple circuit which does not bound a face of $\Gamma^*$,  and such that exactly two of the edges are dual to edges of $\gamma$, then $\theta(e_1)+\cdots+\theta(e_k)> 0$.
\item If $e^*_1, \ldots, e^*_k$ form a simple path starting and ending on the boundary of the same face of $\Gamma^*$, but not contained in the boundary of that face, and such that exactly one of the edges is dual to one edge of $\gamma$, then $\theta(e_1)+\cdots+\theta(e_k)> 0$.
\end{enumerate}
\end{definition}

\begin{theorem}
\label{thm:dihedral angles AdS}
Let $\Gamma\in \Graph(\Sigma_{0,N},\gamma)$ and let $\theta: E(\Gamma)\rightarrow \R$ be a function. Then $\theta$ can be realized as the exterior dihedral angles at the edges of a hyperideal AdS polyhedron with 1-skeleton $\Gamma$ if and only if $\theta$ is $\gamma$-admissible.In that case, $\theta$ is realized on a unique hyperideal AdS polyhedron (up to isometries).
\end{theorem}

The uniqueness stated here is of course up to global (orientation and time-orientation preserving) isometries of $\AdS$.

Let $P$ be a hyperideal AdS polyhedron with 1-skeleton $\Gamma\in \Graph(\Sigma_{0,N},\gamma)$ and equator $\gamma$ and let $\theta: E(\Gamma)\rightarrow \mathbb{R}$ be the function assigning to each edge $e\in E(\Gamma)$ the exterior dihedral angle $\theta(e)$ at $e$ of $P$. We will show in Section \ref{section:necessity} that $\theta$ is $\gamma$-admissible.  Moreover, the equality in (ii) holds if and only if the vertex of $P$, dual to the face of $\Gamma^*$ bounded by $e^*_1, \ldots, e^*_k$, is ideal.

Conditions (i)-(iv) in Definition \ref{def:admissible angles}  can be viewed as modified versions of Conditions (1)-(4) in the angle conditions for hyperideal polyhedra in $\mathbb{H}^3$. It is worth mentioning that for a general graph $\Gamma\in\Graph(\Sigma_{0,N},\gamma)$ Condition (iv) in Definition \ref{def:admissible angles} does not follow from Conditions (i)-(iii). It follows only in some special cases (for instance, every vertex of $\Gamma$ has degree 3).

Given a 3-connected planar graph $\Gamma$ which admits a Hamiltonian cycle, we fix an orientation of the Hamiltonian cycle. There is an embedding of $\Gamma$ into $\Sigma_{0,N}$ such that the vertices of $\Gamma$ correspond to the marked points of $\Sigma_{0,N}$, and the Hamiltonian cycle of $\Gamma$ corresponds to $\gamma$ with the same orientation. This embedding is unique up to isotopy among maps sending a fixed vertex of $\Gamma$ to a fixed marked point of $\Sigma_{0,N}$. Therefore, we can identify this graph $\Gamma$ with an element of $\Graph(\Sigma_{0,N},\gamma)$. By Theorem \ref{thm:dihedral angles AdS}, the existence of a hyperideal AdS polyhedron with 1-skeleton $\Gamma$ is equivalent to the existence of a $\gamma$-admissible function on $E(\Gamma)$. Moreover, Definition \ref{def:admissible angles} tells us the existence of a $\gamma$-admissible function on $E(\Gamma)$ is equivalent to the existence of the solution to a system of linear inequalities with finitely many variables (corresponding to the edges of $\Gamma$) given by Condition (i)-(iv). One can check that such a solution always exists by choosing the function on $E(\Gamma)$  such that it satisfies Condition (i) and has large enough angles at the non-equatorial edges (see Claim \ref{clm:topology of A_Gamma} for more details). We therefore have the following.

\begin{corollary}
\label{thm:combinatorics AdS}
For any 3-connected planar graph $\Gamma$ which admits a Hamiltonian cycle $\gamma$, there exists a hyperideal AdS polyhedron $P$ whose 1-skeleton is $\Gamma$ with equator $\gamma$.
\end{corollary}

\begin{theorem}
\label{thm:induced metrics}
Let $P$ be a hyperideal (possibly degenerate) polyhedron in $\AdS$ with $N$
$(N\geq 3)$ vertices. Then the induced metric on $\partial P$ is a complete hyperbolic metric (of infinite area if at least one vertex is strictly hyperideal) on $\Sigma_{0,N}$. Conversely, each complete hyperbolic metric on $\Sigma_{0,N}$, possibly with infinite area, is induced on a unique (up to isometries) marked hyperideal AdS polyhedron.
\end{theorem}

Note that the existence and uniqueness here is for the hyperbolic metric on $\Sigma_{0,N}$, considered up to isotopies. The hyperideal AdS polyhedron realizing a given hyperbolic metric $h$ depends on the position of the ``equator'', and if two metrics $h$ and $h'$ on $\Sigma_{0,N}$ are isometric by an isometry not isotopic to the identity, and not preserving the equator, then the corresponding polyhedra $P$ and $P'$ might not be related by an isometry of $\AdS$. As it will be clear below, hyperideal polyhedra in $\AdS$ are uniquely determined by their induced metric together with the equator, which needs to be a simple closed curve visiting each vertex exactly once.

Theorems \ref{thm:dihedral angles AdS} and \ref{thm:induced metrics} extend to hyperideal polyhedra results obtained in \cite{DMS} for ideal polyhedra in $\AdS$.

\subsection{Outline of the proofs and organization}

Let $\Gamma\in\Graph(\Sigma_{0,N},\gamma)$ and let $\cP_{\Gamma}$ denote the subset of $\cP_N$ ($N\geq 4$) in which every polyhedron has its 1-skeleton identified with $\Gamma$.
Note that each polyhedron in $\cP_N$ realizes a graph in $\Graph(\Sigma_{0,N},\gamma)$, $\cP_N$ is the disjoint union of $\cP_{\Gamma}$ over all $\Gamma\in\Graph(\Sigma_{0,N},\gamma)$.
Let $\cA_{\Gamma}$ be the space of $\gamma$-admissible weight functions $\theta\in \R^{E(\Gamma)}$. 
We denote by $\cA_N$ the disjoint union of $\cA_{\Gamma}$ over all $\Gamma\in\Graph(\Sigma_{0,N},\gamma)$. (More details about $\cP_N$ and $\cA_N$ can be found in 
subsection \ref{subsec:def}.)

Let $\Psi_{\Gamma}:\cP_{\Gamma}\rightarrow \cA_{\Gamma}$ be the map which assigns to each $P\in\cP_{\Gamma}$ the exterior dihedral angle-weight function $\theta\in\R^{E}$. Now we consider the map $\Psi: \cP_N\rightarrow \cA_N$, which is defined by $\Psi(P)=\Psi_{\Gamma}(P)$ if $P\in \cP_{\Gamma}$.

To prove Theorem \ref{thm:dihedral angles AdS}, it suffices to show the following statement.

\begin{theorem}\label{thm:homeo_angles}
The map $\Psi:\cP_N\rightarrow \cA_N$ is a homeomorphism.
\end{theorem}

To show this, we prove that $\Psi$ is a well-defined (i.e. $\Psi(P)\in\cA_{\Gamma}$ if $P\in\cP_{\Gamma}$) proper local homeomorphism, and then argue by using some topological facts concerning $\cP_N$ and $\cA_N$. More precisely, we will prove the following propositions and lemmas:

\begin{proposition}\label{prop:topology of A}
For $N\geq 4$, $\cA_N$ is a 
topological manifold of real dimension $3N-6$. In particular, $\cA_N$ is connected for $N\geq 5$. 
\end{proposition}

\begin{proposition}\label{prop:dimension of P}
$\cP_N$ is a smooth manifoldof real dimension $3N-6$ for $N\geq 4$.
\end{proposition}

\begin{proposition}\label{prop: necessary_angle}
The map $\Psi$ taking each $P\in\cP_{\Gamma}$ to its dihedral angle-assignation $\theta$ has image in $\cA_{\Gamma}$, that is, $\theta$ is $\gamma$-admissible (see Definition \ref{def:admissible angles}).
\end{proposition}

\begin{lemma}\label{lem:properness of psi}
The map  $\Psi:\cP_N\rightarrow \cA_N$ is proper. 
In other words, if $(P_n)_{n\in\N}$ is a sequence in $\cP_N$ with image $\theta_n:=\Psi(P_n)$ converging to an element $\theta_{\infty}\in\cA_N$, then $(P_n)_{n\in\N}$ converges to an element of $\cP_N$.
\end{lemma}

\begin{lemma}\label{lem:local immersion_angles}
Let $\Gamma\in\Graph(\Sigma_{0,N},\gamma)$ be a triangulation of $\Sigma_{0,N}$ and let $E$ be the edge set of $\Gamma$. Then for each $P\in\cP_N$ whose 1-skeleton is a subgraph of $\Gamma$, $\Psi:\cP_N\rightarrow\R^{E}$ is a local immersion near $P$.
\end{lemma}

Proposition \ref{prop:topology of A} and Proposition \ref{prop:dimension of P} are proved in Section \ref{sec:topology}. Proposition \ref{prop: necessary_angle} is proved in Section \ref{section:necessity}. Lemma \ref{lem:properness of psi} and Lemma \ref{lem:local immersion_angles} are shown in Section \ref{sec:properness} and Section \ref{sec:rigidity}, respectively. Combining these results we prove Theorem \ref{thm:homeo_angles} in Section \ref{sec:proof_angles}.

A (convex) \textit{hyperideal polygon} in $\bH^2$ is the intersection of $\mathbb{H}^2$ with a convex polygon in $\mathbb{RP}^2$ whose vertices are all outside of $\mathbb{H}^2$ and whose edges all intersect $\mathbb{H}^2$. Let $\widetilde{\polyg}_N$ ($N\geq 3$) be the space of all marked hyperideal polygons in $\mathbb{H}^2$ with $N$ vertices, and let $\polyg_N$ be the quotient space of $\widetilde{\polyg}_N$ modulo by the group of orientation preserving isometries of $\mathbb{H}^2$. Note that the hyperbolic plane $\mathbb{H}^2$ can be isometrically embedded in $\AdS$ as a space-like hyperplane. Therefore any marked hyperideal polygon $P\in\widetilde{\polyg}_N$ can be viewed as a marked degenerate hyperideal polyhedron lying on a space-like plane in $\AdS$, which is two-sided (future-directed and past-directed) and the union of whose edges form an oriented $N$-cycle graph, also called the equator of $P$. Note that each hyperideal AdS polyhedron with three vertices is degenerate, we define $\cP_3:=\emptyset$.
Let $\cP_N\cup\polyg_N$ ($N\geq 3$) denote the space of all marked, non-degenerate and degenerate hyperideal polyhedra in $\AdS$ with $N$ vertices, up to isometries in $\Isom_0(\AdS)$.

Recall that a \emph{cusp} is a surface isometric to the quotient of the region $\{z=x+iy:y>a\}$ of the upper-half space model of hyperbolic plane, for some $a > 0$, by the isometry group generated by $z\rightarrow z+1$. A \emph{funnel} is a surface isometric to the quotient of a hyperbolic half-plane by an isometry of hyperbolic type whose axis is the boundary of that half-plane. Let $\Sigma_{0,N}$ be the 2-sphere with $N$ marked points $p_1$, \ldots, $p_N$ removed (where $N\geq 3$).  We say that a hyperbolic metric on $\Sigma_{0,N}$ is \emph{complete} if for each $1\leq i\leq N$, there is a (small enough) region of $\Sigma_{0,N}$ near $p_i$ (with the induced metric) isometric to either a cusp or a funnel. 
We keep in mind that a complete hyperbolic metric on $\Sigma_{0,N}$ might have infinite area.

Note that any space-like plane in $\mathbb{A}\rm{d}\mathbb{S}^3$ is isometric to the hyperbolic plane $\mathbb{H}^2$ and each face of a hyperideal polyhedron in $\AdS$ is isometric to a hyperideal polygon in the hyperbolic plane. 
Therefore, the path metric induced on (the intersection with $\AdS$ of) the boundary surface of $P\in\cP_N\cup\polyg_N$ is a complete hyperbolic metric on $\Sigma_{0,N}$ with a cusp (resp. funnel) around $p_i$ if the corresponding vertex $v_i$ of $P$ is ideal (resp. strictly hyperideal). In particular, if $P$ is a degenerate hyperideal polyhedron (i.e. a hyperideal polygon), then the induced metric on the boundary surface of $P$ is obtained by doubling $P$, here the boundary surface is identified with the two copies of $P$ glued along the equator. This determines a point in the space $\cT_{0,N}$ of complete hyperbolic metrics on $\Sigma_{0,N}$, possibly with infinite area, considered up to isotopy fixing each marked point.

Now we consider the map $\Phi:\cP_N\cup\polyg_N\rightarrow \cT_{0,N}$, which assigns to each hyperideal polyhderon $P$ in $\cP_N\cup\polyg_N$ the induced metric on the boundary surface of $P$. It follows from the aforementioned fact that $\Phi$ takes image in $\mathcal{T}_{0,N}$.  To verify Theorem \ref{thm:induced metrics},  it suffices to show the following theorem:

\begin{theorem}\label{thm:homeo_metrics}
The map $\Phi: \cP_N\cup{\polyg_N}\rightarrow \mathcal{T}_{0,N}$ is a homeomorphism.
\end{theorem}

Using a similar idea as for Theorem \ref{thm:homeo_angles}, we need to show the following topological fact concerning $\cP_N\cup\polyg_N$, the local parameterization and properness statements with respect to the induced metrics.

$\cT_{0,N}$ is a contractible manifold with corners of real dimension $3N-6$ (see e.g. \cite{mbh,fock-goncharov-handbook,bonahon-liu}), which can also be obtained from the enhanced Teichm\"uller space of hyperbolic surfaces with $N$ boundary components of sign $0$ (i.e. cusps), sign $+$ or $-$ by identifying the signs ($+$ and $-$) of the boundary components. 
Note that $\cP_N\cup\polyg_N$ is the disjoint union of $\cP^{(\epsilon_1, \ldots, \epsilon_N)}$ over all $(\epsilon_1, \ldots, \epsilon_N)\in \{0, 1\}^N$, where $\mathcal{P}^{(\epsilon_1, \ldots, \epsilon_N)}$ is the subspace of $\cP_N\cup\polyg_N$ in which the polyhedra have ``vertex signature'' $(\epsilon_1, \ldots, \epsilon_N)$, meaning that the vertex $v_i$ is ideal if $\epsilon_i=0$ and strictly hyperideal if $\epsilon_i=1$ (see subsection \ref{subsec:metric} for more details). We will prove the following statements.

\begin{proposition}\label{prop:dimension of bP}
For each $(\epsilon_1, \ldots, \epsilon_N)\in \{0, 1\}^N$, the space $\cP^{(\epsilon_1, \ldots, \epsilon_N)}$ is a smooth manifold of real 
dimension $2N-6+\epsilon_1+\ldots+\epsilon_N$.
\end{proposition}

\begin{lemma}\label{lem:proper_metrics}
The map $\Phi: \cP_N\cup\polyg_N\rightarrow \mathcal{T}_{0,N}$ is proper.
\end{lemma}

\begin{lemma}\label{lem:local immersion_metrics}
For each $(\epsilon_1, \ldots, \epsilon_N)\in \{0, 1\}^N$, the restriction to $\cP^{(\epsilon_1, \ldots, \epsilon_N)}$ of the map $\Phi$ is a local immersion.
\end{lemma}

Proposition \ref{prop:dimension of bP} is proved in Section \ref{sec:topology}. 
Lemma \ref{lem:proper_metrics} and Lemma \ref{lem:local immersion_metrics} are shown in Section \ref{sec:properness} and Section \ref{sec:rigidity}, respectively. Combining these results we prove Theorem \ref{thm:homeo_metrics} in Section \ref{sec:proof_metric}.


    \section{Background materials}\label{sect:AdS}

This section contains definitions and notations needed elsewhere in the paper, as well as some background results.

\subsection{The 3-dimensional anti-de Sitter space $\AdS$ and its dual space ${\AdS}^*$.}
\label{subsec:AdS and dual AdS}

Let $\R^{2,2}$ denote the real 4-dimensional vector space $\R^4$ equipped with a symmetric bilinear form $\langle \cdot, \cdot\rangle_{2,2}$ of signature $(2,2)$, where $\langle x, y\rangle_{2,2}:=x_1y_1+x_2y_2-x_3y_3-x_4y_4$. We denote
$$AdS^3:=\{x\in\R^{2,2} ~|~ \langle x,x\rangle_{2,2}=-1\}~, $$ and
denote
$$AdS^{3*}:=\{x\in\R^{2,2}~|~ \langle x,x\rangle_{2,2}=1\}~.$$
The restriction of the bilinear form $\langle \cdot, \cdot\rangle_{2,2}$ to the tangent space at each point of of $AdS^3$ (resp. $AdS^{3*}$) induces a pseudo-Riemannian metric of signature $(2,1)$ (resp. $(1,2)$). In particular, $AdS^3$ is a 3-dimensional Lorentzian symmetric space of constant curvature $-1$ diffeomorphic to $\bH^2\times S^1$, called the \textit{quadric model} of the 3-dimensional \textit{anti-de Sitter} (AdS) space. We denote by $\AdS$ (resp. ${\AdS}^*$) the projection of $AdS^3$ (resp. $AdS^{3*}$) to $\RP^3$ (with the metric induced from the projection), which is the \textit{Klein model} of the 3-dimensional (resp. \textit{dual}) AdS space (we always use this model throughout the paper). Similarly we can define $\bAdS^2$ (also ${\bAdS}^{2*}$) by considering the bilinear form $\langle\cdot,\cdot\rangle_{1,2}$ in $\R^{1,2}$. We refer to \cite[Section 2]{bonsante-seppi:anti} for more details about the models of $n$-dimensional AdS space ($n\geq 2$).

\subsection{The boundary of $\AdS$}\label{subsec:boundary}
The boundary of ${\AdS}$ in $\RP^3$ is exactly the projective quadric:
$$ \partial{\AdS} = \{[x]\in \RP^{3}~|~ \langle x,x\rangle _{2,2}=0\} ~. $$
It is homeomorphic to $S^1\times S^1$ and admits two foliations whose leaves are projective lines (called the left and right leaves respectively) and such that each left leaf and each right leaf intersect at exactly one point.

\subsection{The geodesics and totally geodesic planes in $\AdS$}

The geodesics in $\AdS$ are obtained by intersecting $\AdS$ with the projective lines in $\RP^3$: the spacelike geodesics correspond to the projective lines intersecting the boundary $\partial\AdS$ in two points, while lightlike geodesics are tangent to $\partial\AdS$, and timelike geodesics are disjoint from $\partial\AdS$ (see Figure \ref{fig:geodesics and planes}).

The totally geodesic planes in $\AdS$ are obtained by intersecting $\AdS$ with the projective planes in $\RP^3$ (as shown in Figure \ref{fig:geodesics and planes}). Among them, a spacelike plane is a compression disk (whose closure in $\RP^3$ intersects $\partial \AdS$ along an ellipse) with induced metric hyperbolic, while a lightlike plane is the intersection with $\AdS$ of a plane tangent to $\partial\AdS$ (whose closure in $\RP^3$ intersects $\partial\AdS$ along two lightlike geodesics) with induced metric degenerate, and a timelike plane is topologically a M\"obius band and it is isometric to $\bAdS^2$ (see \cite[Section 4]{mbh} and \cite[Section 2.4]{bonsante-seppi:anti}).

\begin{figure}
\begin{subfigure}[b]{0.46\textwidth}
\centering
\begin{tikzpicture}[scale=0.8]
\draw  (-0.5,2.55) ellipse (2.5 and 0.5);
\draw[densely dashed] (-2.95,-2.5) .. controls (-2.95,-1.85) and (2,-1.9) .. (2,-2.55);
\draw (-2.95,-2.5) .. controls (-2.95,-3.15) and (2,-3.2) .. (2,-2.55);
\draw (-3.5,3) .. controls (-0.5,0.05) and (-0.5,0.05) .. (-3.5,-3);
\draw (2.5,3) .. controls (-0.5,0.05) and (-0.5,0.05) .. (2.5,-3);
\draw[thick][densely dashed](-1.25,0.25) .. controls (0.25,0.1) and (0.25,0.1) .. (0.25,0.1);
\draw[thick] (-0.5,3.2) .. controls (-0.5,2.25) and (-0.5,2.25) .. (-0.5,2.25);
\draw[thick](-0.95,2.2) .. controls (-1.25,0.1) and (-1.25,0.1) .. (-1.25,0.1);
\draw[thick][densely dashed](-1.25,0.1) .. controls (-1.55,-2.15) and (-1.55,-2.15) .. (-1.55,-2.15);
\draw[densely dashed] [thick](-0.5,2.25) .. controls (-0.5,-3.15) and (-0.5,-3.15) .. (-0.5,-3.15);
\draw[thick] (-0.5,-3.15) .. controls (-0.5,-3.15) and (-0.5,-3.15) .. (-0.5,-3.15);
\draw[densely dashed] [thick](-1.5,3.05) .. controls (-1.25,0.1) and (-1.25,0.1) .. (-1.25,0.1);
\draw [thick](-1.05,-3) .. controls (-1.25,0.1) and (-1.25,0.1) .. (-1.25,0.1);
\node at (-2.3,-0.65) {lightlike};
\node at (-0.35,-3.35) {timelike};
\node at (1.15,0.1) {spacelike};
\end{tikzpicture}
\end{subfigure}
\begin{subfigure}[b]{0.46\textwidth}
\centering
\begin{tikzpicture}[scale=0.8]
\draw  (-0.5,2.55) ellipse (2.5 and 0.5);
\draw[densely dashed] (-2.95,-2.5) .. controls (-2.95,-1.85) and (2,-1.9) .. (2,-2.55);
\draw (-2.95,-2.5) .. controls (-2.95,-3.15) and (2,-3.2) .. (2,-2.55);
\draw (-3.5,3) .. controls (-0.5,0.05) and (-0.5,0.05) .. (-3.5,-3);
\draw (2.5,3) .. controls (-0.5,0.05) and (-0.5,0.05) .. (2.5,-3);
\draw[thick][densely dashed](-1.25,0.25) .. controls (0.25,0.1) and (0.25,0.1) .. (0.25,0.1);
\draw[thick] (-0.5,3.2) .. controls (-0.5,2.25) and (-0.5,2.25) .. (-0.5,2.25);
\draw[thick](-0.95,2.2) .. controls (-1.25,0.1) and (-1.25,0.1) .. (-1.25,0.1);
\draw[thick][densely dashed](-1.25,0.1) .. controls (-1.55,-2.15) and (-1.55,-2.15) .. (-1.55,-2.15);
\draw[densely dashed] [thick](-0.5,2.25) .. controls (-0.5,-3.15) and (-0.5,-3.15) .. (-0.5,-3.15);
\draw[thick] (-0.5,-3.15) .. controls (-0.5,-3.15) and (-0.5,-3.15) .. (-0.5,-3.15);
\draw [fill=gray, fill opacity=0.3, densely dashed] (-1.25,0.25) .. controls (-1.25,0.6) and (0.3,0.45) .. (0.25,0.1);
\draw[fill=gray, fill opacity=0.3, densely dashed]  (-1.25,0.25) .. controls (-1.3,-0.1) and (0.25,-0.25) .. (0.25,0.1);
\draw [fill=gray, fill opacity=0.3, densely dashed] (-0.2,2.1) .. controls (-0.85,3.1) and (-0.85,3.1) .. (-0.85,3.1) .. controls (-0.5,0.05) and (-0.5,0.05) .. (-0.85,-2.1) .. controls (-0.1,-3.1) and (-0.1,-3.1) .. (-0.1,-3.1) .. controls (-0.5,0.05) and (-0.5,0.05) .. (-0.2,2.1);
\draw[densely dashed] [thick](-1.5,3.05) .. controls (-1.25,0.1) and (-1.25,0.1) .. (-1.25,0.1);
\draw [thick](-1.05,-3) .. controls (-1.25,0.1) and (-1.25,0.1) .. (-1.25,0.1);
\draw [fill=gray, fill opacity=0.3,densely dashed] (-1.5,3.05) .. controls (-0.95,2.2) and (-0.95,2.2) .. (-0.95,2.2) .. controls (-1.25,0.1) and (-1.25,0.1) .. (-1.25,0.1) .. controls (-1.5,3.05) and (-1.5,3.05) .. (-1.5,3.05);
\draw [fill=gray, fill opacity=0.3,densely dashed](-1.05,-3) .. controls (-1.25,0.1) and (-1.25,0.1) .. (-1.25,0.1) .. controls (-1.55,-2.15) and (-1.55,-2.15) .. (-1.55,-2.15) .. controls (-1.05,-3) and (-1.05,-3) .. (-1.05,-3);
\node at (-2.3,-0.65) {lightlike};
\node at (0.15,-3.35) {timelike};
\node at (1.15,0.1) {spacelike};
\draw[thick] (-0.1764,2.1176) .. controls (-0.5292,0.0884) and (-0.4704,0.088) .. (-0.1176,-3.1176);
\draw [thick](-1.2352,0.2648) .. controls (-1.294,-0.1176) and (0.2648,-0.2648) .. (0.2648,0.1176);
\draw[-latex][densely dashed](-0.1768,2.0888)--(-0.7176,2.9064);
\draw[-latex][densely dashed](-0.8236,-2.086)--(-0.306,-2.856);
\end{tikzpicture}
\end{subfigure}
\caption{\small{Examples of spacelike, timelike and lightlike geodesics (in the left picture) and totally geodesic planes (in the right picture) in $\AdS$ (in an affine chart).}}
\label{fig:geodesics and planes}
\end{figure}

\subsection{The isometry group of $\AdS$}
The isometry group of $AdS^3$ is the indefinite orthogonal group $\rO(2,2)$ of reversible linear transformations on $\R^{2,2}$ preserving the bilinear form $\langle \cdot, \cdot \rangle_{2,2}$. In particular, the group $\Isom_0(\AdS)$ of orientation and time-orientation preserving isometries of $\AdS$ is the identity component $\PO_0(2,2)$ of $\PO(2,2)$. It is interesting to notice that $\AdS$ has a Lie group model, in which $\AdS$ is identified with $\PSL_2(\R)$, and $\Isom_0(\AdS)$ is identified with $\PSL_2(\R)\times\PSL_2(\R)$, see e.g. \cite{mess, bonsante-seppi:anti}.

\subsection{Hyperideal polyhedra in $\AdS$}

We give some key properties of the edges and faces of hyperideal AdS polyhedra, which will be used in subsequent sections.
\begin{lemma}\label{lm:face_spacelike}
 Let $P$ be a hyperideal AdS polyhedron. Then all the edges and faces of $P$ are spacelike.
\end{lemma}
\begin{proof}
By definition, all the edges of $P$ intersect $\partial \AdS$ in two points and are thus spacelike. It remains to show that each face of $P$ is spacelike. We argue by contradiction. Assume that $P$ has a non-spacelike (i.e. timelike or lightlike) face, say $f$. Then $f$ is contained in the plane (in an affine chart), say $F$, whose intersection with $\AdS$ is non-spacelike. Moreover, the set $F\setminus(\AdS\cup\partial\AdS)$ has two connected components (say $C_1$, $C_2$) whenever $f$ is timelike or lightlike (as shown in Figure \ref{fig:non-spacelike}). Combining the fact that all vertices of $f$ lie outside of $\AdS$ and $f$ (as a face) has at least three vertices, at least two adjacent vertices, say $v_1$, $v_2$, of $f$ lie on the closure of the same component $C_i$. By the assumption that $F$ is timelike or lightlike, the closures of $C_1$ and $C_2$ are both convex. 
As a consequence, the edge of $f$ connecting $v_1$ and $v_2$ must lie in the closure of $C_i$ and thus lie outside of $\AdS$ (see Figure \ref{fig:non-spacelike} for instance), which leads to contradiction. The lemma follows.
\end{proof}

  \begin{figure}
   \begin{subfigure}[b]{0.46\textwidth}
    \centering
\begin{tikzpicture}
\draw [white][fill=gray,opacity=0.3](2.5,6.5) .. controls (4.5,4.5) and (4.5,4.5) .. (2.5,2.5) .. controls (6.5,2.5) and (6.5,2.5) .. (6.5,2.5) .. controls (4.5,4.5) and (4.5,4.5) .. (6.5,6.5) .. controls (2.5,6.5) and (2.5,6.5) .. (2.5,6.5);
\draw [thick](3.85,5)   .. controls (3.85,4) and (3.85,4) .. (3.85,4) .. controls (6,4.5) and (6,4.5) .. (6,4.5) .. controls (3.85,5) and (3.85,5) .. (3.85,5);
\node at (3.6,5.05) {$v_1$};
\node at (3.55,4.05) {$v_2$};
\draw (2.5,6.5) .. controls (4.5,4.5) and (4.5,4.5) .. (2.5,2.5);
\draw (6.5,6.5) .. controls (4.5,4.5) and (4.5,4.5) .. (6.5,2.5);
\node at (2.65,3.55) {$C_1$};
\node at (6.45,3.5) {$C_2$};
\node at (6.4,4.5) {$f$};
\end{tikzpicture}
\end{subfigure}
    \begin{subfigure}[b]{0.46\textwidth}
     \centering
\begin{tikzpicture}
\draw[white][fill=gray,opacity=0.3] (2.5,6.5) .. controls (4.5,4.5) and (4.5,4.5) .. (4.5,4.5) .. controls (6.5,6.5) and (6.5,6.5) .. (6.5,6.5) .. controls (2.5,6.5) and (2.5,6.5) .. (2.5,6.5);
\draw [white][fill=gray,opacity=0.3](4.5,4.5) .. controls (2.5,2.5) and (2.5,2.5) .. (2.5,2.5) .. controls (6.5,2.5) and (6.5,2.5) .. (6.5,2.5) .. controls (4.5,4.5) and (4.5,4.5) .. (4.5,4.5);
\draw (2.5,2.5) .. controls (6.5,6.5) and (6.5,6.5) .. (6.5,6.5);
\draw (2.5,6.5) .. controls (6.5,2.5) and (6.5,2.5) .. (6.5,2.5);
\draw [thick](4,5) .. controls (2.5,4.5) and (2.5,4.5) .. (2.5,4.5) .. controls (4,4) and (4,4) .. (4,4) .. controls (6,4.5) and (6,4.5) .. (6,4.5) .. controls (4,5) and (4,5) .. (4,5);
\node at (4.05,5.25) {$v_1$};
\node at (2.5,4.85) {$v_2$};
\node at (2.7,3.5) {$C_1$};
\node at (6.3,3.5) {$C_2$};
\node at (6.5,4.5) {$f$};
\end{tikzpicture}
    \end{subfigure}
    \caption{\small{Examples of non-spacelike face $f$. In the left (resp. right) picture, $f$ is a timelike (resp. lightlike) face, lying in the plane $F$ in an affine chart. The shaded region represents $F\cap\AdS$, while the white regions $C_1$ and $C_2$ represent the two connected components of $F\setminus(\AdS\cup\partial\AdS)$}.}
    \label{fig:non-spacelike}
    \end{figure}

Combining Lemma \ref{lm:face_spacelike} and the convexity of hyperideal AdS polyhedra, we have the following result.

\begin{corollary}\label{cor:geodesic}
Let $P$ be a hyperideal AdS polyhedron and let $v$, $v'$ be two distinct vertices of $P$. Then the geodesic segment connecting $v$ and $v'$ intersects $\AdS$.
\end{corollary}

\subsection{The 2-dimensional plane $\HSS$ and 3-dimensional space $\HS$}

Faces of hyperideal polyhedra in $\AdS$ typically have a part in $\AdS$, which is space-like and therefore locally modelled on the hyperbolic plane, but also a part outside $\AdS$. This ``exterior'' part is time-like in ${\AdS}^*$.
So, when describing the induced geometric structure on the boundary of a hyperideal polyhedron in $\AdS$, it will be useful below to consider the induced geometric structure as an HS-structure, in the sense that it is locally modelled on the 2-dimensional plane $\HSS$.

The plane $\HSS$ can be described as $\RP^2$, with the geometric structure induced by the Minkowski bilinear form $\langle\cdot, \cdot\rangle_{2,1}$. 
One open disk in $\HSS$, corresponding to the time-like lines in $\R^{2,1}$, is isometric to the hyperbolic plane $\HH^2$, while the complement of the closure of this disk, corresponding to  the space-like lines in $\R^{2,1}$, is isometric to the de Sitter plane. In addition, there is a well-defined notion of ``distance'' (usually defined as a complex number) between a point of the hyperbolic part and a point of the de Sitter part of $\HSS$ (see also \cite[Section 2]{shu}).

Let $\HS$ denote the space $\RP^3$ equipped with the geometric structure induced by the bilinear form $\langle \cdot, \cdot\rangle_{2,2}$. 

By construction its geometry is invariant under $\PO_0(2,2)$. Here too, $\HS$ has a part  corresponding to the time-like lines in $\R^{2,2}$, which is isometric to $\AdS$, while the complement of the closure of this part is isometric to ${\AdS}^*$.

Let $\sigma: {\AdS\cup{\AdS}^*}\rightarrow \R^{2,2}$ be a local section of the canonical projection $\pi:\R^{2,2}\rightarrow \RP^3$ with $\sigma({\AdS})\subset AdS^3$ and $\sigma({\AdS}^*)\subset {\AdS}^*$. For each point $[x]\in \AdS\cup{\AdS}^*$, we define the \textit{HS scalar product} on $T_{[x]}{\HS}$ as
$$\langle v,w\rangle_{HS}:=\langle \sigma_{*}(v), \sigma_{*}(w)\rangle_{2,2}~,$$
for all $v, w\in T_{[x]}\HS$. The restriction of $\langle \cdot,\cdot\rangle_{HS}$ to $T\AdS$ 
exactly induces the AdS 
metric. Note that for any $[x]\in\partial\AdS$, there is no way to define a natural metric on $T_{[x]}\HS$ (see e.g. \cite[Section 2.2]{bonsante-seppi:anti}).

\subsection{Duality in $\HS$}\label{sect:dual polyhedra}
The bilinear form $\langle \cdot,\cdot \rangle_{2,2}$ induces a duality in $\HS$. Let $W$ denote a totally geodesic subspace of $\HS$. We define the \textit{dual} of $W$ as
\begin{center}
$W^{\perp}:=\{\,[x]\in\HS: \langle x, y\rangle_{2,2}=0$ for all $[y]\in W\,\}~.$
\end{center}

Let $P$ be a hyperideal AdS polyhedron. The \textit{dual polyhedron} $P^*$ of $P$ is defined as the polyhedron $P^*$ in $\HS$ such that
\begin{itemize}
\item a $k$-face $f^*$ of $P^*$ is contained in the dual of a $(2-k)$-face $f$ of $P$ for $k=0,1,2$;
\item a $k$-face $f^*$ of $P^*$ is contained in the boundary of a $(k+1)$-face $(f')^*$ if and only if the boundary of the corresponding $(2-k)$-face $f$ of $P$ contains the corresponding $(2-k-1)$-face $f'$ for $k=0,1$.
\end{itemize}

 We call the face $f^*$ of $P^*$ the \emph{dual} of the face $f$ of $P$. It is clear that the dual polyhedron $P^*$ is uniquely determined by $P$, and furthermore, $P^*$ is a convex if and only if $P$ is convex.

    \section{Necessity}\label{section:necessity}

In this section, we show that the conditions in Theorem \ref{thm:dihedral angles AdS}
are necessary by proving Proposition \ref{prop: necessary_angle}.

\subsection{Angle conditions}
Let $P$ be a hyperideal AdS polyhedron with 1-skeleton $\Gamma\in \Graph(\Sigma_{0,N},\gamma)$. Let $\theta: E(\Gamma)\rightarrow \mathbb{R}$ be a function assigning to each edge $e\in E(\Gamma)$ the (exterior) dihedral angle $\theta(e)$ at $e$ of $P$. We prove in the following that $\theta$ is a $\gamma$-admissible function. 

\begin{proof}[\textbf{Proof of Proposition \ref{prop: necessary_angle}}]
We show that $\theta$ satisfies Conditions (i)-(iv) in Definition \ref{def:admissible angles}.

\begin{condition1}
\end{condition1}
Condition (i) follows from the sign convention for dihedral angles.

\begin{condition2}
\end{condition2}
For Condition (ii), let $v$ be the vertex of $\Gamma$ dual to that face of $\Gamma^*$.

If $v$ is ideal, we use a similar argument as in \cite[Section 6]{DMS}: the intersection of $P$ with a small ``horo-torus'' centered at $v$ is a convex polygon (with space-like edges) in the Minkowski plane. The exterior angle at each vertex of this polygon is equal to the exterior dihedral angle at the corresponding edge (containing that vertex) of $P$.
Let's recall the Gauss-Bonnet formula for Lorentzian polygons $p$ with all edges spacelike (see \cite[Section 6]{DJJ}, note that the dihedral angle defined here is the opposite of the imaginary part of that in \cite{DJJ}):
\begin{equation}\label{formula:Gauss-Bonnet}
\iint_{p} K ds+ \int_{\partial p} k_g dl + \sum_{i}\theta_i=0~,
\end{equation}
where $K$, $k_g$, $ds$, $dl$ and $\theta_i$ denote the Gaussian curvature, the geodesic curvature,  the area form, the arc-length element and the exterior angle at the $i$-th vertex of $p$, respectively. It follows that the sum of the exterior angles at the vertices of a Minkowski polygon with space-like edges is zero, therefore, $\theta(e_1)+\cdots+\theta(e_k)= 0$.

If $v$ is strictly hyperideal, consider the dual plane $v^{\perp}$ (whose intersection with $\AdS$ is isometric to $\bAdS^2$). A key remark is that if $L$ is a line going through $v$ and through $\AdS$, then $v^\perp$ separates the two intersections of $L$ with $\partial\AdS$ --- this can be seen for instance by applying an isometry to send $v$ to infinity (as shown in Figure \ref{fig:dual planes}).
As a consequence, if $v$ is a strictly hyperideal vertex of a hyperideal polyhedron $P$, then $v^{\perp}$ separates $v$ from all other vertices of $P$ --- otherwise, there would be a segment connecting $v$ to another vertex of $P$ without entering $\AdS$, which contradicts Corollary \ref{cor:geodesic}.

  \begin{figure}
   \begin{subfigure}[b]{0.46\textwidth}
    \centering
\begin{tikzpicture}[scale=0.8]
\draw  (-0.5,2.55) ellipse (2.5 and 0.5);
\draw[densely dashed] (-2.95,-2.5) .. controls (-2.95,-1.85) and (2,-1.9) .. (2,-2.55);
\draw (-2.95,-2.5) .. controls (-2.95,-3.15) and (2,-3.2) .. (2,-2.55);
\draw (-3.5,3) .. controls (-0.5,0.05) and (-0.5,0.05) .. (-3.5,-3);
\draw (2.5,3) .. controls (-0.5,0.05) and (-0.5,0.05) .. (2.5,-3);
\draw (-1.25,0.0625) .. controls (-1.1875,-0.3125) and (0.1875,-0.375) .. (0.25,0.0625);
\draw [densely dashed] (-1.25,0.0625) .. controls (-1.25,0.4375) and (0.1875,0.5) .. (0.25,0.0625);
\draw (-2.4375,0) .. controls (-1.25,0.25) and (-1.25,0.25) .. (-1.25,0.25);
\draw (-2.4375,0) .. controls (1.125,-0.4375) and (1.125,-0.4375) .. (1.125,-0.4375);
\draw [densely dashed](-0.875,0.3125) .. controls (0.375,0.5625) and (0.375,0.5625) .. (0.375,0.5625);
\draw  (-0.875,0.3125) .. controls (-0.5,-0.25) and (-0.5,-0.25) .. (-0.5,-0.25);
\node at (-2.75,0) {$v$};
\draw[white,fill=gray, fill opacity=0.3](-1.125,3.0625) .. controls (-0.3125,2.0625) and (-0.3125,2.0625) .. (-0.3125,2.0625) .. controls (-0.6875,0.875) and (-0.5,-1.9375) .. (-0.25,-3) .. controls (-1.25,-2.0625) and (-1.25,-2.0625) .. (-1.25,-2.0625) .. controls (-0.75,-0.6875) and (-0.8125,1.625) .. (-1.125,3.0625);
\draw (-0.3125,2.0625) .. controls (-0.3125,2.0625) and (-0.3125,2.0625) .. (-0.3125,2.0625) .. controls (-0.6875,0.875) and (-0.5,-1.9375) .. (-0.25,-3);
\draw [densely dashed](-1.125,3.0625) .. controls (-0.8125,1.625) and (-0.75,-0.6875) .. (-1.25,-2.0625);
\draw (-0.25,-3) .. controls (-0.25,-3) and (-0.25,-3) .. (-0.25,-3);
\node at (0.4375,2.625) {$v^{\perp}\cap\mathbb{A}\rm{d}\mathbb{S}^3$};
\draw[thick]  (-2.4375,0) .. controls (1.125,0.1875) and (1.125,0.1875) .. (1.125,0.1875);
\draw(0.375,0.5625) .. controls (1.3125,0.75) and (1.3125,0.75) .. (1.3125,0.75);
\node at (1.375,0.1875) {$L$};
\draw [densely dashed](-1.25,0.25) .. controls (-0.875,0.3125) and (-0.875,0.3125) .. (-0.875,0.3125);
\end{tikzpicture}
\end{subfigure}
\begin{subfigure}[b]{0.46\textwidth}
     \centering
\begin{tikzpicture}[scale=0.8]
\draw  (-0.5,2.55) ellipse (2.5 and 0.5);
\draw[densely dashed] (-2.95,-2.5) .. controls (-2.95,-1.85) and (2,-1.9) .. (2,-2.55);
\draw (-2.95,-2.5) .. controls (-2.95,-3.15) and (2,-3.2) .. (2,-2.55);
\draw (-3.5,3) .. controls (-0.5,0.05) and (-0.5,0.05) .. (-3.5,-3);
\draw (2.5,3) .. controls (-0.5,0.05) and (-0.5,0.05) .. (2.5,-3);
\draw (-1.25,0.0625) .. controls (-1.1875,-0.3125) and (0.1875,-0.375) .. (0.25,0.0625);
\draw [dashed](-1.25,0.0625) .. controls (-1.25,0.4375) and (0.1875,0.5) .. (0.25,0.0625);
\draw (-0.6875,0.375) .. controls (-0.6875,0.375) and (-0.6875,0.375) .. (-0.6875,0.375);
\draw (-2.5,-0.25) .. controls (1.5,-0.25) and (1.5,-0.25) .. (1.5,-0.25);
\draw [densely dashed](-0.6875,0.375) .. controls (1.5,0.375) and (1.5,0.375) .. (1.5,0.375);
\draw (-0.6875,0.375) .. controls (-0.3125,-0.25) and (-0.3125,-0.25) .. (-0.3125,-0.25);
\draw[white,fill=gray, fill opacity=0.3](-0.8125,3.0625) .. controls (-0.0625,2.0625) and (-0.0625,2.0625) .. (-0.0625,2.0625) .. controls (-0.375,0.8125) and (-0.4375,-1.4375) .. (-0.0625,-3) .. controls (-1,-2) and (-1,-2) .. (-1,-2) .. controls (-0.625,-0.5) and (-0.5625,1.875) .. (-0.8125,3.0625);
\draw (-0.0625,2.0625) .. controls (-0.0625,2.0625) and (-0.0625,2.0625) .. (-0.0625,2.0625) .. controls (-0.375,0.8125) and (-0.4375,-1.4375) .. (-0.0625,-3);
\draw [densely dashed](-0.8125,3.0625) .. controls (-0.5625,1.875) and (-0.625,-0.5) .. (-1,-2);
\draw [densely dashed](-0.0625,-3) .. controls (-0.0625,-3) and (-0.0625,-3) .. (-0.0625,-3);
\node at (0.75,2.5625) {$v^{\perp}\cap\mathbb{A}\rm{d}\mathbb{S}^3$};
\draw [densely dashed](-0.6875,0.375) .. controls (-1.3125,0.375) and (-1.3125,0.375) .. (-1.3125,0.375);
\draw (-1.3125,0.375) .. controls (-2.5,0.375) and (-2.5,0.375) .. (-2.5,0.375);
\draw (0.3125,0.375) .. controls (1.5,0.375) and (1.5,0.375) .. (1.5,0.375);
\draw [thick](-2.5,0.0625) .. controls (1.5,0.0625) and (1.5,0.0625) .. (1.5,0.0625);
\node at (1.75,0.0625) {$L$};
\end{tikzpicture}
    \end{subfigure}
    \caption{\small{Examples of the dual plane $v^{\perp}$ of a strictly hyperideal vertex $v$ of $P$. The vertex $v$ contained in the affine chart $\{x_4=1\}$ (in the left picture) is sent to a point at infinity in $\{x_4=0\}$ (in the right picture) by an isometry. The bold line $L$ is a line going through $v$ and $\AdS$.}}
    \label{fig:dual planes}
    \end{figure}

The dual plane $v^{\perp}$ is orthogonal to all the faces and edges of $P$ adjacent to $v$ and the intersection of $v^{\perp}$ with $P$ is a convex compact polygon $p$ in $v^{\perp}\cap\AdS\cong\bAdS^2$ with space-like edges. Thus the exterior angle at each vertex of this polygon is equal to the exterior dihedral angle at the corresponding edge (containing that vertex) of $P$.
Using the formula \eqref{formula:Gauss-Bonnet} again, we have that $\theta(e_1)+\cdots+\theta(e_k)= \area(p)>0$.

\begin{condition3}
\end{condition3}
For Condition (iii), let $e^*_1, e^*_2, \ldots, e^*_k, e^*_{k+1}=e^*_1$ be a simple circuit which does not bound a face of $\Gamma^*$,  and such that exactly two of the edges are dual to edges of $\gamma$. Denote the common endpoint of $e^*_i$ and $e^*_{i+1}$ by $f^*_i$ for $i=1, \ldots, k$, which is the vertex of $\Gamma^*$ dual to the face $f_i$ of $P$ (equivalently, the boundary of $f_i$ contains the edges $e_i$ and $e_{i+1}$). To show that $\theta(e_1)+\cdots+\theta(e_k)>0$, it suffices to consider the extension polyhedron $Q$ (see e.g. \cite[Section 6]{DMS}), which is a polyhedron obtained by extending the faces $f_1, f_2, \ldots, f_k$ and forgetting about the other faces of $P$. Note that the analysis (see \cite[Section 6]{DMS}) of the sum of the dihedral angles at the edges (containing $e_1, e_2, \ldots, e_k$) of the corresponding extension polyhedron in the case of ideal polyhedra in  $\mathbb{A}\rm{d}\mathbb{S}^3$ completely applies to the extension polyhedron $Q$ here, assuming that the circuit considered has exactly two edges dual to edges of $\gamma$. This implies that $\theta$ satisfies Condition (iii).

\begin{condition4}
\end{condition4}
The condition (iv) is an analogue of Condition (4) for the dihedral angles of hyperideal polyhedra in hyperbolic 3-space $\mathbb{H}^3$ \cite[Theorem 1]{bao-bonahon}. Let $v$ be the vertex of $\Gamma$ dual to that face, say $f$, of $\Gamma^*$ on which the simple path $e^*_1, \ldots, e^*_k$ starts and ends.

We claim that the unique edge in the path $e^*_1, \ldots, e^*_k$ which is dual to one edge of $\gamma$, say $e^*_j$, is not contained in the boundary of the face $f$. To prove this, notice that since all vertices of $\Gamma$ are also vertices of $\gamma$, the complement (say $\Gamma^*_1$) in $\Gamma^*$ of the union of the edges dual to the edges of $\gamma$ is the disjoint union of two trees. So if $e^*_j$ were contained in $\partial f$, it would follow that $e^*_1,\ldots, e^*_{j-1}$ would also be contained in $\partial f$, and so would $e^*_{j+1},\ldots, e^*_k$ (otherwise, $\Gamma^*_1$ contains at least one cycle, see e.g. Figure \ref{fig:necessary}).
This would contradict the hypothesis that the simple path $e^*_1, \ldots, e^*_k$ is not contained in the boundary of $f$.

\begin{figure}
\usetikzlibrary{decorations.pathreplacing}
\begin{tikzpicture}
\draw (-1,2) ellipse (2 and 2);
\draw[red](-3,2) .. controls (-3,1.3) and (1,1.3) .. (1,2);
\draw[red][densely dashed](-3,2) .. controls (-3,2.7) and (1,2.7) .. (1,2);
\draw(-1,1.5) .. controls (-1.38,2.2) and (-1.38,2.2) .. (-2,3);
\draw(-1,1.5) .. controls (-0.7,2.2) and (-0.7,2.2) .. (0,3);
\draw(-1,1.5) .. controls (-0.6,0.9) and (-0.6,0.9) .. (-0.3,0.4);
\draw(-1,1.5) .. controls (-1.3,1) and (-1.3,1) .. (-1.6,0.4);
\draw(-1,1.7) ellipse (0.7 and 0.7);
\draw [fill=black](-1.66,1.9) ellipse (0.04 and 0.04);
\draw [fill=black](-0.3,1.8) ellipse (0.04 and 0.04);
\draw [fill=black](-1,1) ellipse (0.04 and 0.04);
\draw [fill=black](-1,2.4) ellipse (0.04 and 0.04);
\draw [fill=black](-1.58,1.3) ellipse (0.04 and 0.04);
\draw [fill=black](-0.45,1.26) ellipse (0.04 and 0.04);
\draw[thick,blue](-1.68,1.9) .. controls (-2.2,2.3) and (-2.7,2.7) .. (-2.6,3.2);
\draw[densely dashed][thick,blue](-2.6,3.2) .. controls (-2.5,3.68) and (0.56,3.68) .. (0.6,3.2);
\draw[thick,blue](0.6,3.2) .. controls (1,2.5) and (0,2) .. (-0.3,1.8);
\draw[thick,blue](-0.45,1.26) .. controls (0,1.26) and (0.9,1.25) .. (0.75,1);
\draw[densely dashed][thick,blue](-2.74,1) .. controls (-2.2,0.4) and (0.4,0.4) .. (0.75,1);
\draw[thick,blue](-2.74,1) .. controls (-3,1.3) and (-2,1.3) .. (-1.6,1.3);
\draw[thick,blue](-0.3,1.8) .. controls (-0.3,1.5) and (-0.3,1.5) .. (-0.45,1.26);
\node at (-2,1.1) {$e^*_1$};
\node at (0,1) {$e^*_{j-1}$};
\node at (-0.08,1.6) {$e^*_{j}$};
\node at (0.45,1.94) {$e^*_{j+1}$};
\node at (-2, 1.85) {$e^*_{k}$};
\node at (-0.76,1.6) {$v$};
\node at (1.28,2) {$\gamma$};
\end{tikzpicture}
\caption{\small{A counter-example of the simple path (shown in the bold lines) $e^*_1,e^*_2,\dots,e^*_k$ in $\Gamma^*$ with the unique edge $e^*_j$ (dual to one edge of $\gamma$) contained in the boundary of the face dual to $v$, where $\Gamma$ contains the equator $\gamma$ and is partially drawn near the vertex $v$.}}
\label{fig:necessary}
\end{figure}

Therefore, it suffices to consider the case that the simple path $e^*_1, \ldots, e^*_k$ meets the boundary of $f$ only at their two endpoints (which coincide with two vertices of $\Gamma^*$), since the dihedral angles at the edges of $P$ which are dual to those in the intersection (if non-empty) of the simple path $e^*_1, \ldots, e^*_k$ with the boundary of $f$ are positive.

If $v$ is strictly hyperideal, we denote by $X_v$ the half-space containing $v$ bounded by the dual plane $v^{\perp}$. Let $P'$ be the polyhedron obtained by doubling $P\setminus X_v$ via the reflection across  $v^{\perp}$. Note that $v^{\perp}$ separates $v$ from the other vertices of $P$ and is orthogonal to the edges and faces of $P$ that it meets, $P'$ is a hyperideal polyhedron. Let $\Gamma'$ be the 1-skeleton of $P'$ and let $\Gamma'^*$ be the dual graph of $\Gamma'$. It is clear that $\Gamma'^*$ is obtained by gluing two copies of $\Gamma^*$ along the the boundary of $f$. By the assumption, the path $e^*_1, \ldots, e^*_k$ and its images $e'^*_1, \ldots, e'^*_k$ in $\Gamma'^*$ under the reflection form a simple circuit which does not bound a face of $\Gamma'^*$, and exactly two of the edges of the circuit are dual to edges of $\gamma$. Note that $\theta(e_i)=\theta(e'_i)$ for $i=1,\ldots,k$. Combined with the above result that $\theta$ satisfies Condition (iii), we have
\begin{equation*}
\theta(e_1)+\cdots+\theta(e_k)+\theta(e'_1)+\cdots+\theta(e'_k)
=2\big(\theta(e_1)+\cdots+\theta(e_k)\big)>0~,
\end{equation*}
 which implies $\theta$ satisfies Condition (iv).

If $v$ is ideal, let $\delta$ be the simple circuit in $\Gamma^*$ which bounds the face $f$ of $\Gamma^*$. The endpoints of the simple path   $e^*_1, \ldots, e^*_k$ decompose $\delta$ into two parts, say $\delta_1$, $\delta_2$. Since the path $e^*_1, \ldots, e^*_k$ contains exactly one edge dual to an edge of $\gamma$, then the endpoints of this path coincide with two vertices on $\delta$ of $\Gamma^*$ which are dual to two faces of $\Gamma$ on the 
left and right side of $\gamma$, respectively. Therefore, $\delta_1$ (resp. $\delta_2$) contains exactly one edge dual to an edge of $\gamma$. Without loss of generality, we assume that
\begin{equation}\label{eq:assumption}
\sum_{e^*\subset\delta_1}\theta(e)\leq\sum_{(e')^*\subset\delta_2}\theta(e')~.
\end{equation}
Since $\theta$ satisfies Condition (ii), then
\begin{equation}\label{eq:ideal}
\sum_{e^*\subset\delta_1}\theta(e)+\sum_{(e')^*\subset\delta_2}\theta(e')=0~.
\end{equation}
Combining the assumption \eqref{eq:assumption} and the equality \eqref{eq:ideal}, we have
\begin{equation}\label{ineq1}
\sum_{e^*\subset\delta_1}\theta(e)\leq 0~.
\end{equation}

The union of the simple path $e^*_1, \ldots, e^*_k$ and $\delta_1$ forms a simple circuit in $\Gamma^*$, say $\xi$, which contains exactly two edges dual to edges of $\gamma$. If $\xi$ does not bound a face of $\Gamma^*$, since $\theta$ satisfies Condition (iii), then
\begin{equation}\label{ineq2}
\theta(e_1)+\cdots+\theta(e_k)+\sum_{e^*\subset\delta_1}\theta(e)>0~.
\end{equation}
By the inequalities \eqref{ineq1} and \eqref{ineq2}, we have $\theta(e_1)+\cdots+\theta(e_k)>0$.

If $\xi$ bounds a face $f'$ of $\Gamma^*$, then $f'\not=f$ (since the path $e^*_1, \ldots, e^*_k$ is not contained in the boundary of $f$). Therefore $\delta_1$ is contained in $\partial f\cap\partial f'$ and thus is an edge of $\Gamma^*$, say $\delta_1:=e_0^*$. Note that the dihedral angle along each edge of $P$ is non-zero (see also Condition (i)). Combined with the inequality \eqref{ineq1}, we have
\begin{equation}\label{ineq3}
\sum_{e^*\subset\delta_1}\theta(e)=\theta(e_0)< 0~.
\end{equation}
By Condition (ii), we have $\theta(e_0)+\theta(e_1)+\cdots+\theta(e_k)=0$. Combined with the inequality \eqref{ineq3}, we have $\theta(e_1)+\cdots+\theta(e_k)>0$.
Therefore, Condition (iv) follows. (Note that this argument is specific to the case where $v$ is ideal, and does not apply as it is to the case where $v$ is strictly hyperideal, since it uses in an essential way the fact that the sum of the values of $\theta$ on the edges of $\Gamma^*$ bounding the face dual to $v$ is equal to $0$.)

\end{proof}

    \section{Properness}\label{sec:properness}

In this section, we prove Lemma \ref{lem:properness of psi} and Lemma \ref{lem:proper_metrics}, that is, the parameterization map $\Psi$ (resp. $\Phi$) of the space $\cP_N$ (resp. $\cP_N\cup\polyg_N$) of marked non-degenerate (resp. degenerate and non-degenerate) hyperideal AdS polyhedra up to elements of $\Isom_0(\AdS)$  in terms of dihedral angles (resp. the induced metric on the boundary) is proper.

\subsection{Properness of $\Psi$}
\label{sec:properness_angles}

In this part, we aim to show that the map  $\Psi:\cP_N\rightarrow \cA_N$ (which takes each $P\in\cP_N$ to the exterior dihedral angle-assignment on its edges) is proper (see Lemma \ref{lem:properness of psi}).

\subsubsection{Definitions and notations}\label{subsec:def}

We first clarify the topology of the space $\cP_N\cup\polyg_N$.

Recall that each equator of a marked polyhedron or polygon in $\tilde{\cP}_N\cup\widetilde{\polyg}_N$ is identified with an oriented $N$-cycle graph. Let $V$ be the set of vertices of this graph. Assume that we fix once and for all an affine chart $x_4=1$. For each $P\in\tilde{\cP}_N\cup\widetilde{\polyg}_N$, the convexity of $P$ implies that $P$ is the convex hull in $\RP^3$ of its $N$ vertices. It is natural to endow the space $\tilde{\cP}_N\cup\widetilde{\polyg}_N$ with the topology induced by the embedding $\phi: \tilde{\cP}_N\cup\widetilde{\polyg}_N\rightarrow (\RP^3)^V$ which associates to each $P\in\tilde{\cP}_N\cup\widetilde{\polyg}_N$ its vertex set in $\RP^3$.

The space $\cP_N\cup\polyg_N=(\tilde{\cP}_N\cup\widetilde{\polyg}_N)/\Isom_0(\AdS)$ therefore inherits the quotient topology from $\tilde{\cP}_N\cup\widetilde{\polyg}_N$. It is not hard to check that $\cP_N\cup\polyg_N$ is Hausdorff (see e.g. \cite[Lemma 8]{bao-bonahon}).  More precisely, whenever two sequences $(\tilde{P}_n)_{n\in \N}$ in $\tilde{\cP}_N\cup\widetilde{\polyg}_N$ and $(g_n)_{n\in \N}$ in $\rm{Isom}_0(\AdS)$ are such that $(\tilde{P}_n)$ converges to some $\tilde{P}_{\infty}\in\tilde{\cP}_N\cup\widetilde{\polyg}_N$ and $(g_n\tilde{P}_n)$ converges to some $\tilde{P}'_{\infty}\in\tilde{\cP}_N\cup\widetilde{\polyg}_N$, there is a $g\in\Isom_0(\AdS)$ such that $\tilde{P}'_{\infty}=g(\tilde{P}_{\infty})$ and $g_n\to g$. (Notice that those relatively simple statements hold only under the assumption that the number of vertices remains constant in the limit. When this is not the case, more interesting phenomena can occur.)

The spaces $\cP_N$ and $\cP_{\Gamma}$ (with $\Gamma\in\Graph(\Sigma_{0,N},\gamma)$), endowed with the topology induced from that of $\cP_N\cup\polyg_N$, are naturally Hausdorff. 
Recall that $\cP_N$ is a disjoint union of $\cP_{\Gamma}$ over all $\Gamma\in\Graph(\Sigma_{0,N},\gamma)$. We now define a way of gluing for those $\cP_{\Gamma}$s. If $\Gamma'\in\Graph(\Sigma_{0,N},\gamma)$ is a 
proper subgraph of $\Gamma$ (a relation that we will abusively denote by $\Gamma'\subset \Gamma$) and if $P'\in \cP_{\Gamma'}$ (assume that $\cP_{\Gamma'}$ is non-empty), then $P'$ can be deformed to a polyhedron $P\in \cP_\Gamma$ (this follows from the Steiner theorem because, since $\Gamma'$ is 3-connected, $\Gamma$ must also be $3$-connected). As a consequence, if $\cP_{\Gamma'}$ is non-empty, it can be identified with a subset of $\partial\cP_\Gamma$. Now if $\Gamma_1, \Gamma_2,\Gamma'\in \Graph(\Sigma_{0,N},\gamma)$ are such that $\Gamma'$ is a proper subgraph of both $\Gamma_1$ and $\Gamma_2$, then any $P'\in \cP_{\Gamma'}$ can be identified with a boundary point of both $\cP_{\Gamma_1}$ and $\cP_{\Gamma_2}$. Therefore, if $\Gamma_1\cap\Gamma_2$ (the graph having as edges the edges common to both $\Gamma_1$ and $\Gamma_2$) is 3-connected, there is a natural \emph{gluing} of $\cP_{\Gamma_1}$ to $\cP_{\Gamma_2}$ along parts of their boundaries, which can be identified with the closure of $\cP_{\Gamma_1\cap \Gamma_2}$ in the relative topology of $\cP_N$.

For each $\Gamma\in\Graph(\Sigma_{0,N},\gamma)$, the space $\cA_{\Gamma}$ has a natural subspace topology from $\R^{E(\Gamma)}$ (which is equipped with the weak$^*$ topology). Whenever $\Gamma'\in \Graph(\Sigma_{0,N},\gamma)$ satisfies $\Gamma'\subset\Gamma$, there is a natural way to identify $\cA_{\Gamma'}$ as a subset of $\partial\cA_{\Gamma}$ whose functions have value 0 at the edges of $E(\Gamma)\setminus E(\Gamma')$. As for spaces of polyhedra, if $\Gamma_1\cap \Gamma_2$ is 3-connected, there is a natural \emph{gluing} along parts of $\partial \cA_{\Gamma_1}$ and $\partial\cA_{\Gamma_2}$ corresponding to $\overline{\cA_{\Gamma_1\cap \Gamma_2}}\cap\cA_N$, where the closure $\overline{\cA_{\Gamma_1\cap \Gamma_2}}$ is taken with respect to the topology of $\R^{E(\Gamma_1\cap\Gamma_2)}$.
This defines a topology on $\cA_{N}$.

Let $(P_n)_{n\in \N}$ be a sequence of marked non-degenerate polyhedra (up to isometry) in $\cP_N$. Note that the 1-skeleton of each $P_n$ is identified to a graph $\Gamma_n$ in $\Graph(\Sigma_{0,N},\gamma)$ and $\Graph(\Sigma_{0,N},\gamma)$ is finite, after passing to a subsequence, we can assume that all the $P_n$ have the same combinatorics say $\Gamma\in\Graph(\Sigma_{0,N},\gamma)$.  Let $\tilde{P}_n$ be a representative of $P_n$. Since all the $\tilde{P}_n$ lie in $\RP^3$ and $\RP^3$ is compact, after passing to a subsequence, each vertex of the $\tilde{P}_n$ converges to a point in $\RP^3$. Therefore, $(\tilde{P}_n)_{n\in \N}$ (we keep the notation for the convergent subsequence) converges to the convex hull in $\RP ^3$ of the limit points, denoted by $\tilde{P}_{\infty}$, called the \emph{limit} of $(\tilde{P}_n)_{n\in \N}$. In particular, we call the limit of a vertex (resp. an edge, a face) of the $\tilde{P}_n$ a \emph{limit vertex} (resp. \emph{limit edge}, \emph{limit face}).

It is easy to see that the limit $\tilde{P}_{\infty}$ is locally convex in $\RP^3$ (i.e. $\tilde{P}_{\infty}$ is convex in the chosen affine chart). However, $\tilde{P}_{\infty}$ is possibly not contained in the affine chart. It may not be convex, or not be hyperideal (for instance, with some edges tangent to $\partial\AdS$), or possibly have fewer vertices than the $\tilde{P}_n$. Moreover, vertices of the $\tilde{P_n}$ may converge to a point of $\tilde{P}_{\infty}$ which is not a vertex.

\subsubsection{Some basic lemmas}
The following gives some basic properties about the limit of a convergent sequence of polyhedra in $\tilde{\cP}_N$, which will be used in the proof of Lemma \ref{lem:properness of psi} and Lemma \ref{lem:proper_metrics}.

\begin{lemma}\label{lm:limit property}
Let $(\tilde{P}_n)_{n\in \N}$ be a sequence of marked non-degenerate hyperideal AdS polyhedra in $\tilde{\cP}_N$ which converge to the limit $\tilde{P}_{\infty}$. Assume that $\tilde{P}_{\infty}$ is contained in an affine chart of $\RP^3$.
Then the following statements hold:
\begin{enumerate}
\item  The vertices of the $\tilde{P}_n$ converging to a limit point say $w$ of $\tilde{P}_{\infty}$ are consecutive on the equator of the $\tilde{P}_n$.
\item If at least two vertices of the $\tilde{P}_n$ converge to a limit point say $w$ of $\tilde{P}_{\infty}$, then $w$ lies on $\partial\AdS$.

\item Let $(f_n)_{n\in\N}$ be a sequence of faces of $\tilde{P}_n$ converging to a limit face $f_{\infty}$ of $\tilde{P}_{\infty}$. Assume that the dihedral angle at one edge on $\partial f_n$ converges to a finite real number, then $f_{\infty}$ is not lightlike.

\item Let $(e_n)_{n\in \N}$ be a sequence of edges of $\tilde{P}_n$ with the limit $e_{\infty}$ tangent to $\partial\AdS$. Assume that the dihedral angles at $e_n$ converges to a finite real number, then $e_{\infty}$ is not lightlike.
\item Assume that the limit $\tilde{P}_{\infty}$ has at least three non-linear vertices and all its faces are spacelike. If $\tilde{P}_{\infty}$ has an edge $e_{\infty}$ tangent to $\partial\AdS$, then $e_{\infty}$ lies on the equator of $\tilde{P}_{\infty}$.
\end{enumerate}
\end{lemma}

\begin{proof}
    We assume that all the polyhedra $\tilde{P}_n$ are contained in the affine chart $\{x_4=1\}$.

    Statement (1) clearly holds if there is only one vertex of the $\tilde{P}_n$ converging to $w$. For the case that at least two vertices of the $\tilde{P}_n$ converge to $w$, suppose by contradiction that the vertices are not consecutive on the equator of $\tilde{P}_n$.
    Then there are two non-adjacent vertices say $v^1_n$, $v^2_n$ on the equator of $\tilde{P}_n$ converging to $w$ outside of $\AdS$. It is clear that $w$ lies outside of $\AdS$. We claim that $w$ is a point at infinity of the affine chart $\{x_4=1\}$. Otherwise, if $w$ lies on $\partial\AdS$, then one of the two paths contained in the equator of $\tilde{P}_n$ separated by $v^1_n$ and $v^2_n$ would degenerate to $w$, which implies that the other vertices in this path also converge to $w$, a contradiction. If $w$ lies outside of the closure of $\AdS$ and is contained in $\{x_4=1\}$, then the geodesic segment connecting $v^1_n$, $v^2_n$ would lie outside of the closure of $\AdS$ for $n$ large enough, which contradicts Corollary \ref{cor:geodesic}. Note that $w$ is the common limit of the endpoints of the geodesic segment connecting $v^1_n$ and $v^2_n$ (which intersects $\AdS$ by Corollary \ref{cor:geodesic}), see Figure \ref{fig:consecutive} for instance. 
    This implies that $\tilde{P}_{\infty}$ contains a projective line, a contradiction.
\begin{figure}
\usetikzlibrary{decorations.pathreplacing}
\begin{tikzpicture}[scale=0.83]
\usetikzlibrary{decorations.pathreplacing}
\draw (-7.5,4) .. controls (-1,4) and (-1,4) .. (5.4,4);
\draw (-7.7,0) .. controls (-1,0) and (-1,0) .. (5.6,0);
\draw[dashed] (-7.6,2) .. controls (-1,2) and (-1,2) .. (5.6,2);
\draw (-1,4) .. controls (-2.8,2) and (-2.8,2) .. (-2.8,2) .. controls (-1,0) and (-1,0) .. (-1,0) .. controls (0.8,2) and (0.8,2) .. (0.8,2) .. controls (-1,4) and (-1,4) .. (-1,4);
\draw (-1,4) .. controls (-4,2) and (-4,2) .. (-4,2) .. controls (-1,0) and (-1,0) .. (-1,0) .. controls (2,2) and (2,2) .. (2,2) .. controls (-1,4) and (-1,4) .. (-1,4);
\draw (-1,4) .. controls (-7,2) and (-7,2) .. (-7,2) .. controls (-1,0) and (-1,0) .. (-1,0) .. controls (5,2) and (5,2) .. (5,2) .. controls (-1,4) and (-1,4) .. (-1,4);
\node at (-7,1.65) {$v^1_n$};
\node at (5,1.65) {$v^2_n$};
\node at (-4,1.65) {$v^1_2$};
\node at (-3,1.65) {$v^1_1$};
\node at (0.9,1.65) {$v^2_1$};
\node at (2,1.65) {$v^2_2$};
\draw [dotted][thick] (-5,1.75) -- (-5.4,1.75);
\draw [dotted][thick] (3,1.75) -- (3.5,1.75);
\draw[white][fill=gray,opacity=0.3](-1,4) .. controls (-1,0) and (-1,0) .. (-1,0) .. controls (-0.7,2) and (-0.7,2) .. (-0.7,2) .. controls (-1,4) and (-1,4) .. (-1,4);
\draw(-1,4) .. controls (-1,0) and (-1,0) .. (-1,0);
\node at (0,2.5) {$\tilde{P}_1$};
\node at (0.98,2.4) {$\tilde{P}_2$};
\node at (2.79,2.4) {$\tilde{P}_n$};
\draw [dotted][thick] (1.8,2.4) -- (2.2,2.4);
\node at (5.2,3.5) {$\tilde{P}_{\infty}$};
\end{tikzpicture}
\caption{\small{A sequence of tetrahedra $\tilde{P}_n$ with the limit of two non-adjacent vertices $v^1_n$ and $v^2_n$ converging to a point $w$ at infinity in the chosen affine chart, whose limit $\tilde{P}_{\infty}$ is a prism with three parallel bi-infinite edges (sharing a single vertex $w$ at infinity).}}
\label{fig:consecutive}
\end{figure}

   For Statement (2), we first note that by Statement (1), the vertices of $\tilde{P}_n$ converging to $w$ are consecutive on the equator. Suppose by contradiction that    $w$ lies outside of the closure of $\AdS$. Then the edge connecting two adjacent vertices (converging to $w$) on the equator of the $\tilde{P}_n$ does not intersect $\AdS$ for $n$ sufficiently large, a contradiction.

   For Statement (3), suppose by contradiction that $f_{\infty}$ is lightlike. Then the dihedral angles of $\tilde{P}_n$ at the edges on $\partial f_n$ tend to infinity (recall that the amount \eqref{def:angle amount} of dihedral angle is defined in terms of hyperbolic rotation).

   For Statement (4), suppose by contradiction that $e_{\infty}$ is lightlike. Then the faces of $\tilde{P}_{\infty}$ adjacent to $e_{\infty}$ have to be lightlike, and thus there would be a face say $f_n$ of $\tilde{P}_n$ adjacent to $e_n$ whose limit is lightlike. By Statement (3) and the assumption on the dihedral angles at $e_n$, the limit of $f_n$ is not lightlike, which leads to a contradiction.

   For Statement (5), $\tilde{P}_{\infty}$ is convex by assumption and there are exactly two spacelike faces adjacent to $e_{\infty}$ (note that the boundary surface is viewed as two-sided if $\tilde{P}_{\infty}$ is degenerate), which are either future-directed or past-directed. We claim that the two faces adjacent to $e_{\infty}$ are respectively future-directed and past-directed. 
   Indeed, since the two faces are spacelike and their common edge $e_{\infty}$ is tangent to $\partial\AdS$, they intersect $\AdS$ and lie in the same spacelike quadrant (locally divided by the light-cone of $e_{\infty}$).
   As a consequence, $e_{\infty}$ lies on the equator of $\tilde{P}_{\infty}$.
\end{proof}

We will see from the following lemma that after choosing appropriate representatives say $\tilde{P}_n$ of $P_n$, the limit of $(\tilde{P}_n)_{n\in\N}$ has at least three non-colinear vertices.

\begin{lemma}\label{lm:representatives}
  Let $(P_n)_{n\in \N}$ be a sequence of marked non-degenerate polyhedra (considered up to isometries) in $\cP_{\Gamma}$ with representatives $\tilde{P_n}$ converging to a limit in $\RP^3$ which is contained in an affine chart.
  Then there exist a sequence $(g_n)_{n\in \N}$ of elements of $\Isom_0(\AdS)$ such that $(g_n\tilde{P}_n)_{n\in \N}$ converges to a limit in $\RP^3$ with at least three non-colinear vertices.
\end{lemma}

\begin{proof}
	The proof uses the following steps. The idea is to normalize three points in the intersection of the $\tilde{P}_n$ with $\partial\AdS$, by applying a sequence of isometries.
	
	\textbf{Step 1.} By extracting a subsequence, we can assume that for each vertex $v$ of $\Gamma$, the corresponding vertex $v_n$ of $\tilde{P_n}$ is either ideal for all $n$, or is strictly hyperideal for all $n$. We arbitrarily fix a face say $f$ of $\Gamma$ and consider the corresponding face $f_n$ of the $\tilde{P}_n$ (which are all spacelike by Lemma \ref{lm:face_spacelike}).
We fix a totally geodesic spacelike plane $H_0:=\{[x_1,x_2,0,1]\in\RP^3:x^2_1+x^2_2-1<0\}$ in $\AdS$. For each face $f_n$ of $\tilde{P}_n$, there is an $h_n\in\Isom_0(\AdS)=\PO_0(2,2)$ such that $h_n(f_n)$ lies on the same projective plane in $\RP^3$ as $H_0$.

\textbf{Step 2.} It is clear that $h_n(\tilde{P}_n)$ has the same combinatorics $\Gamma$ as $\tilde{P}_n$ for all $n$.
Moreover, the boundary of the face $h_n(f_n)$ of $h_n(\tilde{P}_n)$ intersects $\partial\AdS$ in at least three points (denoted by $\mathcal{I}_n$), which lie on the boundary of $H_0$. 
    We arbitrarily choose three distinct points from $\mathcal{I}_1$ (and the corresponding three points from $\mathcal{I}_n$ for all $n>1$), denoted by $v^1_n$, $v^2_n$, $v^3_n$.
   Since the stabilizer of $H_0$ contained in $\Isom_0(\AdS)$, say ${\Stab}_0(H_0)$, consists of the matrices of the following form (written in block matrices):
     $$ \left[\left(
   \begin{array}[h]{cccc}
     A  & \bigzero & B \\
     \bigzero & 1 & 0 \\
     C & 0 & D \\
   \end{array}
 \right)\right]~~ $$
with $$ \left[\left(\begin{array}[h]{cccc}
     A  & B   \\
     C  & D \\
   \end{array}
   \right)\right]~\in\PO_0(2,1)=\Isom_0(\bH^2)~,$$
   which acts transitively on triples of points in $\partial H_0$ (identified with $\partial\bH^2$), there exists a unique $u_n\in{\Stab}_0(H_0)$ which takes $v^1_n$, $v^2_n$, $v^3_n$ in $\partial H_0$ to $\xi_1:=[1,0,0,1]$, $\xi_2:=[0,1,0,1]$, $\xi_3:=[-1,0,0,1]$ in $\partial H_0$, respectively.

   We conclude that after applying $g_n=u_n\circ h_n\in\Isom_0(\AdS)$ to each $\tilde{P}_n$, the limit contains an ideal triangle on $P_0$ with vertices $\xi_1$, $\xi_2$, $\xi_3$.
   Combined with the assumption that $\tilde{P}_{\infty}$ is contained in an affine chart, therefore, $\tilde{P}_{\infty}$ has at least three non-colinear vertices.
\end{proof}

\begin{remark}\label{rk:representatives}
	Using Lemma \ref{lm:representatives}, henceforth, for any sequence $(P_n)_{n\in \N}$ in $\cP_{\Gamma}$ satisfying the assumptions in this lemma, we always choose representatives $\tilde{P}_n$ of $P_n$ such that the sequence $(\tilde{P}_n)_{n\in \N}$ have a limit with at least three non-colinear vertices.
\end{remark}

\subsubsection{The dynamics of pure AdS translations}\label{subsubsect:dynamics}

 The argument of Lemma \ref{lem:properness of psi} is based on the dynamics of a particular class of AdS isometries acting on $\RP^3$.

 More precisely, we consider the action on $\RP^3$ of a pure translation along a complete space-like geodesic in $\AdS$.
 Such pure translations are transformations of $\RP^3$ induced by matrices of the form
 $$ \left(
   \begin{array}[h]{cccc}
     \cosh(t)  & 0 & 0 & \sinh(t) \\
     0 & 1 & 0 & 0 \\
     0 & 0 & 1 & 0 \\
     \sinh(t) & 0 & 0 & \cosh(t)
   \end{array}
 \right)~, $$
 for some 
 $t>0$. This transformation (denoted by $u_t$) is a pure translation along the line $\ell$ with endpoints the points at infinity $v_+=[1,0,0,1]$ and $v_-=[-1,0,0,1]$.
 Let $P_{-}$, $P_{+}$ denote the planes tangent to $\partial\AdS$ at $v_{-}$, $v_{+}$.
 Note that $u_t$ preserves $v_{-}$, $v_{+}$ and $\partial\AdS$. It is quite clear that:
 \begin{itemize}
 \item $v_{-}$ corresponds to an eigenvector of $u_t$ with eigenvalue $e^{-t}$ less than 1;
 \item $v_{+}$ corresponds to an eigenvector of $u_t$ with eigenvalue $e^t$ greater than 1;
 \item the circle at infinity (in the affine chart $x_4=1$ of $\AdS$) of the tangent plane $P_{-}$ or $P_{+}$ corresponds to a 2-dimensional eigenspace of $u_t$ with eigenvalue 1.
 \end{itemize}

 By direct computation, the action of $u_t$ on the planes $P_{\pm}$ in the affine chart $x_4=1$ is given by $(\pm 1,x,y)\mapsto (\pm 1,e^{\mp t}x, e^{\mp t}y)$.
 Similarly, we can compute the action of $u_t$ on the planes containing $\ell$. To conclude, the action of $u_t$ in the chosen affine chart has the following description:
 \begin{itemize}
 \item $u_t$ preserves the tangent planes $P_{-}$ and $P_{+}$ respectively, and acts as an expanding (resp. contracting) homothety on $P_{-}$ (resp. $P_{+}$) with fixed point $v_{-}$ (resp. $v_{+}$);
 \item on the region between $P_{-}$ and $P_{+}$, $u_t$ acts with repulsive fixed point at $v_{-}$ and attractive fixed point at $v_{+}$ (this can be viewed as a ``north-south type" dynamics);
 \item on the complement in $\R^3$ of the region between $P_{-}$ and $P_{+}$, the dynamics of $u_t$ is similar, with each trajectory away from $P_{-}$ and towards $P_{+}$.
 \end{itemize}

\subsubsection{The proof}

 We are ready to prove the properness.
 Let $(P_n)_{n\in \N}$ be a sequence of hyperideal polyhedra in $\cP_N$ with image $\theta_n:=\Psi(P_n)$ converging to an element say $\theta_{\infty}$ of $\cA_N$. Note that $\cP_N$ is a disjoint union of $\cP_{\Gamma}$ over all $\Gamma\in\Graph(\Sigma_{0,N},\gamma)$ and $\Graph(\Sigma_{0,N},\gamma)$ is finite. Up to extracting a subsequence, we assume that $(P_n)_{n\in \N}$ is a sequence in $\cP_{\Gamma}$ for a fixed graph $\Gamma\in\Graph(\Sigma_{0,N},\gamma)$. Note also that $\cA_N$ is a disjoint union of $\cA_{\Gamma}$ over all $\Gamma\in\Graph(\Sigma_{0,N},\gamma)$. Then $\theta_{\infty}\in\cA_{\Gamma'}$ for some graph $\Gamma'\in\Graph(\Sigma_{0,N},\gamma)$. Since $\theta_n\in\cA_{\Gamma}$ for all $n$, by the gluing structure in $\cA_{N}$ (see subsection \ref{subsec:def}), it follows that $\Gamma'$ is a subgraph of $\Gamma$ (namely $\Gamma'$ is either equal to $\Gamma$ or is a proper subgraph of $\Gamma$), where $E(\Gamma')$ is the support of $\theta_{\infty}$ (viewed as a limit function of the sequence $\theta_{n}: E(\Gamma)\rightarrow\R$). We aim to show that $(P_n)_{n\in \N}$ converges to an element of $\cP_{\Gamma'}$.

\begin{proof}[\textbf{Proof of Lemma \ref{lem:properness of psi}}]
	
	Let $(P_n)_{n\in \N}$ be a sequence of marked polyhedra in $\cP_{\Gamma}$ with representatives $\tilde{P}_n$ converging to a limit say $\tilde{P}_{\infty}$ in $\RP^3$, and with the angle-assignments  $\theta_n:=\Psi(P_n)$ converging to a function $\theta_{\infty}\in\cA_{\Gamma'}$ for $\Gamma'\in\Graph(\Sigma_{0,N},\gamma)$ which is a subgraph of $\Gamma$.
We need to show that the equivalence class of $\tilde{P}_{\infty}$ lies in $\cP_{\Gamma'}$ (namely, $\tilde{P}_\infty$ is non-degenerate convex hyperideal with the 1-skeleton $\Gamma'$). We prove the lemma using the following steps.

\begin{Step} \label{step:no.0}
$\tilde{P}_{\infty}$ is convex in $\RP^3$.
\end{Step}

As the limit of the sequence $(\tilde{P}_n)$ of convex polyhedra, $\tilde{P}_{\infty}$ is locally convex. It suffices to show that $\tilde{P}_{\infty}$ is contained in an affine chart. Otherwise, $\tilde{P}_{\infty}$ contains a projective line and is therefore a bi-infinite prism or strip in any affine chart, with the endpoints at infinity on both sides of all the (parallel) edges corresponding to the same point (say $q_{\infty}$) in $\RP^3$.

We first claim that there are exactly two vertices of $\Gamma$, say $v^1$, $v^2$, such that the corresponding vertices $v^1_n$, $v^2_n$ of $\tilde{P}_n$ converge to $q_{\infty}$ along the two opposite sides of the geodesic segment connecting $v^1_n$ and $v^2_n$ respectively (see Figure \ref{fig:consecutive} for instance). Indeed, if two vertices of $\tilde{P}_n$ converge to $q_{\infty}$ on the same side of parallel edges of $\tilde{P}_{\infty}$, then the geodesic segment connecting them would lie outside of the closure of $\AdS$ for $n$ sufficiently large, which contradicts Corollary \ref{cor:geodesic}. We now split the discussion into the following two cases:
\begin{figure}
   \begin{subfigure}[b]{0.46\textwidth}
    \centering
\begin{tikzpicture}[scale=0.8]
\draw [thick] (-0.5,1) node (v1) {} ellipse (2 and 2);
\draw  [black][fill=black,opacity=1](1.5,0.85) ellipse (0.04 and 0.04);
\draw  [black][fill=black,opacity=1](-1.3,2.83) ellipse (0.04 and 0.04);
\draw  [black][fill=black,opacity=1](-1.3,-0.83) ellipse (0.04 and 0.04);
\draw [blue](-1.3,-0.83) .. controls (-1.3,2.83) and (-1.3,2.83) .. (-1.3,2.83);
\draw  [black][fill=black,opacity=1](0.3,2.83) ellipse (0.04 and 0.04);
\draw  [black][fill=black,opacity=1](0.3,-0.83) ellipse (0.04 and 0.04);
\draw [blue](0.3,-0.83) .. controls (0.3,2.83) and (0.3,2.83) .. (0.3,2.83);
\draw  [black][fill=black,opacity=1](-2.5,0.85) ellipse (0.04 and 0.04);
\draw[decorate,decoration=brace] (-1.35,3.1) -- (0.38,3.1);
\draw[decorate,decoration={brace,mirror}] (-1.38,-1.15) -- (0.42,-1.15);
\node at (-0.5,3.5) {$\mathcal{V}_1$};
\node at (-0.45,-1.5) {$\mathcal{V}_2$};
\draw (-1.15,2.9);
\node at (-2.85,0.9) {$v^1$};
\node at (1.8,0.85) {$v^2$};
\draw[densely dashed](-2.5,0.85) .. controls (1.5,0.85) and (1.5,0.85) .. (1.5,0.85);
\end{tikzpicture}
\end{subfigure}
    \begin{subfigure}[b]{0.46\textwidth}
     \centering
\begin{tikzpicture}[scale=0.8]
\draw [thick] (-0.5,1) node (v1) {} ellipse (2 and 2);
\draw  [black][fill=black,opacity=1](1.5,0.85) ellipse (0.04 and 0.04);
\draw  [black][fill=black,opacity=1](-1.3,2.83) ellipse (0.04 and 0.04);
\draw  [black][fill=black,opacity=1](-1.3,-0.83) ellipse (0.04 and 0.04);
\draw [blue](-1.3,-0.83) .. controls (-1.3,2.83) and (-1.3,2.83) .. (-1.3,2.83);
\draw  [black][fill=black,opacity=1](0.3,2.83) ellipse (0.04 and 0.04);
\draw  [black][fill=black,opacity=1](0.3,-0.83) ellipse (0.04 and 0.04);
\draw [blue](0.3,-0.83) .. controls (0.3,2.83) and (0.3,2.83) .. (0.3,2.83);
\draw  [black][fill=black,opacity=1](-2.5,0.85) ellipse (0.04 and 0.04);
\draw[decorate,decoration=brace] (-2.45,3.2) -- (1.6,3.2);
\node at (-0.5,3.6) {$\mathcal{V}_1$};
\draw (-1.15,2.9);
\node at (-1.3,-1.2) {$v^1$};
\node at (0.35,-1.2) {$v^2$};
\node at (-0.5,-0.75) {$e$};
\draw[densely dashed](-2.5,0.85) .. controls (1.5,0.85) and (1.5,0.85) .. (1.5,0.85);
\end{tikzpicture}
    \end{subfigure}
    \caption{\small{Examples of $\Gamma$ for cases (1) and (2) in Step \ref{step:no.0}} (where $N=6$), shown in the left and right pictures respectively. The (bold) circle represents the equator $\gamma$ and the dashed edges are on the bottom (delimited by the equator $\gamma$) of $\Sigma_{0,N}$.}
    \label{fig:step1_convex}
    \end{figure}
\begin{enumerate}
\item If $v^1$ and $v^2$ are non-adjacent along the equator $\gamma$ of $\Gamma$ (see e.g. Figure \ref{fig:step1_convex}), they separate the vertices of $\gamma$ into two groups, collected in $\cV_1$ and $\cV_2$ respectively. We denote by $E_{12}$ the set of the edges of $\Gamma$ connecting one point in $\cV_1$ and the other point in $\cV_2$. For simplicity, we identify the vertices and edges of $\tilde{P}_n$ with corresponding ones in $\Gamma$. Note that any vertex in $\cV_1$ is non-adjacent to any vertex in $\cV_2$ along the equator, the edges of $\tilde{P}_n$ in $E_{12}$ are non-equatorial and lie either on the top or the bottom of $\tilde{P}_n$, moreover, they are non-adjacent to $v^1_n$ or $v^2_n$. As a result, the limit of each edge of $\tilde{P}_n$ in $E_{12}$ would lie in the interior of a face of $\tilde{P}_{\infty}$ and intersects the two parallel edges  of that face at their endpoints. This implies that the dihedral angles at each edge of $\tilde{P}_n$ in $E_{12}$ tends to zero as $n\rightarrow \infty$. Recall that $E(\Gamma')$ is the support of the limit function $\theta_{\infty}$, $\Gamma'$ is thus contained in the subgraph (say $\Gamma_1$) of $\Gamma$ obtained by removing the edges in $E_{12}$. Since $\Gamma_1$ is disconnected after removing $v^1$, $v^2$ and their incident edges, $\Gamma'$ is not 3-connected, which contradicts our hypothesis that $\Gamma'\in\Graph(\Sigma_{0,N},\gamma)$.

  \item If $v^1$ and $v^2$ are adjacent along the equator $\gamma$ of $\Gamma$ (see e.g. Figure \ref{fig:step1_convex}), let $\cV_1:=V(\Gamma)\setminus\{v^1,v^2\}$ and let $e$ denote the equatorial edge connecting $v^1$ and $v^2$. Note that the corresponding vertices $v^1_n$, $v^2_n$ of $\tilde{P}_n$ tend to $q_{\infty}$ on opposite sides of the edge $e_n$ connecting them. By the convexity of $\tilde{P}_n$, the limit of the vertices of $\tilde{P}_n$ in $\cV_1$ would lie on the same line parallel to the limit of $e_n$.
      Therefore, $\tilde{P}_{\infty}$ is a bi-infinite strip and the dihedral angle at the equatorial edge $e_n$ of $\tilde{P}_n$ would tend to zero as $n\rightarrow \infty$. This would imply that $\gamma$ is not contained in $\Gamma'$ and thus $\Gamma'\not\in\Graph(\Sigma_{0,N},\gamma)$, a contradiction.
\end{enumerate}

Combining the above two cases, we complete Step \ref{step:no.0}.

	\begin{Step} \label{step:no.1}
We claim that $\tilde{P}_{\infty}$ has $N$ vertices.
	\end{Step}
	
	Note that all the $\tilde{P}_n$ are convex hyperideal with each vertex passed through by a Hamiltonian cycle which separates the boundary $\partial\tilde{P}_n$ into future and past faces, the limit in $\tilde{P}_{\infty}$ of any vertex of the $\tilde{P}_n$ cannot lie in the interior of any face of $\tilde{P}_{\infty}$ (here $\tilde{P}_{\infty}$ has at least three non-colinear vertices by Step \ref{step:no.0}, Lemma \ref{lm:representatives} and Remark \ref{rk:representatives}). Suppose by contradiction that $\tilde{P}_{\infty}$ has fewer vertices. Then it falls into one of the following two cases:

\begin{enumerate}[(a)]
\item There are at least two vertices of the $\tilde{P}_n$ converging to a point say $v_{\infty}$ of $\tilde{P}_{\infty}$.
\item Case (a) does not occur, but there is a vertex of the $\tilde{P}_n$ whose limit lies in the interior of an edge of $\tilde{P}_{\infty}$.
\end{enumerate}

We first discuss Case (a). By Statement (1) of Lemma \ref{lm:limit property},  the vertices of $\tilde{P}_n$ converging to $v_{\infty}$ are consecutive on the equator of $\tilde{P}_n$.
By Statement (2) of Lemma \ref{lm:limit property}, $v_{\infty}$ lies on $\partial\AdS$. Moreover, $v_{\infty}$ is either a vertex of $\tilde{P}_{\infty}$ or a point lying in the interior of an edge of $\tilde{P}_{\infty}$. Note that an edge of $\tilde{P}_{\infty}$ either passes through $\AdS$ or is tangent to $\partial\AdS$, we split the discussion into the following several claims according to the possible cases of the limit edges adjacent to $v_{\infty}$.

The trick is applying the dynamics of AdS pure translations (along a space-like geodesic in $\AdS$) to ``zoom in'' the domain near $v_{\infty}$ of $\tilde{P}_{\infty}$ so that we can reveal the asymptotically degenerate behavior of the dihedral angles at some edges (for instance, the edges adjacent to those vertices converging to $v_{\infty}$) of $\tilde{P}_n$ from the degenerate geometric phenomenon near $v_{\infty}$ of $\tilde{P}_{\infty}$. Indeed, up to isometries, we can assume that $v_{\infty}$ lies at the endpoint $v_{-}=[-1,0,0,1]$ at infinity of the complete spacelike geodesic $\ell$ (which is contained in the coordinate axis $Ox_1$ of the affine chart $x_4=1$ of $\AdS$) mentioned in subsection \ref{subsubsect:dynamics}. We denote by $w_{\infty}$ the other endpoint at infinity of $\ell$, which coincides with $v_{+}=[1,0,0,1]$. Let $P_v$ (resp. $P_w$) be the tangent plane of $\partial\AdS$ at $v_{\infty}$ (resp. $w_{\infty}$). We consider a pure AdS translation say $u_n$ in $\Isom_0(\AdS)$ along $\ell$ with translation length $L_n>0$, attractive fixed point $w_{\infty}$ and repulsive fixed point $v_{\infty}$.  The dynamics of the $u_n$ acting on $\RP^3$ (see subsection \ref{subsubsect:dynamics}) is a technical tool throughout the following proof.
	
	\begin{Claim}\label{clm:no tangent}
		If all the limit edges adjacent to $v_{\infty}$ pass through $\AdS$, then $\theta_{\infty}$ does not satisfy Condition (iii).
	\end{Claim}
		
	  By carefully choosing the translation length $L_n$ of $u_n$, we can control the approaching speed towards $w_{\infty}$ of the images under $u_n$ of the points disjoint from $P_{v}$ in the affine chart. In particular, we choose $L_n$ such that under the action of $u_n$, those vertices of the $\tilde{P}_n$ converging to $v_{\infty}$ keep converging to $v_{\infty}$, while all the other vertices (note that the number of the other vertices is at least two, since $\tilde{P}_{\infty}$ has at least two more vertices apart from $v_{\infty}$ by Lemma \ref{lm:representatives}) of $\tilde{P}_n$ go towards $w_{\infty}$ in the limit. This is realizable since the number of those vertices of the $\tilde{P}_n$ converging to $v_{\infty}$ is finite and all the other vertices of $\tilde{P}_n$ are disjoint from $P_v$, we can choose $L_n$ with $L_n\rightarrow +\infty$ and such that the convergence speed of those vertices of the $\tilde{P}_n$ converging to $v_{\infty}$ is faster than their expanding speed (towards $w_{\infty}$) under $u_n$ as $n$ tends to infinity.
	
          Therefore, the limit of $u_n(\tilde{P}_n)$ is the space-like line $\ell$ in $\AdS$ with endpoints $v_\infty$, $w_{\infty}$. This divides the vertices of the $\Gamma$ into two groups: one group (say $\cV_v$) consisting of the vertices whose correspondence in $u_n(\tilde{P}_n)$ converge to $v_{\infty}$ and the other group for $w_{\infty}$ (say $\cV_w$). By the hypothesis of case (a) and the above discussion, both $\cV_v$ and $\cV_w$ have at least two vertices, and moreover, their vertices are respectively consecutive on the equator $\gamma$ (by Statement (1) of Lemma \ref{lm:limit property}).
As a subgraph of $\Gamma$ and a graph in $\Graph(\Sigma_{0,N},\gamma)$, $\Gamma'$ is 3-connected and contains the same equator $\gamma$ as $\Gamma$. Let $c$ (resp. $c'$) be the simple circuit in the dual graph $\Gamma^*$ (resp. $\Gamma'^*$) that separates the vertices in $\cV_v$ and $\cV_w$. We conclude that $c$ (resp. $c'$) does not bound any face of $\Gamma^*$ (resp. $\Gamma'^*$) and exactly two of the edges in $c$ (resp. $c'$) are dual to edges of $\gamma$ (as shown in Figure \ref{fig:claim1}). Let $E_{vw}$ (resp. $E'_{vw}$) denote the set of the edges of $\Gamma$ (resp. $\Gamma'$) connecting one point in $\cV_v$ and the other in $\cV_w$, which are exactly the edges dual to those in $c$ (resp. $c'$).

\begin{figure}
\usetikzlibrary{decorations.pathreplacing}
\begin{tikzpicture}[scale=1]
\draw  [thick](-0.5,0.5) ellipse (1.5 and 1.5);
\draw  [black][fill=black,opacity=1](0.4,1.7) ellipse (0.03 and 0.03);
\draw  [black][fill=black,opacity=1](-1.4,1.7) ellipse (0.03 and 0.03);
\draw  [black][fill=black,opacity=1](-1.4,-0.7) ellipse (0.03 and 0.03);
\draw  [black][fill=black,opacity=1](0.4,-0.7) ellipse (0.03 and 0.03);
\draw [thick][dotted] (-1.5,0.25) -- (-1.5,0.7);
\draw [thick][dotted] (0.6,0.25) -- (0.6,0.75);
\draw[decorate,decoration=brace] (-2.15,-0.75) -- (-2.2,1.75);
\draw[decorate,decoration={brace,mirror}] (1.15,-0.75) -- (1.1,1.75);
\node at (-3,0.5) {$\mathcal{V}_v$};
\node at (2,0.5) {$\mathcal{V}_w$};
\draw [blue](-0.5,2) .. controls (-0.35,2) and (-0.4,-1) .. (-0.5,-1);
\node at (-0.25,0.35) {\color{blue}$c'$};
\draw [dashed][blue](-0.5,2) .. controls (-0.6,2) and (-0.6,-1) .. (-0.5,-1);
\end{tikzpicture}
\caption{\small{An example of the simple circuit 
$c'$ in $\Gamma'^*$ in Claim \ref{clm:no tangent}. The vertices in $\cV_v$ and $\cV_w$ are consecutive on the equator $\gamma$ (shown in the bold lines) and exactly two edges of the simple circuit 
$c'$ (separating the vertices in $\cV_v$ and $\cV_w$) are dual to edges of $\gamma$.}}
\label{fig:claim1}
\end{figure}

Let $H$ be the timelike plane in $\AdS$ which is orthogonal to the spacelike geodesic $\ell$ (with endpoints $v_{\infty}$, $w_{\infty}$ at infinity) at the middle point $[0,0,0,1]$. Note that $H$ is isometric to $\bAdS^2$.  Set $q_n:=u_n(\tilde{P}_n)\cap H$. For $n$ large enough, $q_n$ is a convex polygon in $H\cong\bAdS^2$ with all the edges spacelike. For simplicity, we identify the vertices and edges of $u_n(\tilde{P}_n)$ with corresponding ones in $\Gamma$. The vertex set of $q_n$ consists of the intersection points of $H$ with the edges of $u_n(\tilde{P}_n)$ in $E_{vw}$ (see e.g. Figure \ref{fig:proper1}). Since the edges of $u_n(\tilde{P}_n)$ in $E_{vw}$ tend to $\ell$ (which is orthogonal to $H$) in the limit, the dihedral angle of the $u_n(\tilde{P}_n)$ at each edge in $E_{vw}$ tends to equal the exterior angle of $q_n$ at the vertex lying on that edge. Applying the Gauss-Bonnet formula \eqref{formula:Gauss-Bonnet} to $q_n$, the sum of the exterior angles at the vertices of $q_n$ is equal to the area of $q_n$, which tends to zero as $n\rightarrow\infty$ (observe that $q_n$ converges to the point $[0,0,0,1]$). Combined with the assumption that the dihedral angle of the $\tilde{P}_n$ at each edge of $\Gamma\setminus\Gamma'$ (if non-empty) tends to zero, the sum of the dihedral angles of the $\tilde{P}_n$ (note that $u_n$ preserves the marking and dihedral angles) at the edges dual to those in $c'$ tends to zero. This implies that the limit function $\theta_{\infty}\in\cA_{\Gamma'}$ fails to satisfy Condition (iii), a contradiction.

\begin{figure}
 \centering
   \begin{subfigure}[b]{0.46\textwidth}
     \begin{tikzpicture}[scale=0.76]
\draw[thick] (-5.5,3.95) .. controls (-6.15,3.5) and (-6.15,3.5) .. (-6.15,3.5);
\draw[thick](-6.15,3.5) .. controls (-5.5,3.05) and (-5.5,3.05) .. (-5.5,3.05);
\draw[thick] (-5.5,3.95) .. controls (1.5,3.95) and (1.5,3.95) .. (1.5,3.95);
\draw[thick] (-5.5,3.05) .. controls (1.5,3.05) and (1.5,3.05) .. (1.5,3.05);
\draw[thick] (1.5,3.95) .. controls (2.15,3.5) and (2.15,3.5) .. (2.15,3.5);
\draw[thick] (2.15,3.5) .. controls (1.5,3.05) and (1.5,3.05) .. (1.5,3.05);
\draw (-5.5,3.95) .. controls (-5.5,3.05) and (-5.5,3.05) .. (-5.5,3.05);
\draw (1.5,3.95) .. controls (1.5,3.05) and (1.5,3.05) .. (1.5,3.05);
\draw[densely dashed] (-6.15,3.5) .. controls (2.15,3.5) and (2.15,3.5) .. (2.15,3.5);
\node at (-5.5,4.35) {$v^1_n$};
\node at (-6.45,3.5) {$v^2_n$};
\node at (-5.5,2.7) {$v^3_n$};
\node at (1.5,4.35) {$w^1_n$};
\node at (2.5,3.5) {$w^2_n$};
\node at (1.5,2.7) {$w^3_n$};
\draw [line width=0.2mm,blue](-2.05,3.96) .. controls (-2.05,3.05) and (-2.05,3.05) .. (-2.05,3.05);
\draw[line width=0.2mm,blue] [densely dashed](-2.05,3.96) .. controls (-2.2,3.5) and (-2.2,3.5) .. (-2.2,3.5);
\draw[line width=0.2mm,blue][densely dashed] (-2.05,3.05) .. controls (-2.2,3.5) and (-2.2,3.5) .. (-2.2,3.5);
\node at (-2.05,4.3) {$q_n$};
     \end{tikzpicture}
   \end{subfigure}
    \begin{subfigure}[b]{0.47\textwidth}
    \begin{tikzpicture}[scale=1.05]
\draw (-0.9,2.45) .. controls (-0.7,2.1) and (-0.7,1.45) .. (-0.9,1);
\draw[thick] (-0.9,2.45) .. controls (-0.7,2.1) and (-1.1,1.7) .. (-1.45,1.7);
\draw [thick](-1.45,1.7) .. controls (-1.1,1.7) and (-0.7,1.45) .. (-0.9,1);
\draw [thick](-0.9,2.45) .. controls (0.5,1.55) and (3.5,1.55) .. (4.9,2.45);
\draw [thick](-0.9,1) .. controls (0.5,1.75) and (3.5,1.75) .. (4.9,1);
\draw (4.9,2.45) .. controls (4.65,2.05) and (4.65,1.5) .. (4.9,1);
\draw [thick](4.9,2.45) .. controls (4.65,2.05) and (5.2,1.75) .. (5.5,1.75);
\draw [thick] (4.9,1) .. controls (4.65,1.5) and (5.2,1.75) .. (5.5,1.75);
\draw[densely dashed] (-1.45,1.7) .. controls (2,1.65) and (2,1.65) .. (5.5,1.75);
\node at (-1.2,2.5) {$v^1_n$};
\node at (-1.7,1.7) {$v^2_n$};
\node at (-1.15,1.05) {$v^3_n$};
\node at (5.2,2.55) {$w^1_n$};
\node at (5.8,1.8) {$w^2_n$};
\node at (5.2,1.05) {$w^3_n$};
\node at (2,1.95) {$q_n$};
\draw[line width=0.2mm,blue] (2.01,1.78) .. controls (2,1.65) and (2,1.65) .. (2.03,1.56);
\draw[line width=0.2mm,blue][densely dashed] (2.01,1.78) .. controls (1.95,1.7) and (1.95,1.7) .. (1.9,1.67);
\draw[line width=0.2mm,blue][densely dashed] (1.9,1.67) .. controls (1.96,1.64) and (1.96,1.64) .. (2.03,1.56);
\end{tikzpicture}
   \end{subfigure}
   \caption{\small{An example of $u_n(\tilde{P}_n)$ in Claim \ref{clm:no tangent}, with vertices $v^1_n,v^2_n,v^3_n$ (resp. $w^1_n$, $w^2_n$, $w^3_n$) ideal and converging to $v_{\infty}$ (resp. $w_{\infty}$). The bold lines represent the equator. The left figure is drawn in the projective model, while the right figure is a visualization of the induced (path) metric on $\partial u_n(\tilde{P}_n)$. The convex polygon $q_n$ is contained in the timelike plane $H$ with all the edges spacelike.}}
   \label{fig:proper1}
\end{figure}

    \begin{Claim}\label{clm:one tangent}
If exactly one limit edge adjacent to $v_{\infty}$ is tangent to $\partial\AdS$ at $v_{\infty}$, then $\theta_{\infty}$ does not satisfy Condition (iv). 
    \end{Claim}

    Let $e_{\infty}$ denote the (only) limit edge tangent to $\partial\AdS$ at $v_{\infty}$ and let $v'_{\infty}$ denote the other endpoint of $e_{\infty}$.  By Statement (4) of Lemma \ref{lm:limit property} and the assumption that $\theta_{\infty}\in \cA_{\Gamma'}$ (which takes zero value on the edges of $\Gamma\setminus\Gamma'$ (if non-empty) and takes non-zero value in $(-\infty,+\infty)$ on the edges of $\Gamma'$), $e_{\infty}$ is not lightlike.
     Therefore, $v'_{\infty}$ lies outside of the closure of $\AdS$.

    By Statement (2) of Lemma \ref{lm:limit property}, exactly one vertex of the $\tilde{P}_n$ converges to $v'_{\infty}$. Since $v'_{\infty}\not=v_{\infty}$ lies on the tangent plane $P_v$ of $v_{\infty}$, the dynamics of $u_n$ (see subsection \ref{subsubsect:dynamics}) shows that $u_n(v'_{\infty})$ will tend to a point in $\RP^3$, say $v''_{\infty}$, at infinity in the affine chart along the direction from $v_{\infty}$ to $v'_{\infty}$ as the translation length $L_n\rightarrow+\infty$.  Moreover, a diagonal extraction argument shows that the vertex of the $\tilde{P}_n$ that converges to $v'_{\infty}$ will tend to $v''_{\infty}$ under $u_n$ whenever $L_n\rightarrow+\infty$. In particular, we choose $u_n$ with $L_n\rightarrow +\infty$ such that 
    under the action of $u_n$, those vertices of the $\tilde{P}_n$ converging to $v_{\infty}$ keep converging to $v_{\infty}$, the vertex of the $\tilde{P}_n$ converging to $v'_{\infty}$ tends to the point $v''_{\infty}$ at infinity, while all the other vertices (note that 
    the number of the other vertices is at least one, since $\tilde{P}_{\infty}$ has at least one more vertex apart from $v_{\infty}$ and $v'_{\infty}$ by Lemma \ref{lm:representatives}) of the $\tilde{P}_n$ go towards $w_{\infty}$. Combined with the discussion in Claim \ref{clm:no tangent}, this is realizable, observing that all the vertices of the $\tilde{P}_n$ apart from those converging to $v_{\infty}$ and $v'_{\infty}$ are disjoint from $P_v$. Therefore, the limit of $u_n(\tilde{P}_n)$ is a half-infinite strip with one edge connecting $v_{\infty}$ and $w_{\infty}$, and two edges tangent to $\partial\AdS$ at $v_{\infty}$ and $w_{\infty}$ on one side and sharing a common vertex $v''_{\infty}$ at infinity on the other side (as shown in Figure \ref{fig:half-infinite}). Up to reordering indices, let $v_0$ be the vertex of $\Gamma$ whose correspondence in the $\tilde{P}_n$ (resp. $u_n(\tilde{P}_n)$) converges to $v'_{\infty}$ (resp. $v''_{\infty}$).

  \begin{figure}
   \begin{subfigure}[b]{0.46\textwidth}
    \centering
\begin{tikzpicture}[scale=0.9]
\draw  (-0.5,1) ellipse (2 and 2);
\draw[white][fill=gray,opacity=0.3](3.8,3) .. controls (-0.5,3) and (-0.5,3) .. (-0.5,3) .. controls (-0.5,-1) and (-0.5,-1) .. (-0.5,-1) .. controls (3.8,-1) and (3.8,-1) .. (3.8,-1) .. controls (3.8,3) and (3.8,3) .. (3.8,3);
\draw [thick](3.75,3) .. controls (-0.5,3) and (-0.5,3) .. (-0.5,3) .. controls (-0.5,-1) and (-0.5,-1) .. (-0.5,-1) .. controls (3.75,-1) and (3.75,-1) .. (3.75,-1);
\node at (-0.5,3.35) {$v_{\infty}$};
\node at (-0.5,-1.35) {$w_{\infty}$};
\node at (-0.8,1) {$\ell$};
\end{tikzpicture}
\caption{\small{the half-infinite strip in Claim \ref{clm:one tangent}.}}
\label{fig:half-infinite}
\end{subfigure}
    \begin{subfigure}[b]{0.46\textwidth}
     \centering
    \begin{tikzpicture}[scale=0.9]
\draw  (-10.45,1) ellipse (2 and 2);
\draw[white][fill=gray,opacity=0.3](-7,3) .. controls (-14,3) and (-14,3) .. (-14,3) .. controls (-14,-1) and (-14,-1) .. (-14,-1) .. controls (-7,-1) and (-7,-1) .. (-7,-1) .. controls (-7,3) and (-7,3) .. (-7,3);
\node at (-10.45,3.3) {$v_{\infty}$};
\node at (-10.5,-1.3) {$w_{\infty}$};
\draw [thick](-14,3) .. controls (-7,3) and (-7,3) .. (-7,3);
\draw [thick](-14,-1) .. controls (-7,-1) and (-7,-1) .. (-7,-1);
\draw (-10.5,3) .. controls (-10.5,-1) and (-10.5,-1) .. (-10.5,-1);
\node at (-10.8,1) {$\ell$};
\end{tikzpicture}
\caption{\small{the bi-infinite strip in in Claim \ref{clm:two tangent}.}}
\label{fig:bi-infinite}
    \end{subfigure}
    \caption{\small{The limit of $u_n(\tilde{P}_n)$, shown in the shaded region. The bold lines represent the limit of the equator of $u_n(\tilde{P}_n)$.}}
    \label{fig:strip}
    \end{figure}

    Let $\mathcal{V}_v$ (resp. $\mathcal{V}_w$) denote the set consisting of the vertices of $\Gamma$ whose correspondence in the $u_n(\tilde{P}_n)$ converge to $v_{\infty}$ (resp. $w_{\infty}$).
    Note that by the hypothesis of case (a) and Statement 
    (1) of Lemma \ref{lm:limit property},  $\cV_v$ consists of at least two vertices and they are consecutive on the equator $\gamma$. Combined with the above discussion, $\cV_w$ consists of at least one vertex and they are also consecutive on the equator $\gamma$ (note that $v_0$ is adjacent along $\gamma$ to a vertex in $\cV_v$ and $V(\Gamma)=\cV_v\cup\cV_w\cup\{v_0\}$, see Figure \ref{fig:simple path}).

    Note that $\Gamma'$ is 3-connected and contains the same equator $\gamma$ as $\Gamma$. We denote by $E_{vw}$ (resp. $E'_{vw}$) the set of the edges of $\Gamma$ (resp. $\Gamma'$) connecting one point in $\mathcal{V}_v$ and the other point in $\mathcal{V}_w$. In particular, exactly two edges (say $e_1$, $e_2$) in $E'_{vw}$ are adjacent to the two faces (say $f_1$, $f_2$ respectively) containing the vertex $v_0$. Let $c'$ be the simple path in the dual graph $\Gamma'^*$ which contains the dual edges of those in $E'_{vw}$ and has endpoints dual to $f_1$ and $f_2$ respectively. It follows that $c'$ starts and ends in the boundary of the same face $v^*_0$ (dual to $v_0$) in $\Gamma'^*$ but is not contained in the boundary of $v^*_0$, and exactly one of the edges in $c'$ is dual to an edge of $\gamma$. Moreover, the dual of the edges in $c'$ are exactly those in $E'_{vw}$.

     Recall that $H$ is the timelike plane in $\AdS$ which is orthogonal to the spacelike geodesic $\ell$ (with endpoints $v_{\infty}$, $w_{\infty}$ at infinity) at the middle point $[0,0,0,1]$. Let $v^0_n$ denote the vertex of $u_n(\tilde{P}_n)$ corresponding to $v_0\in V(\Gamma)$ and let $Y_n$ be the half-space in the affine chart $\{x_4=1\}$ delimited by the dual plane ${(v^0_n)}^{\perp}$ and disjoint from $v^0_n$. Set $q_n:=u_n(\tilde{P}_n)\cap Y_n\cap H$. For $n$ large enough, $q_n$ is a convex polygon in $H\cong\bAdS^2$ with all the edges spacelike except for one timelike edge contained in ${(v^0_n)}^{\perp}$. 
    The vertex set of $q_n$ consists of the intersection points of $H$ with the edges of $u_n(\tilde{P}_n)$ in $E_{vw}$ and the endpoints, say $a_n$, $b_n$, of its timelike edge (see e.g. Figure \ref{fig:proper2}). Since the edges of the $u_n(\tilde{P}_n)$ in $E_{vw}$ tend to $\ell$ (which is orthogonal to $H$) in the limit, the dihedral angle of the $u_n(\tilde{P}_n)$ at each edge in $E_{vw}$ tends to equal the exterior angle of $q_n$ at the vertex lying on that edge. Note that the limit of the dual plane $(v^0_n)^{\perp}$ tends to be the timelike plane containing $\ell$ and orthogonal to both $H$ and the limit (as indicated in Figure \ref{fig:half-infinite}) of $u_n(\tilde{P}_n)$, the timelike edge of $q_n$ tends to be orthogonal to its two adjacent spacelike edges and the exterior angle of $q_n$ at $a_n$ (resp. $b_n$) tends to zero as $n\rightarrow\infty$. Observe that $q_n$ converges to the point $[0,0,0,1]$ as $n\rightarrow \infty$ and the Gauss-Bonnet formula \eqref{formula:Gauss-Bonnet} also holds for the polygon $q_n$ by appropriately defining the exterior angle between spacelike and timelike edges (which equals the opposite of the imaginary part of that in \cite{{DJJ}}),
    the sum of the exterior angles at the vertices of $q_n$ is equal to the area of $q_n$, which tends to zero as $n\rightarrow\infty$.  Combined with the assumption that the dihedral angle of the $\tilde{P}_n$ at each edge of $\Gamma\setminus\Gamma'$ (if non-empty) tends to zero, the sum of the dihedral angles of the $\tilde{P}_n$ (note that $u_n$ preserves the marking and dihedral angles) at the edges dual to those in $c'$ tends to zero. This implies that the limit function $\theta_{\infty}\in\cA_{\Gamma'}$ fails to satisfy Condition (iv), which leads to contradiction.

    \begin{figure}
    \begin{subfigure}[b]{0.46\textwidth}
     \centering
\begin{tikzpicture}[scale=0.98]
\draw [thick] (-0.5,1) node (v1) {} ellipse (2 and 2);
\draw  [black][fill=black,opacity=1](1.5,1) ellipse (0.04 and 0.04);
\draw  [black][fill=black,opacity=1](-1.449,2.75) ellipse (0.04 and 0.04);
\draw  [black][fill=black,opacity=1](-1.346,-0.802) ellipse (0.04 and 0.04);
\draw  [black][fill=black,opacity=1](0.348,-0.801) ellipse (0.04 and 0.04);
\draw  [black][fill=black,opacity=1](0.348,2.801) ellipse (0.04 and 0.04);
\draw [white][fill=black,opacity=0.2] (0.348,2.801) .. controls (0.95,2.55) and (1.5,1.75) .. (1.5,1) .. controls (1.447,-0.105) and (0.647,-0.751) .. (0.348,-0.801) .. controls (0.348,2.801) and (0.348,2.801) .. (0.348,2.801);
\draw [white][fill=black,opacity=0.2](-1.449,2.75) .. controls (-2.252,2.35) and (-2.6,1.5) .. (-2.5,0.9) .. controls (-2.5,0.3) and (-2,-0.5) .. (-1.346,-0.802) .. controls (1.5,1) and (1.5,1) .. (1.5,1) .. controls (-1.449,2.75) and (-1.449,2.75) .. (-1.449,2.75);
\draw [dotted][thick] (-0.699,2.7) -- (-0.149,2.7);
\draw [dotted] [thick](-0.7,-0.75) -- (-0.15,-0.75);
\draw[decorate,decoration=brace] (-1.55,3.1) -- (0.449,3.1);
\draw[decorate,decoration={brace,mirror}] (-1.499,-1.15) -- (0.349,-1.15);
\node at (-0.5,3.5) {$\mathcal{V}_v$};
\node at (-0.5,-1.5) {$\mathcal{V}_w$};
\draw (0.348,2.801) .. controls (0.348,-0.801) and (0.348,-0.801) .. (0.348,-0.801);
\draw (-1.449,2.75);
\draw[densely dashed] (-1.449,2.75) .. controls (1.5,1) and (1.5,1) .. (1.5,1);
\draw[densely dashed] (-1.346,-0.802) .. controls (1.5,1) and (1.5,1) .. (1.5,1);
\node at (0.598,2.305) {$e_1$};
\node at (-2.7,1.35) {$e_2$};
\node at (0.899,-0.001) {$f_1$};
\node at (-1.55,0) {$f_2$};
\draw [blue](0.598,0.8) .. controls (0,0.749) and (-2.5,0.699) .. (-2.5,0.9) .. controls (-2.5,1.101) and (-2.5,1.101) .. (-2.5,0.9);
\node at (-1.057,0.498) {\color{blue}$c'$};
\draw  [blue][fill=red,opacity=1](0.598,0.8) ellipse (0.03 and 0.03);
\draw  [blue][fill=red,opacity=1](-0.697,1.152) ellipse (0.03 and 0.03);
\node at (-0.697,1.451) {{\color{blue}$f_2^*$}};
\node at (0.649,1.1) {{\color{blue}$f_1^*$}};
\node at (1.8,0.95) {$v_0$};
\draw [densely dashed][blue](-2.5,0.9) .. controls (-2.5,1.101) and (-1.4,1.152) .. (-0.697,1.152);
\end{tikzpicture}
 \label{fig:case1}
    \end{subfigure}
   \begin{subfigure}[b]{0.46\textwidth}
    \centering
\begin{tikzpicture}
\draw [white][fill=black,opacity=0.2] (0.55,2.7) .. controls (1.15,2.4) and (1.55,1.8) .. (1.5,0.8) .. controls (1.5,0.1) and (0.75,-1.05) .. (-0.65,-1) .. controls (0.55,2.7) and (0.55,2.7) .. (0.55,2.7);
\draw [white][fill=black,opacity=0.2](-1.75,2.55) .. controls (-2.25,2.2) and (-2.55,1.5) .. (-2.5,0.9) .. controls (-2.45,0.2) and (-1.9,-0.9) .. (-0.65,-1) .. controls (0.75,-1.05) and (1.5,0.1) .. (1.5,0.95) .. controls (-1.75,2.55) and (-1.75,2.55) .. (-1.75,2.55);
\draw [thick] (-0.5,1) node (v1) {} ellipse (2 and 2);
\draw  [black][fill=black,opacity=1](1.5,0.95) ellipse (0.04 and 0.04);
\draw  [black][fill=black,opacity=1](-1.75,2.55) ellipse (0.04 and 0.04);
\draw  [black][fill=black,opacity=1](-0.65,-1) ellipse (0.04 and 0.04);
\draw  [black][fill=black,opacity=1](0.55,2.7) ellipse (0.04 and 0.04);
\draw [dotted][thick] (-0.75,2.65) -- (-0.2,2.65);
\draw[decorate,decoration=brace] (-1.85,3.1) -- (0.7,3.1);
\draw[decorate,decoration={brace,mirror}] (-0.8,-1.25) -- (-0.4,-1.25);
\node at (-0.5,3.5) {$\mathcal{V}_v$};
\node at (-0.55,-1.6) {$\mathcal{V}_w$};
\draw (0.55,2.7) .. controls (-0.65,-1) and (-0.65,-1) .. (-0.65,-1);
\draw (-1.75,2.55);
\draw[densely dashed] (-1.75,2.55) .. controls (1.5,0.95) and (1.5,0.95) .. (1.5,0.95);
\draw[densely dashed]   ;
\node at (0.65,2.25) {$e_1$};
\node at (-2.8,1) {$e_2$};
\node at (0.75,0.05) {$f_1$};
\node at (-1.5,0.05) {$f_2$};
\draw [blue](0.45,0.7) .. controls (0,0.7) and (-2.5,0.6) .. (-2.5,0.9) .. controls (-2.5,0.9) and (-2.5,0.9) .. (-2.5,0.9);
\node at (-0.8,0.45) {\color{blue}$c'$};
\draw  [blue][fill=red,opacity=1](0.45,0.7) ellipse (0.03 and 0.03);
\draw  [blue][fill=red,opacity=1](-1.15,1.15) ellipse (0.03 and 0.03);
\node at (-1.25,1.5) {{\color{blue}$f_2^*$}};
\node at (0.45,1.05) {{\color{blue}$f_1^*$}};
\node at (1.85,0.95) {$v_0$};
\draw [densely dashed][blue](-2.5,0.9) .. controls (-2.5,1.15) and (-1.4,1.15) .. (-1.15,1.15);
\end{tikzpicture}
  \label{fig:case2}
\end{subfigure}
    \caption{\small{Two examples of the simple path $c'$ 
    in $\Gamma'^*$ of Claim \ref{clm:one tangent}. Exactly two edges $e_1$, $e_2$ in $E'_{vw}$ are adjacent to the two faces $f_1$, $f_2$ (shown in the shaded regions) containing the vertex $v_0$ of $\Gamma'$. The vertices in $\cV_v$ and $\cV_w$ are respectively consecutive on the equator $\gamma$ (shown in the bold 
    cycle). Exactly one edge of the simple path $c'$ (which contains the dual edges of those in $E'_{vw}$ and has endpoints dual to $f_1$ and $f_2$) is dual to an edge of $\gamma$. $\cV_w$ has at least two vertices (resp. exactly one vertex) in the left (resp. right) picture. 
    Here the dashed edges are on the bottom (delimited by the equator $\gamma$) of $\Sigma_{0,N}$.}}
    \label{fig:simple path}
    \end{figure}

    \begin{figure}
 \centering
   \begin{subfigure}[b]{0.46\textwidth}
\begin{tikzpicture}[scale=1.15]
\draw (-0.1305,0) .. controls (-0.1305,-3.5) and (-0.1305,-3.5) .. (-0.1305,-3.5);
\draw [thick](-0.1305,0) .. controls (-0.5,-0.35) and (-0.5,-0.35)
.. (-0.5,-0.35);
\draw [thick](-0.5,-0.35) .. controls (-0.5,-3.15) and
(-0.5,-3.15) .. (-0.5,-3.15);
\draw [thick](-0.5,-3.15) .. controls (-0.1305,-3.5) and (-0.1305,-3.5)
.. (-0.1305,-3.5);
\draw[thick] (-0.1305,0) .. controls (4.3,-2.2175) and (4.3,-2.2175) ..
(4.3,-2.2175);
\draw[thick] (-0.1305,-3.5) .. controls (4.3,-2.2175) and (4.3,-2.2175) ..
(4.3,-2.2175);
\draw[densely dashed] (-0.5,-0.35) .. controls (4.3,-2.2175) and
(4.3,-2.2175) .. (4.3,-2.2175);
\draw[densely dashed] (-0.5,-3.15) .. controls (4.3,-2.2175) and
(4.3,-2.2175) .. (4.3,-2.2175);

\node at (-0.7565,-0.3) {$v^2_n$};
\node at (-0.7565,-3.137) {$w^2_n$};
\node at (-0.387,-3.7) {$w^1_n$};
\node at (4.55,-2.261) {$v^0_n$};
\draw [blue][thick](1.455,-1.739) .. controls (1.479,-1.761) and (1.479,-1.761)
.. (-0.1305,-1.587);
\draw [line width=0.2mm][densely dashed][blue](-0.522,-1.935) .. controls
(-0.5,-1.935) and (-0.5,-1.935) .. (-0.5,-1.935) .. controls
(1.281,-1.9565) and (1.281,-1.9565) .. (1.281,-1.9565);
\node at (0.1515,-1.761) {\color{blue}$q_n$};
\draw [white, fill=gray,opacity=0.3](1.4565,-0.7825) .. controls (1.4565,-3.0435) and (1.4565,-3.0435) .. (1.4565,-3.0435) .. controls (1.261,-2.8045) and (1.261,-2.8045) .. (1.261,-2.8045) .. controls (1.3045,-1.0435) and (1.3045,-1.0435) .. (1.3045,-1.0435) .. controls (1.4565,-0.7825) and (1.4565,-0.7825) .. (1.4565,-0.7825);
.. (-0.5,-1.913);
\draw[line width=0.2mm] (1.4565,-0.7825) .. controls (1.4565,-3.0435) and (1.4565,-3.0435) .. (1.4565,-3.0435);
\draw [line width=0.2mm,densely dashed](1.4565,-0.8045) .. controls (1.2825,-1.0435) and (1.2825,-1.0435) .. (1.2825,-1.0435) .. controls (1.2825,-2.8045) and (1.2825,-2.8045) .. (1.2825,-2.8045) .. controls (1.2825,-2.8045) and (1.4565,-3.0435) .. (1.4565,-3.0435);
\node at (2.4785,-0.5) {\small{$(v^0_n)^{\perp}\cap u_n(\tilde{P}_n)$}};
\draw[line width=0.2mm,blue] (-0.5,-1.913) .. controls (-0.1305,-1.587) and (-0.1305,-1.587) .. (-0.1305,-1.587);
\draw[line width=0.2mm,red] (1.4565,-1.761) .. controls (1.2825,-1.9565) and (1.2825,-1.9565) .. (1.2825,-1.9565);
\node at (-0.348,0.1955) {$v^1_n$};
\node at (1.9914,-1.603) {$a_n$};
\node at (2.0331,-2.3671) {$b_n$};
\draw (1.7821,-1.5638) .. controls (1.6524,-1.5978) and (1.618,-1.6249) .. (1.497,-1.7441);
\draw (1.3229,-1.951) .. controls (1.5804,-2.0079) and (1.7193,-2.0638) .. (1.871,-2.1893);
\end{tikzpicture}
   \end{subfigure}
   \begin{subfigure}[b]{0.35\textwidth}
 \begin{tikzpicture}[scale=1.06]
\draw (0,0) .. controls (-0.05,-3.75) and (-0.05,-3.75) .. (-0.05,-3.75);
\draw [thick](0,0) .. controls (-0.5,-0.25) and (-0.5,-0.25) .. (-0.5,-0.25);
\draw [thick](-0.5,-0.25) .. controls (-0.05,-3.75) and (-0.05,-3.75) ..
(-0.05,-3.75);
\draw[thick] (0,0) .. controls (4.5,-2) and (4.5,-2) .. (4.5,-2);
\draw[thick] (-0.05,-3.75) .. controls (4.5,-2) and (4.5,-2) .. (4.5,-2);
\draw[densely dashed] (-0.5,-0.25) .. controls (4.5,-2) and (4.5,-2) .. (4.5,-2);
\node at (-0.15,0.3) {$v^1_n$};
\node at (-0.8,-0.2) {$v^2_n$};
\node at (-0.35,-3.85) {$w^1_n$};
\node at (4.8,-2.05) {$v^0_n$};
(-0.18,-2.6);
\draw [white, fill=gray,opacity=0.3](1.3112,-0.8584) .. controls (1.3112,-0.8776) and (1.3112,-0.8776) .. (1.3112,-0.8776) .. controls (1.5,-3.1224) and (1.5,-3.1224) .. (1.5,-3.1224) .. controls (1.5752,-0.6696) and (1.5752,-0.6696) .. (1.5752,-0.6696) .. controls (1.3112,-0.8776) and (1.3112,-0.8776) .. (1.3112,-0.8776);
\draw [line width=0.2mm, blue](1.5576,-1.8304) .. controls (0.0192,-1.6888) and (0.0192,-1.6888) .. (0.0192,-1.6888) .. controls (-0.292,-2) and (-0.292,-2) .. (-0.292,-2);
\draw [densely dashed, blue, line width=0.2mm](-0.3112,-2) .. controls (1.4248,-2) and (1.4248,-2) .. (1.4248,-2);
\draw [line width=0.2mm](1.5472,-0.6888) .. controls (1.5472,-0.6696) and (1.5472,-0.6872) .. (1.5472,-0.6872) .. controls (1.5192,-3.1312) and (1.5192,-3.1312) .. (1.5192,-3.1312);
\draw [line width=0.2mm,densely dashed](1.3112,-0.8776) .. controls (1.5192,-3.1416) and (1.5192,-3.1416) .. (1.5192,-3.1416);
\draw [line width=0.2mm, red](1.5472,-1.8112) .. controls (1.4248,-1.9808) and (1.4248,-1.9808) .. (1.4248,-1.9808);
\node at (0.2168,-1.8776) {\color{blue}$q_n$};
\node at (2.5384,-0.4056) {\small{$(v^0_n)^{\perp}\cap u_n(\tilde{P}_n)$}};
\draw [line width=0.2mm,densely dashed](1.3024,-0.8496) .. controls (1.5664,-0.68) and (1.5664,-0.68) .. (1.5664,-0.68);
\node at (2.0349,-1.3855) {$a_n$};
\node at (2.0766,-2.3671) {$b_n$};
\draw (1.8256,-1.4768) .. controls (1.7394,-1.5543) and (1.6615,-1.6684) .. (1.584,-1.7876);
\draw (1.4534,-1.9945) .. controls (1.6674,-2.0079) and (1.8063,-2.0638) .. (1.9145,-2.1893);
\end{tikzpicture}
\end{subfigure}
   \caption{\small{Two examples of $u_n(\tilde{P}_n)$ in Claim \ref{clm:one tangent} with the strictly hyperideal vertex $v^0_n$ converging to $v''_{\infty}$. The bold lines represent the equator. $\cV_w$ has at least two vertices (resp. exactly one vertex) in the left (resp. right) picture. The convex polygon $q_n$ is contained in the timelike plane $H$ and has exactly one timelike edge (which is contained in $(v^0_n)^{\perp}$).
    }}
   \label{fig:proper2}
\end{figure}
	\begin{Claim}\label{clm:two tangent}
		 If at least two limit edges adjacent to $v_{\infty}$ are tangent to $\partial\AdS$ at $v_{\infty}$, then $\theta_{\infty}$ does not satisfy Condition (i).
	\end{Claim}

   Let $v^l_{\infty}$, $v^r_{\infty}$ be the endpoints (other than $v_{\infty}$) of two limit edges of $\tilde{P}_{\infty}$ tangent to $\partial\AdS$ at $v_{\infty}$. By Statement (4) of Lemma \ref{lm:limit property} and the assumption that $\theta_{\infty}\in \cA_{\Gamma'}$, these two edges are not lightlike. One can check that $v^l_{\infty}$, $v^r_{\infty}$ lie on different sides of $v_{\infty}$ and share the same (non-lightlike) line with $v_{\infty}$. Otherwise, if they lie on the same side, then the edge connecting them lies outside of the closure of $\AdS$. If $v^l_{\infty}$, $v^r_{\infty}$ and $v_{\infty}$ were not colinear, then the face of $\tilde{P}_{\infty}$ containing the two tangent edges would lie in a lightlike plane. 
   There would be a face say $f$ of $\Gamma$ such that the limit of its correspondence $f_n$ in $\tilde{P}_{n}$ is lightlike. This is not possible, combined with Statement (3) of Lemma \ref{lm:limit property} and the assumption on $\theta_{\infty}$.
   Therefore, there are exactly two limit edges adjacent to $v_{\infty}$ and they are contained in an edge of $\tilde{P}_{\infty}$ tangent to $\partial\AdS$ at $v_{\infty}$ with endpoints $v^l_{\infty}$, $v^r_{\infty}$ lying outside of the closure of $\AdS$.

   Similar to the argument in Claim \ref{clm:one tangent}, using the dynamics of $u_n$ with the translation length $L_n\rightarrow+\infty$, the (only) vertex of the $\tilde{P}_n$ converging to $v^l_{\infty}$ (resp. $v^r_{\infty}$) converges to the same point say $z_{\infty}$ in $\RP^3$ at infinity in the affine chart along the direction from $v_{\infty}$ to $v^l_{\infty}$ (resp. $v^r_{\infty}$), noting that $v^l_{\infty}$ and $v^r_{\infty}$ lie on the same line tangent to $\partial\AdS$ at $v_{\infty}$. In particular, we choose $u_n$ with $L_n\rightarrow+\infty$ such that under the action of $u_n$, the vertices of the $\tilde{P}_n$ converging to $v_{\infty}$ keep converging to $v_{\infty}$, the vertices of the $\tilde{P}_n$ converging to $v^l_{\infty}$ or $v^r_{\infty}$ tend to $z_{\infty}$, while all the other vertices (note that the number of the other vertices is at least one, since $\tilde{P}_{\infty}$ has at least one more vertex in addition to $v^l_{\infty}$ and $v^r_{\infty}$ by Lemma \ref{lm:representatives}) of the $\tilde{P}_n$ converge to $w_{\infty}$. As discussed before, this is realizable, observing that all the vertices of the $\tilde{P}_n$ apart from those converging to $v_{\infty}$, $v^l_{\infty}$ and $v^r_{\infty}$ are disjoint from $P_v$. Therefore, the limit of $u_n(\tilde{P}_n)$ is a bi-infinite strip with two infinite edges tangent to $\partial\AdS$ at $v_{\infty}$ and $w_{\infty}$ and sharing a common vertex $z_{\infty}$ at infinity (as indicated in Figure \ref{fig:bi-infinite}). Let $v^l$ (resp. $v^r$) denote the vertex of $\Gamma$ corresponding to the (only) vertex of $\tilde{P}_n$ that converges to $v^l_{\infty}$ (resp. $v^r_{\infty}$). Let $v^l_n$ (resp. $v^r_n$) denote the vertex of $u_n(\tilde{P}_n)$ whose preimage in $\tilde{P}_n$ converges to $v^l_{\infty}$ (resp. $v^r_{\infty}$).

   We 
   denote by $\cV_v$ (resp. $\cV_w$) the set of the vertices of $\Gamma$ whose correspondence in $u_n(\tilde{P}_n)$ converge to $v_{\infty}$ (resp. $w_{\infty}$). Then we have $V(\Gamma)=V(\Gamma')=\cV_v\cup\cV_w\cup\{v^l,v^r\}$. From the hypothesis of case (a) and the above discussion, $\cV_v$ has at least two vertices and $\cV_w$ has at least one vertex. Moreover, by Statement (1) of Lemma \ref{lm:limit property} and the definition of $v^l, v^r$, the vertices in $\cV_v$ and $\cV_w$ are respectively consecutive on the equator $\gamma$ and separated by $v^l$ and $v^r$ (as shown in Figure \ref{fig:claim3}). Let $E_{vw}$ (resp. $E'_{vw}$) denote the set of the edges of $\Gamma$ (resp. $\Gamma'$) with one endpoint in $\mathcal{V}_v$ and the other in $\mathcal{V}_w$. We claim that $E'_{vw}$ (and thus $E_{vw}$) is non-empty (otherwise, $\Gamma'$ is disconnected after removing the vertex $v^l, v^r$ and their adjacent edges). Moreover, the edges in $E_{vw}$ (and thus $E'_{vw}$) are non-equatorial.

   Recall that $H$ denotes the timelike plane in $\AdS$ orthogonal to the spacelike geodesic $\ell$ (with endpoints $v_{\infty}$, $w_{\infty}$ at infinity) at the middle point $[0,0,0,1]$. Let $Z_n$ denote the subset in the chosen affine chart bounded by the two dual planes $(v^l_n)^{\perp}$ and $(v^r_n)^{\perp}$. Set $q_n:=u_n(\tilde{P}_n)\cap Z_n\cap H$. For $n$ large enough, $q_n$ is a convex polygon in $H$ with all the edges spacelike except for two timelike edges contained in $(v^l_n)^{\perp}$ and $(v^r_n)^{\perp}$ respectively.

   The vertex set of $q_n$ consists of the intersection points of $H$ with the edges of $u_n(\tilde{P}_n)$ in $E_{vw}$, two pairs of endpoints of its two disjoint timelike edges and possibly one transversely intersection point (say $s_n$) of $H$ with the edge (if non-empty) connecting $v^l_n$ and $v^r_n$ (see e.g. Figure \ref{fig:proper3}).  Since the edges of the $u_n(\tilde{P}_n)$ in $E_{vw}$ tend to $\ell$ (which is orthogonal to $H$) as $n\rightarrow \infty$, the dihedral angle of the $u_n(\tilde{P}_n)$ at each edge in $E_{vw}$ tends to equal the exterior angle of $q_n$ at the vertex lying on that edge. Note that the dual planes $(v^l_n)^{\perp}$ and $(v^r_n)^{\perp}$ tend to be the same timelike plane which contains $\ell$ and is orthogonal to the intersection of $H$ with the limit (as indicated in Figure \ref{fig:bi-infinite}) of $u_n(\tilde{P}_n)$, each timelike edge of $q_n$ tends to be orthogonal to its two adjacent spacelike edges and therefore, the exterior angle of $q_n$ at each endpoint of its timelike edges tends to zero as $n\rightarrow\infty$. Applying the Gauss-Bonnet formula \eqref{formula:Gauss-Bonnet} again, the sum of the exterior angles at the vertices of $q_n$ is equal to the area of $q_n$, which tends to zero as $n\rightarrow\infty$ (note that $q_n$ converges to the point $[0,0,0,1]$ as $n\rightarrow \infty$). Note that the geodesic line passing through $v^l_n$ and $v^r_n$ tends to be parallel to $H$ as $n\rightarrow\infty$, the exterior angle of $q_n$ at $s_n$ (if exists) tends to zero. Combined with the assumption that the dihedral angle of the $\tilde{P}_n$ at each edge of $\Gamma\setminus\Gamma'$ (if non-empty) tends to zero, the sum of the dihedral angles of the $\tilde{P}_n$ (note that $u_n$ preserves the marking and dihedral angles) at the edges in $E'_{vw}$ tends to zero. Note also that the dihedral angle of $\tilde{P}_n$ at each edge in $E'_{vw}$ is positive (recall that the edges in $E'_{vw}$ are all non-equatorial), the dihedral angle of $\tilde{P}_n$ at each edge in $E'_{vw}$ tends to zero. This implies that the limit function $\theta_{\infty}\in\cA_{\Gamma'}$ fails to satisfy Condition (i), a contradiction.

          \begin{figure}
   \begin{subfigure}[b]{0.46\textwidth}
    \centering
\begin{tikzpicture}
\draw [thick] (-0.5,1) node (v1) {} ellipse (2 and 2);
\draw  [black][fill=black,opacity=1](1.5,0.9) ellipse (0.04 and 0.04);
\draw  [black][fill=black,opacity=1](-1.7,2.6) ellipse (0.04 and 0.04);
\draw  [black][fill=black,opacity=1](-1.7,-0.6) ellipse (0.04 and 0.04);
\draw  [black][fill=black,opacity=1](-1.7,-0.6) ellipse (0.04 and 0.04);
\draw  [black][fill=black,opacity=1](0.7,2.6) ellipse (0.04 and 0.04);
\draw  [black][fill=black,opacity=1](-2.5,0.85) ellipse (0.04 and 0.04);
\draw  [black][fill=black,opacity=1](-0.5,3) ellipse (0.04 and 0.04);
\draw  [black][fill=black,opacity=1](0.7,-0.6) ellipse (0.04 and 0.04);
\draw  [black][fill=black,opacity=1](-0.45,-1) ellipse (0.04 and 0.04);
\draw [dotted][thick] (-0.8,2.75) -- (-0.2,2.75);
\draw [dotted] [thick](-0.75,-0.8) -- (-0.15,-0.8);
\draw[decorate,decoration=brace] (-1.7,3.1) -- (0.75,3.1);
\draw[decorate,decoration={brace,mirror}] (-1.75,-1.15) -- (0.75,-1.15);
\node at (-0.5,3.5) {$\mathcal{V}_v$};
\node at (-0.45,-1.5) {$\mathcal{V}_w$};
\draw (-1.15,2.9);
\node at (-2.85,0.9) {$v^l$};
\node at (1.8,0.85) {$v^r$};
\draw [blue](-0.5,3) .. controls (-0.45,-1) and (-0.45,-1) .. (-0.45,-1);
\draw [blue][densely dashed](-1.7,-0.6) .. controls (-1.7,2.6) and (-1.7,2.6) .. (-1.7,2.6);
\draw [blue][densely dashed](0.7,-0.6) .. controls (0.7,2.6) and (0.7,2.6) .. (0.7,2.6);
\end{tikzpicture}
\end{subfigure}
    \begin{subfigure}[b]{0.46\textwidth}
     \centering
\begin{tikzpicture}
\draw [thick] (-0.5,1) node (v1) {} ellipse (2 and 2);
\draw  [black][fill=black,opacity=1](1.5,0.9) ellipse (0.04 and 0.04);
\draw  [black][fill=black,opacity=1](-1.7,2.6) ellipse (0.04 and 0.04);
\draw  [black][fill=black,opacity=1](-0.55,-1) ellipse (0.04 and 0.04);
\draw  [black][fill=black,opacity=1](-0.55,-1) ellipse (0.04 and 0.04);
\draw  [black][fill=black,opacity=1](0.7,2.6) ellipse (0.04 and 0.04);
\draw  [black][fill=black,opacity=1](-2.5,0.85) ellipse (0.04 and 0.04);
\draw [dotted][thick] (-0.75,2.75) -- (-0.2,2.75);
\draw[decorate,decoration=brace] (-1.7,3.1) -- (0.7,3.1);
\draw[decorate,decoration={brace,mirror}] (-0.75,-1.15) -- (-0.35,-1.15);
\node at (-0.5,3.5) {$\mathcal{V}_v$};
\node at (-0.45,-1.5) {$\mathcal{V}_w$};
\draw (-1.15,2.9);
\node at (-2.85,0.9) {$v^l$};
\node at (1.8,0.85) {$v^r$};
\draw [blue](-1.7,2.6) .. controls (-0.55,-1) and (-0.55,-1) .. (-0.55,-1);
\draw [blue](0.7,2.6) .. controls (-0.55,-1) and (-0.55,-1) .. (-0.55,-1);
\draw [densely dashed](0.7,2.6) .. controls (-2.5,0.85) and (-2.5,0.85) .. (-2.5,0.85) .. controls (1.5,0.9) and (1.5,0.9) .. (1.5,0.9) .. controls (1.5,0.9) and (1.5,0.9) .. (1.5,0.9);
\end{tikzpicture}
    \end{subfigure}
    \caption{\small{Two examples of $E'_{vw}$ in $\Gamma'$ in Claim \ref{clm:two tangent}. 
    In the left picture, $E'_{vw}$ has both top and bottom edges, while 
    in the right picture, $E'_{vw}$ has only top edges. The vertices in $\cV_v$ and $\cV_w$ are respectively consecutive on the equator $\gamma$ (shown in the bold lines), separated by the vertices $v^l, v^r$. Here the dashed edges are on the bottom (delimited by the equator $\gamma$) of $\Sigma_{0,N}$.}}
    \label{fig:claim3}
    \end{figure}

\begin{figure}
 \centering
   \begin{subfigure}[b]{0.45\textwidth}
\begin{tikzpicture}[scale=1.04]
\draw (-1,4.55) .. controls (-4.4467,3.3599) and (-4.4467,3.3599) .. (-4.4467,3.3599);
\draw (-4.4467,3.3599) .. controls (-1,2.4) and (-1,2.4) .. (-1,2.4);
\draw (-1,4.55) .. controls (2.5401,3.3132) and (2.5401,3.3132) .. (2.5401,3.3132);
\draw (2.5401,3.3132) .. controls (-1,2.4) and (-1,2.4) .. (-1,2.4);
\draw[thick][densely dashed] (-1,4.55) .. controls (-1.35,4.1) and (-1.35,4.1) .. (-1.35,4.1);
\draw [thick][densely dashed](-1,4.55) .. controls (-0.65,4.1) and (-0.65,4.1) .. (-0.65,4.1);
\draw [thick](-1,2.4) .. controls (-1.5,2.2) and (-1.5,2.2) .. (-1.5,2.2);
\draw[thick] (-1,2.4) .. controls (-0.5,2.2) and (-0.5,2.2) .. (-0.5,2.2);
\draw [thick](-4.4467,3.3599) .. controls (-1.5,2.2) and (-1.5,2.2) .. (-1.5,2.2);
\draw[thick] (2.5401,3.3132) .. controls (-0.5,2.2) and (-0.5,2.2) .. (-0.5,2.2);
\draw[thick][densely dashed] (-1.35,4.1) .. controls (-4.4467,3.3599) and (-4.4467,3.3599) .. (-4.4467,3.3599);
\draw[thick][densely dashed] (-0.65,4.1) .. controls (2.5401,3.3132) and (2.5401,3.3132) .. (2.5401,3.3132);
\draw[densely dashed] (-1.35,4.1) .. controls (-0.65,4.1) and (-0.65,4.1) .. (-0.65,4.1);
\draw (-1.5,2.2) .. controls (-0.5,2.2) and (-0.5,2.2) .. (-0.5,2.2);
\draw [densely dashed](-1.35,4.1) .. controls (-1.5,2.2) and (-1.5,2.2) .. (-1.5,2.2);
\draw[densely dashed] (-0.65,4.1) .. controls (-0.5,2.2) and (-0.5,2.2) .. (-0.5,2.2);
\draw (-1,4.55) .. controls (-1,2.4) and (-1,2.4) .. (-1,2.4);
\node at (-0.9519,4.8) {$v^1_n$};
\node at (-1.6033,3.8) {$v^2_n$};
\node at (-0.4,3.7533) {$v^3_n$};
\node at (-4.4467,3.0599) {$v^l_n$};
\node at (2.5382,3.0132) {$v^r_n$};
\node at (-1,2.7) {$w^1_n$};
\node at (-1.7519,1.9519) {$w^2_n$};
\node at (-0.3481,1.9538) {$w^3_n$};
\draw [line width=0.2mm,blue](-2.2407,3.4561) .. controls (-1,3.5481) and (-1,3.5481) .. (-1,3.5481) .. controls (0.4784,3.4094) and (0.4784,3.4094) .. (0.4784,3.4094);
\draw [line width=0.2mm,densely dashed][blue](-2.4331,3.2637) .. controls (-1.4066,3.092) and (-1.4066,3.092) .. (-1.4066,3.092) .. controls (-0.5934,3.092) and (-0.5934,3.092) .. (-0.5934,3.092) .. controls (0.6227,3.217) and (0.6227,3.217) .. (0.6227,3.217);
\draw [white, fill=gray,opacity=0.3](-2.4228,3.8456) .. controls (-2.2595,4.1051) and (-2.2595,4.1051) .. (-2.2595,4.1051) .. controls (-2.3076,2.8177) and (-2.3076,2.8177) .. (-2.3076,2.8177) .. controls (-2.4709,2.6253) and (-2.4709,2.6253) .. (-2.4709,2.6253) .. controls (-2.4228,3.8456) and (-2.4228,3.8456) .. (-2.4228,3.8456);
\draw [white, fill=gray,opacity=0.3](0.4329,4.076) .. controls (0.5962,3.8165) and (0.5962,3.8165) .. (0.5962,3.8165) .. controls (0.6443,2.5962) and (0.6443,2.5962) .. (0.6443,2.5962) .. controls (0.481,2.8367) and (0.481,2.8367) .. (0.481,2.8367) .. controls (0.4329,4.076) and (0.4329,4.076) .. (0.4329,4.076);
\node at (-0.7939,3.3132) {\color{blue}$q_n$};
\draw [line width=0.2mm](-2.2595,4.0962) .. controls (-2.3076,2.8076) and (-2.3076,2.8076) .. (-2.3076,2.8076);
\draw[line width=0.2mm,densely dashed] (-2.2595,4.0962) .. controls (-2.2595,4.0962) and (-2.4228,3.8367) .. (-2.4228,3.8367) .. controls (-2.4709,2.6152) and (-2.4709,2.6152) .. (-2.4709,2.6152) .. controls (-2.4709,2.6152) and (-2.4709,2.5962) .. (-2.4709,2.5962) .. controls (-2.4519,2.6152) and (-2.4519,2.6152) .. (-2.4519,2.6152);
\draw [line width=0.2mm] (-2.4709,2.6253) .. controls (-2.3076,2.8076) and (-2.3076,2.8076) .. (-2.3076,2.8076);
\draw [line width=0.2mm](0.4519,4.0481) .. controls (0.5,2.8076) and (0.5,2.8076) .. (0.5,2.8076) .. controls (0.6633,2.5962) and (0.6633,2.5962) .. (0.6633,2.5962);
\draw [line width=0.2mm,densely dashed](0.4519,4.0481) .. controls (0.6152,3.7886) and (0.6152,3.7886) .. (0.6152,3.7886) .. controls (0.6443,2.5962) and (0.6443,2.5962) .. (0.6443,2.5962);
\draw [line width=0.2mm,red](-2.2595,3.4519) .. controls (-2.4519,3.2405) and (-2.4519,3.2405) .. (-2.4519,3.2405);
\draw [line width=0.2mm,red](0.5,3.4038) .. controls (0.6152,3.2405) and (0.6152,3.2405) .. (0.6152,3.2405);
\node at (-2.7886,4.5481) {\small{$(v^l_n)^{\perp}\cap u_n(\tilde{P}_n)$}};
\node at (1.0291,4.5481) {\small{$(v^r_n)^{\perp}\cap u_n(\tilde{P}_n)$}};
\end{tikzpicture}
\label{fig:claim3.1}
   \end{subfigure}
    \begin{subfigure}[b]{0.45\textwidth}
\begin{tikzpicture}[scale=1.05]
\draw [thick] (-1.5,1.5) .. controls (-4.7,2.05) and (-4.7,2.05) .. (-4.7,2.05) .. controls (-2,3.5) and (-2,3.5) .. (-2,3.5) .. controls (-1,3.5) and (-1,3.5) .. (-1,3.5) .. controls (1.75,2.55) and (1.75,2.55) .. (1.75,2.55) .. controls (-1.5,1.5) and (-1.5,1.5) .. (-1.5,1.5);
\draw (-1.5,1.5) .. controls (-2,3.5) and (-2,3.5) .. (-2,3.5);
\draw (-1.5,1.5) .. controls (-1,3.5) and (-1,3.5) .. (-1,3.5);
\draw [white, fill=gray,opacity=0.3](-0.3,3.25) .. controls (-0.4,1.85) and (-0.4,1.85) .. (-0.4,1.85) .. controls (-0.15,2.4) and (-0.15,2.4) .. (-0.15,2.4) .. controls (-0.3,3.25) and (-0.3,3.25) .. (-0.3,3.25);
\draw [white, fill=gray,opacity=0.3](-2.7,3.15) .. controls (-2.85,2.8) and (-2.85,2.8) .. (-2.85,2.8) .. controls (-2.8,2.2) and (-2.8,2.2) .. (-2.8,2.2) .. controls (-2.5,1.65) and (-2.5,1.65) .. (-2.5,1.65) .. controls (-2.7,3.15) and (-2.7,3.15) .. (-2.7,3.15);
\draw [densely dashed](-4.7,2.05) .. controls (1.75,2.55) and (1.75,2.55) .. (1.75,2.55);
\draw[densely dashed] (-4.7,2.05) .. controls (-1,3.5) and (-1,3.5) .. (-1,3.5);
\draw [line width=0.2mm](-2.7,3.15) .. controls (-2.5,1.65) and (-2.5,1.65) .. (-2.5,1.65);
\draw [densely dashed, line width=0.2mm](-2.7,3.15) .. controls (-2.85,2.8) and (-2.85,2.8) .. (-2.85,2.8) .. controls (-2.8,2.2) and (-2.8,2.2) .. (-2.8,2.2) .. controls (-2.5,1.65) and (-2.5,1.65) .. (-2.5,1.65);
\draw [line width=0.2mm](-0.3,3.25) .. controls (-0.4,1.85) and (-0.4,1.85) .. (-0.4,1.85);
\draw [densely dashed,line width=0.2mm](-0.3,3.25) .. controls (-0.15,2.4) and (-0.15,2.4) .. (-0.15,2.4) .. controls (-0.4,1.85) and (-0.4,1.85) .. (-0.4,1.85);
\draw [line width=0.2mm,blue](-2.65,2.7) .. controls (-1.85,2.95) and (-1.85,2.95) .. (-1.85,2.95) .. controls (-1.15,2.95) and (-1.15,2.95) .. (-1.15,2.95) .. controls (-0.35,2.85) and (-0.35,2.85) .. (-0.35,2.85);
\draw [line width=0.2mm,red](-2.65,2.7) .. controls (-2.8,2.4) and (-2.8,2.4) .. (-2.8,2.4);
\draw [line width=0.2mm,red](-0.35,2.85) .. controls (-0.25,2.2) and (-0.25,2.2) .. (-0.25,2.2);
\node at (-4.95,2.05) {$v^l_n$};
\node at (2.05,2.5) {$v^r_n$};
\node at (-2,3.8) {$v^1_n$};
\node at (-1,3.8) {$v^2_n$};
\node at (-1.5,1.2) {$w^1_n$};
\node at (-1.5,2.6) {\color{blue}$q_n$};
\node at (-3.55,3.5) {\small{$(v^l_n)^{\perp}\cap u_n(\tilde{P}_n)$}};
\node at (0.7,3.55) {\small{$(v^r_n)^{\perp}\cap u_n(\tilde{P}_n)$}};
\draw [line width=0.2mm,densely dashed, blue](-2.8,2.4) .. controls (-2.15,2.25) and (-2.15,2.25) .. (-2.15,2.25) .. controls (-0.25,2.2) and (-0.25,2.2) .. (-0.25,2.2);
\node at (-2.0476,2.0476) {$s_n$};
\end{tikzpicture}
\label{fig:claim3.2}
   \end{subfigure}
   \caption{\small{Two examples of $u_n(\tilde{P}_n)$ in Claim \ref{clm:two tangent} with the strictly hyperideal vertices $v^l_n$ and $v^r_n$ converging to $z_{\infty}$. The bold lines represent the equator. In the left picture, $u_n(\tilde{P}_n)$ has no edge connecting $v^l_n$ and $v^r_n$, while in the right picture, $u_n(\tilde{P}_n)$ has an edge connecting $v^l_n$ and $v^r_n$. The convex polygon $q_n$ is contained in the timelike plane $H$ and has exactly two timelike edges (which are contained in $(v^l_n)^{\perp}$ and $(v^r_n)^{\perp}$ respectively).   }}
   \label{fig:proper3}
\end{figure}

   Now we discuss Case (b). Assume that a sequence $(v_n)_{n\in\N}$ of vertices of $\tilde{P}_n$ converge to a point $v_{\infty}$ lying in the interior of an edge (say $e_{\infty}$) of the $\tilde{P}_{\infty}$. It is not hard to see that $e_{\infty}$ is tangent to $\partial\AdS$ at $v_{\infty}$. Moreover, there are two edges of $\tilde{P}_n$ adjacent to $v_n$ converging to the half-edges contained in $e_{\infty}$ lying on the left and right-hand sides of $v_{\infty}$. The main idea for Claim \ref{clm:two tangent} is still valid in this case. Indeed, it suffices to find a non-empty edge set $E'_{vw}$ of $\Gamma'$. Namely, we need to show that the corresponding sets $\mathcal{V}_{v}$ and $\mathcal{V}_{w}$ are non-empty. This is true by Lemma \ref{lm:representatives} and an analogous analysis. Hence, $\theta_{\infty}\in\cA_{\Gamma'}$ does not satisfy Condition (i) for some non-equatorial edge of $\Gamma'$, a contradiction.

	\begin{Step}\label{step:no.2}
	$\tilde{P}_{\infty}$ is non-degenerate and hyperideal.
    \end{Step}

    By assumption, $\theta_{\infty}\in\cA_{\Gamma'}$. In particular, it satisfies Condition (i) and $\Gamma'$ is 3-connected, which implies that $\tilde{P}_{\infty}$ is non-degenerate.
     It remains to show that $\tilde{P}_{\infty}$ is hyperideal. By Lemma \ref{lm:limit property}, each vertex of $\tilde{P}_{\infty}$ lies either outside of the closure of $\AdS$ or on $\partial\AdS$. It suffices to show that each edge of $\tilde{P}_{\infty}$ intersects $\AdS$. Suppose that $\tilde{P}_{\infty}$ has a limit vertex say $v_{\infty}$ at which at least one adjacent edge is tangent to $\partial\AdS$. Similar to the Claim \ref{clm:one tangent} and Claim \ref{clm:two tangent} in Step \ref{step:no.1}, we split the discussion into one of the following two cases:
    \begin{itemize}
    \item Exactly one limit edge adjacent to $v_{\infty}$ is tangent to $\partial\AdS$ at $v_{\infty}$. Note that the main idea for Claim \ref{clm:one tangent} still works (indeed, in this case the vertex set $\mathcal{V}_{v}$ consists of only one vertex, while $\mathcal{V}_{w}$ consists of at least two vertices).
      As a consequence, $\theta_{\infty}\in\cA_{\Gamma'}$ fails to satisfy Condition (iv).
    \item More than one limit edge adjacent to $v_{\infty}$ are tangent to $\partial\AdS$ at $v_{\infty}$. Similarly as the Case (b) above, using an analogous argument of Claim \ref{clm:two tangent}, it follows that         $\theta_{\infty}\in\cA_{\Gamma'}$ does not satisfy Condition (i) for some non-equatorial edge of $\Gamma'$.
    \end{itemize}

    Combining the above two cases, we conclude that no edge of $\tilde{P}_{\infty}$ is tangent to $\partial\AdS$. As a consequence, $\tilde{P}_{\infty}$ is hyperideal. The Step \ref{step:no.2} is done.

	\begin{Step}\label{step:no.3}
		 $\tilde{P}_{\infty}$ has 1-skeleton $\Gamma'$.
	\end{Step}
     Combining Step      \ref{step:no.0} and Step \ref{step:no.2}, no limit edge in $\tilde{P}_{\infty}$ is reduced to a single point, no limit face in $\tilde{P}_{\infty}$ is collapsed to a point or an edge, and no limit face in $\tilde{P}_{\infty}$ is tangent to $\partial\AdS$. Moreover, the assumption ensures that the limit of the dihedral angle of $\tilde{P}_n$ at any edge of $\Gamma'$ is non-zero and preserves the sign (either positive or negative). Therefore, two vertices $v$, $w$ of      $\Gamma'$ are connected by an edge if and only if the corresponding vertices $v_{\infty}$, $w_{\infty}$ of $\tilde{P}_{\infty}$ are connected by a limit edge. Furthermore, an edge of      $\Gamma'$ is contained in the equator $\gamma$ (resp. top, bottom) of $\Sigma_{0,N}$ if and only if the corresponding limit edge in $\tilde{P}_{\infty}$ lies on the equator (resp. the future-directed faces, past-directed faces) of $\tilde{P}_{\infty}$. The Step \ref{step:no.3} is done.

Combining Steps \ref{step:no.0} to \ref{step:no.3}, Lemma \ref{lem:properness of psi} follows.
\end{proof}

\begin{remark}
The ``dynamical'' argument also works in the hyperbolic hyperideal setting (see \cite[Proposition 21]{bao-bonahon}). Using the dynamics of a pure hyperbolic translation (along a complete geodesic in the projective model of $\bH^3$) acting on $\RP^3$, we provide a new perspective to detect the contradictions against the assumptions (i.e. Conditions (1),(3),(4) in Section \ref{ssc:main}) in the argument of the properness of the parameterization map of non-degenerate hyperbolic hyperideal polyhedra (up to isometries) in terms of dihedral angles. Besides, the method also works for the AdS ideal setting (see \cite[Lemma 1.14]{DMS}, in which  the contradiction is against Condition (iii), corresponding exactly to Claim \ref{clm:no tangent} in the above argument.
\end{remark}

\subsection{Properness of $\Phi$} \label{sec:properness_metric}
In this subsection, we show that the map $\Phi: \cP_N\cup\polyg_N\rightarrow \mathcal{T}_{0,N}$ which assigns to each $P\in \cP_N\cup\polyg_N$ the induced metric on $\partial P$ is proper.

\subsubsection{Definitions and notations}\label{subsec:metric}
We first introduce some relevant definitions and notations.

Let $\Sigma_{0,N}$ be the 2-sphere with $N$ marked points $p_1, \cdots, p_N$ removed (where $N\geq 3$). For each complete hyperbolic metric $h$ on $\Sigma_{0,N}$, we assign to $h$ a \emph{signature} $(\epsilon_1, \ldots, \epsilon_N)\in\{0,1\}^N$ with $\epsilon_i=0$ (resp. $\epsilon_i=1$) if $h$ is isometric to a cusp (resp. funnel) at $p_i$. We denote by $\mathcal{T}_{0,N}^{(\epsilon_1, \ldots, \epsilon_N)}$ the subspace of $\cT_{0,N}$ in which the hyperbolic metrics have signature $(\epsilon_1, \ldots, \epsilon_N)$. It is a basic fact that $\cT^{(\epsilon_1, \ldots, \epsilon_N)}_{0,N}$ is a contractible manifold of real dimension $2N-6+\epsilon_1+\ldots+\epsilon_N$.
Note that $\mathcal{T}_{0,N}$ is the disjoint union of the spaces $\mathcal{T}_{0,N}^{(\epsilon_1,\ldots,\epsilon_N)}$ over all $(\epsilon_1, \ldots, \epsilon_N)\in\{0,1\}^N$.

Recall that each marked hyperideal polyhedron in $\cP_N\cup\polyg_N$ has $N$ ($N\geq 3$) vertices, say $v_1$, \ldots, $v_N$, corresponding to the punctures $p_1$, \ldots, $p_N$ of $\Sigma_{0,N}$, respectively. For each hyperideal polyhedron $P\in\cP_N\cup\polyg_N$, we assign to $P$ a \emph{vertex signature} $(\epsilon_1, \ldots, \epsilon_N)\in\{0,1\}^N$ with $\epsilon_i=0$ (resp. $\epsilon_i=1$) if $P$ is ideal (resp. strictly hyperideal) at the vertex $v_i$. We denote by $\mathcal{P}^{(\epsilon_1, \ldots, \epsilon_N)}$ the subspace of $\cP_N\cup\polyg_N$ in which the polyhedra have vertex signature $(\epsilon_1, \ldots, \epsilon_N)$. 
Note that $\cP_N\cup\polyg_N$ is the disjoint union of $\cP^{(\epsilon_1, \ldots, \epsilon_N)}$ over all $(\epsilon_1, \ldots, \epsilon_N)\in \{0, 1\}^N$.

In order to give a way of gluing for the subspaces (indexed by elements of $\{0,1\}^N$) of $\cT_{0,N}$ and $\cP_N\cup\polyg_N$, we define a partial order $\preceq$ for $\{0,1\}^N$ in the following way: we say $(\epsilon'_1, \epsilon'_2, \ldots, \epsilon'_N)\preceq (\epsilon_1, \epsilon_2, \ldots, \epsilon_N)$ if $\epsilon'_i\leq \epsilon_i$ for all $1\leq i\leq n$, and $(\epsilon'_1, \epsilon'_2, \ldots, \epsilon'_N)\prec(\epsilon_1, \epsilon_2, \ldots, \epsilon_N)$ if in addition $\epsilon'_j<\epsilon_j$ for at least one $j\in\{1,2,\ldots,n\}$. For any two $(\epsilon'_1, \epsilon'_2, \ldots, \epsilon'_N)\not=(\epsilon_1, \epsilon_2, \ldots, \epsilon_N)\in\{0,1\}^N$, if $(\epsilon'_1, \epsilon'_2, \ldots, \epsilon'_N)\prec(\epsilon_1, \epsilon_2, \ldots, \epsilon_N)$, there is a natural way to identify $\cT^{(\epsilon'_1, \epsilon'_2, \ldots, \epsilon'_N)}_{0,N}$ (resp. $\cP^{(\epsilon'_1, \epsilon'_2, \ldots, \epsilon'_N)}$) to a subset of the boundary of $\cT^{(\epsilon_1, \epsilon_2, \ldots, \epsilon_N)}_{0,N}$ (resp. $\cP^{(\epsilon_1, \epsilon_2, \ldots, \epsilon_N)}$) via the degeneration of the corresponding funnels to cusps (resp. strictly hyperideal vertices to ideal vertices). If $(\epsilon'_1, \epsilon'_2, \ldots, \epsilon'_N)$ and $(\epsilon_1, \epsilon_2, \ldots, \epsilon_N)$ are incomparable, we have $(\min\{\epsilon_1,\epsilon'_1\},\ldots,\min\{\epsilon_N,\epsilon'_N\})\prec (\epsilon'_1, \epsilon'_2, \ldots, \epsilon'_N)$ and $(\min\{\epsilon_1,\epsilon'_1\},\ldots,\min\{\epsilon_N,\epsilon'_N\})\prec (\epsilon_1, \epsilon_2, \ldots, \epsilon_N)$. Then there is a natural way to \emph{glue} $\cT^{(\epsilon_1,\ldots,\epsilon_N)}_{0,N}$ (resp. $\cP^{(\epsilon_1,\ldots,\epsilon_N)}$) and $\cT^{(\epsilon'_1,\ldots,\epsilon'_N)}_{0,N}$ (resp. $\cP^{(\epsilon'_1,\ldots,\epsilon'_N)}$) along parts of their boundaries, which can be identified with the closure of $\cT_{0,N}^{(\min\{\epsilon_1,\epsilon'_1\},\ldots,\min\{\epsilon_N,\epsilon'_N\})}$ (resp. $\cP^{(\min\{\epsilon_1,\epsilon'_1\},\ldots,\min\{\epsilon_N,\epsilon'_N\})}$)
in $\cT_{0,N}$ (resp. $\cP_N\cup\polyg_N$).

\subsubsection{The proof}
We are now ready to give the proof of the properness of $\Phi$. Let $(P_n)_{n\in \N}$ be a sequence of hyperideal polyhedra in $\cP_N\cup\polyg_N$ with image $h_n:=\Phi(P_n)$ converging to an element say $h_{\infty}$ of $\cT_{0,N}$. Note that $\cP_N\cup\polyg_N$ is a disjoint union of $\cP^{(\epsilon_1,\ldots, \epsilon_N)}$ over all $(\epsilon_1,\ldots, \epsilon_N)\in\{0,1\}^N$ and $\{0,1\}^N$ is finite. Up to extracting a subsequence, we assume that $(P_n)_{n\in \N}$ is a sequence in $\cP^{(\epsilon_1,\ldots, \epsilon_N)}$ for a fixed signature $(\epsilon_1,\ldots, \epsilon_N)\in\{0,1\}^N$. Note also that $\cT_{0,N}$ is a disjoint union of $\cT_{0,N}^{(\epsilon_1,\ldots, \epsilon_N)}$ over all $(\epsilon_1,\ldots, \epsilon_N)\in\{0,1\}^N$. Then $h_{\infty}\in\cT_{0,N}^{(\epsilon'_1,\ldots, \epsilon'_N)}$ for some $(\epsilon'_1,\ldots, \epsilon'_N)\in\{0,1\}^N$. Since $h_n\in\cT_{0,N}^{(\epsilon_1,\ldots, \epsilon_N)}$ for all $n$, by the gluing structure in $\cT_{0,N}$ (see subsection \ref{subsec:metric}), it follows that $(\epsilon'_1,\ldots, \epsilon'_N)\preceq(\epsilon_1,\ldots, \epsilon_N)$, namely $\epsilon'_i\leq \epsilon_i$ for all $1\leq i\leq N$. We aim to show that $(P_n)_{n\in \N}$ converges to an element of $\cP^{(\epsilon'_1,\ldots, \epsilon'_N)}$.

\begin{proof}[\textbf{Proof of Lemma \ref{lem:proper_metrics}}]
Let $(P_n)$ be a sequence of marked polyhedra in $\cP^{(\epsilon_1,...,\epsilon_N)}$ with representatives $\tilde{P}_n$ converging to a limit say $\tilde{P}_{\infty}$ in $\RP^3$, and with the induced metric $h_n:=\Phi(\tilde{P}_n)$ on $\partial\tilde{P}_n\cap\AdS$   converging to a metric $h_{\infty}\in\cT_{0,N}^{(\epsilon'_1,...,\epsilon'_N)}$ with $(\epsilon'_1,...,\epsilon'_N)\preceq(\epsilon_1,\ldots, \epsilon_N)$. 
We need to show that the equivalence class of $\tilde{P}_{\infty}$ lies in $\cP^{(\epsilon'_1,...,\epsilon'_N)}$ (namely, $\tilde{P}_\infty$ is convex hyperideal with vertex signature $(\epsilon'_1,...,\epsilon'_N)$). We prove the lemma using the following steps.

\begin{Step_metrics}
$\tilde{P}_{\infty}$ is convex in $\RP^3$.
\end{Step_metrics}
As discussed before, $\tilde{P}_{\infty}$ is a locally convex (possibly degenerate) polyhedron in $\RP^3$. Suppose by contradiction that $\tilde{P}_{\infty}$ is not convex, then it is projective equivalent to a bi-infinite prism or strip in an affine chart $\R^3\subset\RP^3$, with edges intersect $\AdS$ or tangent to $\partial\AdS$. Note that the limit of the induced metric $h_n$ on $\partial\tilde{P}_n\cap\AdS$ is hyperbolic. The boundary surface $\partial \tilde{P}_{\infty}\cap\AdS$ with the induced metric is isometric to either a complete hyperbolic annulus, $\bH^2$ or to two copies of $\bH^2$ (see Figure \ref{fig:convex} for instance). 
This implies that $h_{\infty}$ is not an element of $\mathcal{T}_{0,N}$ (where $N\geq 3$), which leads to contradiction.

      \begin{figure}
   \begin{subfigure}[b]{0.32\textwidth}
    \centering
\begin{tikzpicture}[scale=0.93]
\draw [white][fill=gray,opacity=0.3](-2.45,-0.95) .. controls (1.4,-0.95) and (1.4,-0.95) .. (1.4,-0.95) .. controls (1.55,-1.1) and (1.5,-1.9) .. (1.4,-2.15) .. controls (-1.05,-2.15) and (-1.05,-2.15) .. (-2.4,-2.15) .. controls (-2.55,-1.8) and (-2.5,-1.05) .. (-2.45,-0.95);
\draw  (-0.5,-1.5) ellipse (2 and 2);
\draw [thick] (-2.95,-0.95) .. controls (1.95,-0.95) and (1.95,-0.95) .. (1.95,-0.95);
\draw [thick] (-2.95,-2.15) .. controls (1.95,-2.15) and (1.95,-2.15) .. (1.9,-2.15);
\end{tikzpicture}
\end{subfigure}
    \begin{subfigure}[b]{0.32\textwidth}
     \centering
\begin{tikzpicture}[scale=0.92]
\draw [white][fill=gray,opacity=0.3](-2.5,-1.45) .. controls (-1.3,-1.45) and (0.8,-1.45) .. (1.45,-1.45) .. controls (1.7,-1.45) and (1.35,-3.5) .. (-0.5,-3.5) .. controls (-1.8,-3.5) and (-1.6,-3.5) .. (-0.5,-3.5) .. controls (-2.1,-3.5) and (-2.55,-1.95) .. (-2.5,-1.45);
\draw  (-0.5,-1.5) ellipse (2 and 2);
\draw [thick] (-2.95,-1.45) .. controls (1.95,-1.45) and (1.95,-1.45) .. (1.95,-1.45);
\draw [thick] (-2.95,-3.5) .. controls (2.1,-3.5) and (2.1,-3.5) .. (2.1,-3.5);
\end{tikzpicture}
    \end{subfigure}
        \begin{subfigure}[b]{0.32\textwidth}
        \centering
\begin{tikzpicture}[scale=0.92]
\draw [white][fill=gray,opacity=0.3](-0.5,0.5) .. controls (-1.3,0.5) and (0.8,0.5) .. (-0.5,0.5) .. controls (2.25,0.3) and (2.1,-3.35) .. (-0.5,-3.5) .. controls (-1.8,-3.5) and (-1.8,-3.5) .. (-0.5,-3.5) .. controls (-3.3,-3.25) and (-3.1,0.4) .. (-0.5,0.5);
\draw  (-0.5,-1.5) ellipse (2 and 2);
\draw [thick] (-2.95,0.5) .. controls (1.95,0.5) and (1.95,0.5) .. (1.95,0.5);
\draw [thick] (-2.95,-3.5) .. controls (2.1,-3.5) and (2.1,-3.5) .. (2.1,-3.5);
\end{tikzpicture}
     \end{subfigure}
    \caption{\small{The candidates for $\tilde{P}_{\infty}$ in the case that it is a bi-infinite strip, viewed as a degenerate polyhedron whose boundary surfaces are identified with the two copies of $\tilde{P}_{\infty}$ glued along the intersection of $\AdS$ with the limit (shown in the bold lines) of the equators of $\tilde{P}_n$. The round disk represents $\bH^2$ and the shaded region is $\partial\tilde{P}_{\infty}\cap\AdS$.}}
    \label{fig:convex}
    \end{figure}

\begin{Step_metrics}
$\tilde{P}_{\infty}$ is hyperideal.
\end{Step_metrics}

It is clear that any vertex of $\tilde{P}_{\infty}$ is disjoint from $\AdS$. Since $h_{\infty}$ is hyperbolic (non-degenerate), each face of $\tilde{P}_{\infty}$ is spacelike. We claim that each edge of $\tilde{P}_{\infty}$ intersects $\partial\AdS$. Suppose by contradiction that there is at least one edge (say $e_{\infty}$) of $\tilde{P}_{\infty}$ which does not intersect $\AdS$, then this edge must be tangent to $\partial \AdS$ at a point, say $q$ (note that $q$ is either an endpoint or an interior point of $e_{\infty}$). 
Combined with Statement (5) of Lemma \ref{lm:limit property}, $e_{\infty}$ lies on the equator of $\tilde{P}_{\infty}$. Note that the edges on the equator of $\tilde{P}_{\infty}$ are not all tangent to $\partial\AdS$ (otherwise, the induced metric on $\partial\tilde{P}_{\infty}\cap\AdS$ is isometric to two copies of $\bH^2$, which leads to contradiction again). 
Let $\alpha_{\infty}$ be the simple closed geodesic on $\partial\tilde{P}_{\infty}\cap\AdS$ separating the vertices of the consecutive edges (including $e_{\infty}$) tangent to $\partial\AdS$  from other vertices on the equator. Then the region on $\partial\tilde{P}_{\infty}\cap\AdS$ bounded by $\alpha_{\infty}$ on the same side as $e_{\infty}$ is isometric to a funnel (see Figure \ref{fig:counter-examples} for instance). This implies that the total number of the cusps and funnels given by $h_{\infty}$ is less than $N$. This contradicts the assumption that $h_{\infty}\in\cT_{0,N}$.

 \begin{figure}
 \centering
   \begin{subfigure}[b]{0.34\textwidth}
    \begin{tikzpicture}
\draw [blue][fill=gray,opacity=0.2](-1.85,-1) .. controls (-1.85,-1) and (-2.5,-0.4) .. (-2.5,0.5) .. controls (-1.1,1.4) and (-1.2,1.3) .. (-1.1,1.4) .. controls (-1.2,1.3) and (-0.45,-1) .. (-0.35,-1) .. controls (-1.85,-1) and (-1.85,-1) .. (-1.85,-1);
\draw[blue][fill=gray,opacity=0.2] (-2.5,0.5) .. controls (-1.1,1.4) and (-1.1,1.4) .. (-1.1,1.4) .. controls (-1,1.45) and (-0.25,-1) .. (-0.35,-1) .. controls (-1.85,-1) and (-1.85,-1) .. (-1.85,-1) .. controls (-1.85,-1) and (-2.5,-0.4) .. (-2.5,0.5);
\draw  (-0.5,0.5) ellipse (2 and 2);
\draw [thick](-2.5,0.5) .. controls (-2.5,-1) and (-2.5,-1) .. (-2.5,-1) .. controls (1.5,-1) and (1.5,-1) .. (1.5,-1) .. controls (1,2.7) and (1,2.7) .. (1,2.7) .. controls (-2.5,0.5) and (-2.5,0.5) .. (-2.5,0.5);
\draw [blue](-1.1,1.4) .. controls (-1.2,1.3) and (-0.45,-1) .. (-0.35,-1);
\draw [blue][densely dashed](-1.1,1.4) .. controls (-1,1.45) and (-0.25,-1) .. (-0.35,-1);
\node at (-0.15,0.25) {$\alpha_{\infty}$};
\node at (-2.75,0.5) {$q$};
\node at (-2.8,-0.45) {$e_{\infty}$};
\end{tikzpicture}
   \end{subfigure}
    \begin{subfigure}[b]{0.32\textwidth}
\begin{tikzpicture}
\draw[densely dashed][blue] [fill=gray,opacity=0.2] (-2,1.8) .. controls (-1.05,2) and (-1.05,2) .. (-0.95,2.05) .. controls (-0.85,2.05) and (-0.45,-0.9) .. (-0.55,-0.9) .. controls (-2,-0.85) and (-2,-0.85) .. (-2,-0.85) .. controls (-2.3,-0.4) and (-2.5,-0.15) .. (-2.5,0.5) .. controls (-2.5,1) and (-2.3,1.5) .. (-2,1.8);
\draw [blue][fill=gray,opacity=0.2](-2,1.8) .. controls (-0.95,2.05) and (-0.95,2.05) .. (-0.95,2.05) .. controls (-1.05,2) and (-0.65,-0.9) .. (-0.55,-0.9) .. controls (-2,-0.85) and (-2,-0.85) .. (-2,-0.85) .. controls (-2.3,-0.4) and (-2.5,-0.15) .. (-2.5,0.5) .. controls (-2.5,1) and (-2.3,1.5) .. (-2,1.8);
\draw  (-0.5,0.5) ellipse (2 and 2);
\draw [thick](-2.5,1.7) .. controls (-2.5,-0.8) and (-2.5,-0.8) .. (-2.5,-0.8) .. controls (1.5,-1) and (1.5,-1) .. (1.5,-1) .. controls (1.15,2.5) and (1.15,2.5) .. (1.15,2.5) .. controls (-2.5,1.7) and (-2.5,1.7) .. (-2.5,1.7);
\draw [blue](-0.95,2.05) .. controls (-1.05,2) and (-0.65,-0.9) .. (-0.55,-0.9);
\draw [blue][densely dashed](-0.95,2.05) .. controls (-0.85,2.05) and (-0.45,-0.9) .. (-0.55,-0.9);
\node at (-0.2,0.65) {$\alpha_{\infty}$};
\node at (-2.75,0.5) {$q$};
\node at (-2.8,-0.35) {$e_{\infty}$};
\end{tikzpicture}
   \end{subfigure}
    \begin{subfigure}[b]{0.31\textwidth}
\begin{tikzpicture}
\draw[densely dashed][fill=gray,opacity=0.2] (1.35,0.65) .. controls (1.3,0.7) and (1.3,0.7) .. (1.35,0.65) .. controls (1.4,0.55) and (-0.4,-0.95) .. (-0.5,-0.9) .. controls (-2,-0.85) and (-2,-0.85) .. (-2,-0.85) .. controls (-2.3,-0.4) and (-2.5,-0.15) .. (-2.5,0.6) .. controls (-2.2,3.35) and (1.6,2.85) .. (1.35,0.65);
\draw [fill=gray,opacity=0.2](1.35,0.65) .. controls (1.35,0.65) and (1.35,0.65) .. (1.35,0.65) .. controls (1.3,0.7) and (-0.6,-0.85) .. (-0.5,-0.9) .. controls (-2,-0.85) and (-2,-0.85) .. (-2,-0.85) .. controls (-2.3,-0.4) and (-2.5,-0.15) .. (-2.5,0.6) .. controls (-2.2,3.35) and (1.6,2.85) .. (1.35,0.65);
\draw  (-0.5,0.5) ellipse (2 and 2);
\draw [thick](-2.5,2.5) .. controls (-2.5,-0.8) and (-2.5,-0.8) .. (-2.5,-0.8) .. controls (1.5,-1) and (1.5,-1) .. (1.5,-1) .. controls (1.2,2.5) and (1.2,2.5) .. (1.2,2.5) .. controls (-2.5,2.5) and (-2.5,2.5) .. (-2.5,2.5);
\draw [blue](1.35,0.65) .. controls (1.3,0.7) and (-0.6,-0.85) .. (-0.5,-0.9);
\draw [blue][densely dashed](1.35,0.65) .. controls (1.4,0.55) and (-0.4,-0.95) .. (-0.5,-0.9);
\node at (0.05,0.2) {$\alpha_{\infty}$};
\node at (-2.75,0.55) {$q$};
\node at (-2.8,1.35) {$e_{\infty}$};
\end{tikzpicture}
     \end{subfigure}
   \caption{\small{Three counter-examples of $\tilde{P}_{\infty}$ with an edge $e_{\infty}$ tangent to $\partial\AdS$ at $q$ for $N=4$. Here $\tilde{P}_{\infty}$ is a degenerate polyhedron whose boundary surface is identified with the two copies of $\tilde{P}_{\infty}$ glued along the intersection of $\AdS$ with the equator (shown in the bold lines). The round disk represents $\bH^2$ and the shaded region is the funnel corresponding to the consecutive edges (including $e_{\infty}$) tangent to $\partial\AdS$.}}
   \label{fig:counter-examples}
\end{figure}

\begin{Step_metrics}
$\tilde{P}_{\infty}$ has $N$ vertices and its vertex signature is $(\epsilon'_1, ..., \epsilon'_N)$.
\end{Step_metrics}

It is easy to check that the signature assigned to a hyperideal AdS polyhedron $P$ at a vertex $v_i$ is the same as the signature assigned to the induced metric on $\partial P\cap\AdS$ (which is a complete hyperbolic metric on $\Sigma_{0,N}$) at the puncture $p_i$ corresponding to $v_i$ for $1\leq i\leq N$, where $N$ is equal to the number of vertices of $P$. Combined with Step 1 and Step 2, $\tilde{P}_{\infty}$ has $N$ vertices and the vertex signature is $(\epsilon'_1, ..., \epsilon'_N)$, equal to the signature of the induced metric $h_{\infty}$.

Combining the above three steps, the lemma follows.
\end{proof}

    \section{Topology}\label{sec:topology}

In this section, we prove Proposition \ref{prop:topology of A}, and Proposition \ref{prop:dimension of P} and Proposition \ref{prop:dimension of bP}, concerning the topology of $\cA_N$, and the dimensions of the spaces $\cP_N$ and $\cP^{(\epsilon_1,\dots,\epsilon_N)}$.

\subsection{The topology of $\cA_N$}\label{sec:topology of A}

Recall that $\cA_N$ is the disjoint union of $\cA_{\Gamma}$ over all $\Gamma\in\Graph(\Sigma_{0,N},\gamma)$, where $\cA_{\Gamma}$ is the set of $\gamma$-admissible functions on $E(\Gamma)$. 
To understand the topology of $\cA_N$, we first consider the topology of the space $\cA_{\Gamma}$ for each $\Gamma\in\Graph(\Sigma_{0,N},\gamma)$.

\begin{claim}\label{clm:topology of A_Gamma}
  For each $\Gamma\in\Graph(\Sigma_{0,N},\gamma)$, $\cA_{\Gamma}$ is   a contractible smooth manifold   of (real) dimension $|E(\Gamma)|$. As a consequence, the maximum (real) dimension $3N-6$ is attained exactly   when $\Gamma$ is the 1-skeleton of a triangulation.
\end{claim}

\begin{proof}
  Note that $\cA_{\Gamma}$ can be described as the space of the solutions of a system of linear equations and inequalities with variables (corresponding to the edges of $\Gamma$) determined by Conditions (i)-(iv) in Definition \ref{def:admissible angles}. By the definition of Conditions (i)-(iv), the solution space is given by the intersection of finitely many linear half-spaces in $\R^{|E(\Gamma)|}$ (whose boundary contains the origin) corresponding to those conditions.

  Moreover, the compatibility among Conditions (i)--(iv) ensures that this intersection is non-empty and indeed has dimension $|E(\Gamma)|$: to find a solution of Conditions (i)--(iv), one can first choose the value of the angles on the equator, and then choose any set of large enough angles on the non-equatorial edges.

  Observe that if $\theta$ is such a solution, then for any $t>0$, $t\theta$ is also a solution. Therefore, $\cA_{\Gamma}$ is a $|E(\Gamma)|$-dimensional convex cone (without containing the cone point zero) in $\R^{|E(\Gamma)|}$. The lemma follows.
\end{proof}

We say that a triangulation of $\Sigma_{0,N}$ (with vertices located at the marked points of $\Sigma_{0,N}$) is \emph{admissible} if it is 3-connected and its edge set contains all the edges of $\gamma$. Note that a triangulation of $\Sigma_{0,N}$ (with vertices located at the marked points of $\Sigma_{0,N}$) is naturally embedded in $\Sigma_{0,N}$ and is thus a planar graph, so an admissible triangulation of $\Sigma_{0,N}$ belongs to $\Graph(\Sigma_{0,N},\gamma)$.
Let $\cG_N$ denote the set of admissible triangulations of $\Sigma_{0,N}$. 
Using an elementary move (the so-called flip) among the triangulations in $\cG_N$ (see e.g. \cite{HN}, \cite[Section 7]{penner:decorated1}), we define a graph of admissible triangulations of $\Sigma_{0,N}$ as follows.

\begin{definition}
 Let $G_T(\Sigma_{0,N})$ denote the graph of admissible triangulations on $\Sigma_{0,N}$, with vertex set $\cG_N$ and an edge connecting two vertices, say $\Gamma_1$ and $\Gamma_2$, if they differ by a single flip: $\Gamma_2$ is obtained from $\Gamma_1$ by replacing one (non-equatorial) edge say $e$ of $\Gamma_1$ by the other diagonal in the quadrilateral which is the union of the two triangles adjacent to $e$.
\end{definition}

We first discuss the connectivity of the graph $G_T(\Sigma_{0,N})$, which is closely related to the connectivity of $\cA_N$. To see this, we need some preliminary lemmas.

For convenience, we divide the non-equatorial edges of a graph $\Gamma\in\Graph(\Sigma_{0,N},\gamma)$ into two classes: \emph{top edges} and \emph{bottom edges}, in which ``top'' (resp. ``bottom'') means that the edges lie on the left (resp. right) side of the equator (considered as an oriented curve) of $\Sigma_{0,N}$. We denote by $\Gamma^t$ (resp. $\Gamma^b$) the restriction of $\Gamma\in\Graph(\Sigma_{0,N},\gamma)$ to the union of the top (resp. bottom) and the equator $\gamma$ of $\Sigma_{0,N}$, called the \emph{top subgraph} (resp. \emph{bottom subgraph}) of $\Gamma$. Note that by our convention a top (or bottom) edge of a graph $\Gamma\in\Graph(\Sigma_{0,N},\gamma)$ is always non-equatorial, while the top and bottom subgraphs $\Gamma^t$, $\Gamma^b$ of a graph $\Gamma\in \Graph(\Sigma_{0,N},\gamma)$ always contain the equator $\gamma$ of $\Gamma$.

In particular, the top (resp. bottom) subgraph of an admissible triangulation $\Gamma\in\cG_N$ is a triangulation of the closure of the top (resp. bottom) of $\Sigma_{0,N}$. Conversely, given a triangulation $T_1$ of the closure of the top of $\Sigma_{0,N}$ and a triangulation $T_2$ of the closure of the bottom of $\Sigma_{0,N}$, the union $T_1\cup T_2$ is naturally a triangulation of $\Sigma_{0,N}$. However, $T_1\cup T_2$ is not necessarily admissible.

The following lemma gives a criterion for determining whether a given triangulation of $\Sigma_{0,N}$ with vertices located at the marked points and the edge set containing all the edges of $\gamma$ is admissible.

\begin{lemma}\label{lem:criterion for admissible}
Let $\Gamma$ be a triangulation of $\Sigma_{0,N}$ ($N\geq 4$) with vertices located at the marked points and the edge set containing all the edges of $\gamma$. Then $\Gamma$ is admissible if and only if there is no pair of vertices (on the equator) that are connected by both a top edge and a bottom edge of $\Gamma$.
\end{lemma}

\begin{proof}
  By assumption and definition, $\Gamma$ is admissible if and only if $\Gamma$ is 3-connected. 
  We first show the necessity. Suppose by contradiction that there are two vertices say $v_1$, $v_2$ that are connected by a top edge and a bottom edge of $\Gamma$. After removing $v_1$, $v_2$ and their adjacent edges, the resulting graph is divided into two components. Then $\Gamma$ is not 3-connected and is thus not admissible. The necessity follows.

 Now we show the sufficiency. Assume that there is no pair of vertices (on the equator) that are connected by both a top edge and a bottom edge of $\Gamma$.
 It suffices to show that $\Gamma$ is 3-connected.
 Let $\Gamma_{v_1v_2}$ denote the graph obtained from $\Gamma$ by removing the vertices $v_1$, $v_2$ of $\Gamma$ and their adjacent edges. We divide the discussion into the following two cases:

 \textbf{Case 1.} If $v_1$, $v_2$ are adjacent along the equator $\gamma$. We claim that $\Gamma_{v_1v_2}$ is connected. Indeed,  noting that $N\geq 4$ and the edge set of $\Gamma$ contains all the edges of $\gamma$,  the restriction (say $\gamma_0$) of $\Gamma_{v_1v_2}$ to the equator  remains non-empty (which has at least two vertices) and connected, so that the top and bottom subgraphs of $\Gamma_{v_1v_2}$ can be connected through $\gamma_0$. The claim follows.

 \textbf{Case 2.} If $v_1$, $v_2$ are not adjacent along the equator $\gamma$, then, again by the hypothesis that $N\geq 4$, the restriction of $\Gamma_{v_1v_2}$ to the equator has two components, say $\gamma_1$ and $\gamma_2$. We claim that there must be a top edge or a bottom edge of $\Gamma_{v_1v_2}$ connecting $\gamma_1$ and $\gamma_2$. This will imply directly that $\Gamma_{v_1v_2}$ is connected. We just need to prove the claim.

 Suppose by contradiction that there is neither top edge nor bottom edge of $\Gamma_{v_1v_2}$ connecting $\gamma_1$ and $\gamma_2$. We are going to show that there must be a top edge and a bottom edge of $\Gamma$ connecting $v_1$ to $v_2$.

 We first show that there is a top edge of $\Gamma$ connecting $v_1$ to $v_2$. Let $v_{i1}$. $v_{i2}$ denote the two vertices on $\gamma_1$, $\gamma_2$ adjacent to $v_i$ for $i=1,2$. Note that $v_{11}$ and $v_{21}$ (resp. $v_{12}$ and $v_{22}$) might coincide. By the assumption above, there is no top edge of $\Gamma_{v_1v_2}$ connecting $\gamma_1$ and $\gamma_2$. Therefore, there is at least one top edge of $\Gamma$ connecting $v_1$ to some vertex of $\gamma_1\sqcup\gamma_2\sqcup \{v_2\}$. 
 Otherwise, since $\Gamma$ is a triangulation, then there is a top edge of $\Gamma$ connecting $v_{11}$ to $v_{12}$ (thus $\gamma_1$ to $\gamma_2$). It is clear that this top edge also belongs to $\Gamma_{v_1v_2}$, which contradicts the assumption.

 If there is no top edge of $\Gamma$ connecting $v_1$ to $\gamma_1\sqcup\gamma_2$, then $v_1$ is connected to $v_2$ by a top edge of $\Gamma$ (see Figure \ref{fig:case 1 for criterion} for instance). Otherwise, there is at least one top edge of $\Gamma$ connecting $v_1$ to $\gamma_1\cup\gamma_2$. Without 
 loss of generality, we assume that there is at least one top edge of $\Gamma$ connecting $v_1$ to $\gamma_1$. Let $w_1$ denote the vertex on $\gamma_1$ which is connected to $v_1$ by a top edge of $\Gamma$ and is closest to $v_2$ along $\gamma_1$ away from $v_1$.

 If $w_1=v_{21}$, note that by assumption $w_1\in\gamma_1$ cannot be connected to any vertex of $\gamma_2$ and also that $\Gamma$ is a triangulation, then $v_1$ is connected to $v_2$ by a top edge of $\Gamma$ (as shown in Figure \ref{fig:case 2 for criterion}).

 Otherwise $w_1\not=v_{21}$. Note that by assumption there is no top edge connecting $\gamma_1$ to $\gamma_2$ and there is no top edge connecting $v_1$ to any vertex (distinct from $w_1$) in $\gamma_1$ between $w_1$ and $v_2$ away from $v_1$. Since $\Gamma$ is a triangulation, then there must be a top edge connecting $v_1$ to $v_2$ (as indicated in Figure \ref{fig:case 3 for criterion}).

 \begin{figure}
 \centering
   \begin{subfigure}[b]{0.34\textwidth}
     \begin{tikzpicture}[scale=0.55]
        \draw (-1.5,2) ellipse (3.5 and 3.5);
\node at (-1.5,5.9) {$v_1$};
\node at (-3.05,5.6) {$v_{11}$};
\node at (0.05,5.6) {$v_{12}$};
\draw [black][fill=black,opacity=1]  (-1.5,5.5) ellipse (0.05 and 0.05);
\draw [black][fill=black,opacity=1]  (-2.95,5.2) ellipse (0.05 and 0.05);
\draw [black][fill=black,opacity=1]  (-0.05,5.2) ellipse (0.05 and 0.05);
\draw [black][fill=black,opacity=1]  (-1.35,-1.5) ellipse (0.05 and 0.05);
\draw [black][fill=black,opacity=1]  (-2.9,-1.2) ellipse (0.05 and 0.05);
\draw [black][fill=black,opacity=1]  (0.2,-1.05) ellipse (0.05 and 0.05);
\node at (-3,-1.6) {$v_{21}$};
\node at (-1.3,-1.9) {$v_2$};
\node at (0.5,-1.4) {$v_{22}$};
\draw (-1.35,-1.5) .. controls (-2.95,5.2) and (-2.95,5.2) .. (-2.95,5.2);
\draw (-0.05,5.2) .. controls (-1.35,-1.5) and (-1.35,-1.5) .. (-1.35,-1.5);
\draw[red] (-1.5,5.5) .. controls (-1.35,-1.5) and (-1.35,-1.5) .. (-1.35,-1.5);
\node  at (-4.4,1.5) {${\gamma_1}$};
\node at (1.4,1.5) {$\gamma_2 $};
\draw[thick] (-2.95,5.2) .. controls (-5.8,3.8) and (-5.6,0) .. (-2.9,-1.2);
\draw[thick] (-0.05,5.2) .. controls (2.85,3.7) and (2.45,0.2) .. (0.2,-1.05);
     \end{tikzpicture}
      \caption{\footnotesize{No top edge of $\Gamma$ connects $v_1$ to $\gamma_1\sqcup\gamma_2$.}}
       \label{fig:case 1 for criterion}
   \end{subfigure}
    \begin{subfigure}[b]{0.32\textwidth}
    \begin{tikzpicture}[scale=0.55]
     \draw (-1.5,2) ellipse (3.5 and 3.5);
\node at (-1.5,5.9) {$v_1$};
\node at (-3,5.6) {$v_{11}$};
\node at (0.05,5.65) {$v_{12}$};
\draw [black][fill=black,opacity=1]  (-1.5,5.5) ellipse (0.05 and 0.05);
\draw [black][fill=black,opacity=1]  (-2.95,5.2) ellipse (0.05 and 0.05);
\draw [black][fill=black,opacity=1]  (-0.05,5.2) ellipse (0.05 and 0.05);
\draw [black][fill=black,opacity=1]  (-1.35,-1.5) ellipse (0.05 and 0.05);
\draw [black][fill=black,opacity=1]  (-2.9,-1.2) ellipse (0.05 and 0.05);
\draw [black][fill=black,opacity=1]  (0.2,-1.05) ellipse (0.05 and 0.05);
\node at (-3,-1.6) {$v_{21}$};
\node at (-1.3,-1.9) {$v_2$};
\node at (0.45,-1.45) {$v_{22}$};
\draw (-1.5,5.5) .. controls (-2.9,-1.2) and (-2.9,-1.2) .. (-2.9,-1.2);
\draw[red] (-1.5,5.5) .. controls (-1.35,-1.5) and (-1.35,-1.5) .. (-1.35,-1.5);
\node at (-4.4,1.5) {$\gamma_1 $};
\node at (1.4,1.5) {$\gamma_2 $};
\draw[thick] (-2.95,5.2) .. controls (-5.8,3.8) and (-5.6,0) .. (-2.9,-1.2);
\draw[thick] (-0.05,5.2) .. controls (2.85,3.7) and (2.45,0.2) .. (0.2,-1.05);
        \end{tikzpicture}
         \caption{\footnotesize{$w_1=v_{21}$.}}
           \label{fig:case 2 for criterion}
   \end{subfigure}
    \begin{subfigure}[b]{0.31\textwidth}
     \begin{tikzpicture}[scale=0.55]
     \draw (-1.5,2) ellipse (3.5 and 3.5);
\node at (-1.5,5.95) {$v_1$};
\node at (-3.05,5.65) {$v_{11}$};
\node at (0,5.6) {$v_{12}$};
\draw [black][fill=black,opacity=1]  (-1.5,5.5) ellipse (0.05 and 0.05);
\draw [black][fill=black,opacity=1]  (-2.95,5.2) ellipse (0.05 and 0.05);
\draw [black][fill=black,opacity=1]  (-0.05,5.2) ellipse (0.05 and 0.05);
\draw [black][fill=black,opacity=1]  (-1.35,-1.5) ellipse (0.05 and 0.05);
\draw [black][fill=black,opacity=1]  (-2.9,-1.2) ellipse (0.05 and 0.05);
\draw [black][fill=black,opacity=1]  (0.2,-1.05) ellipse (0.05 and 0.05);
\node at (-3,-1.6) {$v_{21}$};
\node at (-1.3,-1.9) {$v_2$};
\node at (0.4,-1.5) {$v_{22}$};
\draw[red] (-1.5,5.5) .. controls (-1.35,-1.5) and (-1.35,-1.5) .. (-1.35,-1.5);
\draw [black][fill=black,opacity=1]  (-5,2) ellipse (0.05 and 0.05);
\draw (-1.5,5.5) .. controls (-5,2) and (-5,2) .. (-5,2);
\draw (-5,2) .. controls (-1.35,-1.5) and (-1.35,-1.5) .. (-1.35,-1.5);
\node at (-5.5,2) {$w_1$};
\node at (-5.1,0.05) {$\gamma_1 $};
\node at (2.05,0) {$\gamma_2 $};
\draw[thick] (-2.95,5.2) .. controls (-5.8,3.8) and (-5.6,0) .. (-2.9,-1.2);
\draw[thick] (-0.05,5.2) .. controls (2.85,3.7) and (2.45,0.2) .. (0.2,-1.05);
     \end{tikzpicture}
     \caption{\footnotesize{$w_1\not=v_{21}$}}
       \label{fig:case 3 for criterion}
     \end{subfigure}
   \caption{\small{The examples of the top subgraph of $\Gamma$ in Case 2 of Lemma \ref{lem:criterion for admissible} (partially drawn, with the other vertices and edges hidden), where the circle represents the equator $\gamma$ of $\Sigma_{0,N}$, the bold lines represent the two components $\gamma_1$, $\gamma_2$ of the restriction of $\Gamma_{v_1v_2}$ to the equator.}}
   \label{fig:criterion}
\end{figure}

This concludes that there is a top edge of $\Gamma$ connecting $v_1$ to $v_2$. Applying an analogous argument, we can show that there is also a bottom edge of $\Gamma$ connecting $v_1$ to $v_2$. This contradicts our assumption at the beginning of Case 2. Therefore the sufficiency follows.
\end{proof}

With this criterion, we have the following results, which are crucial to prove the connectedness of $G_T(\Sigma_{0,N})$.

\begin{observation}\label{obs: admissible and flips}
Let $\Gamma\in\cG_N$ be an admissible triangulation of $\Sigma_{0,N}$ and $T$ be a triangulation of the closure of the top of $\Sigma_{0,N}$ with $N\geq 4$. Assume that $T$ and the top subgraph $\Gamma^t$ of $\Gamma$ differ by one flip. Then there is at most one pair of vertices (on the equator) connected by both an edge of $T$ and a bottom edge of $\Gamma$.
\end{observation}

\begin{proof}
 Note that since $\Gamma$ is admissible, by Lemma \ref{lem:criterion for admissible}, no top edge and bottom edge of $\Gamma$ share two common endpoints. Note also that $T$ is obtained from $\Gamma^t$ by a single flip, thus is distinct from $\Gamma^t$ only at the edge say $e'$ which is obtained from $\Gamma^t$ by a flip. Therefore, the endpoints of the edge $e'$ of $T$ are the only possible pair of vertices on the equator that are also shared with a bottom edge of $\Gamma$. The desired result follows.
\end{proof}

\begin{observation}\label{obs:admissible and union}
Let $\Gamma \in\cG_N$ be an admissible triangulation of $\Sigma_{0,N}$ and $T$ be a triangulation of the closure of the top of $\Sigma_{0,N}$ with $N\geq 5$. Assume that $T$ and $\Gamma^t$ differ by one flip, and that there is a pair of vertices (on the equator) connected by both a bottom edge (say $e$) of $\Gamma$ and an edge (say $e'$) of $T$. Denote by $(\Gamma^b)_e$ the triangulation obtained from $\Gamma^b$ by a flip at $e$. Then the triangulations $T\cup(\Gamma^b)_e$ and $\Gamma^t\cup(\Gamma^b)_e$ are both admissible.
\end{observation}

\begin{proof}
 By our assumption and the argument of Observation \ref{obs: admissible and flips}, the endpoints (say $v_1$, $v_2$) of the edges $e$ and $e'$ are the only pair of vertices (on the equator) connected by both a bottom edge of $\Gamma$ and an edge of $T$. Moreover, the edge $e'$ of $T$ is exactly the edge obtained from $\Gamma^t$ by a flip. Let $(e')^f$ denote the edge of $\Gamma^t$ corresponding to $e'$ by a flip and $e^f$ denote the edge of $(\Gamma^b)_e$ corresponding to $e$ by a flip. It is clear that $(e')^f$ and $e^f$ are both non-equatorial.

 Note that since $(e')^f\subset \Gamma^t$ and $e'\subset T$ (resp. $e\subset\Gamma^b$ and $e^f\subset(\Gamma^b)_e$) differ by a flip, they share the same quadrilateral say $q^t$ (resp. $q^b$). It is clear that $q^t\subset\Gamma^t\cap T$ and $q^b\subset\Gamma^b\cap(\Gamma^b)_e$. We claim that the quadrilaterals $q^t$ and $q^b$ have at most one more common vertex other than $v_1$, $v_2$ (see Figure \ref{fig:admissible graph 1} for instance). 
 Otherwise, at least one non-equatorial edge of $q^t\subset\Gamma$ shares common two endpoints with a non-equatorial edge of $q^b\subset\Gamma$, since $N\geq 5$. This will imply that $\Gamma^t\cup\Gamma^b=\Gamma$ is not admissible by Lemma \ref{lem:criterion for admissible}, which leads to contradiction.

 As a consequence, the edge on the top of $\Sigma_{0,N}$ with the same endpoints as $e^f$ must cross the quadrilateral $q^t\subset\Gamma^t\cap T$. This implies that $e^f$ shares no common two endpoints with any edges of $\Gamma^t$ and $T$. Moreover, by Lemma \ref{lem:criterion for admissible} and $\Gamma\in\cG_N$, no non-equatorial edge of $(\Gamma^b)_e\setminus\{e^f\}=\Gamma^b\setminus\{e\}$ shares common two endpoints with any edge of $\Gamma^t$ and $T$. Using Lemma \ref{lem:criterion for admissible} again, it follows that both $\Gamma^t\cup(\Gamma^b)_e$ and $T\cup(\Gamma^b)_e$ are admissible (see Figure \ref{fig:admissible graph 2} and Figure \ref{fig:admissible graph 3} for instance). This concludes the observation.

  \begin{figure}
 \centering
   \begin{subfigure}[b]{0.31\textwidth}
     \begin{tikzpicture}[scale=0.56]
 \draw (-1.5,2) ellipse (3.5 and 3.5);
\draw [black][fill=black,opacity=1]  (-1.55,5.5) ellipse (0.05 and 0.05);
\draw [black][fill=black,opacity=1]  (-4.99,1.96) ellipse (0.05 and 0.05);
\draw [black][fill=black,opacity=1]  (2,1.9) ellipse (0.05 and 0.05);
\draw [black][fill=black,opacity=1]  (-1.5,-1.5) ellipse (0.05 and 0.05);
\node at (-1.4,-1.9) {$v_2$};
\draw[blue] (-5,1.95) .. controls (2,1.9) and (2,1.9) .. (2,1.9);
\draw[thick] (-1.55,5.5) .. controls (2,1.9) and (2,1.9) .. (2,1.9);
\draw[thick](-1.5,-1.5) .. controls (2,1.9) and (2,1.9) .. (2,1.9);
\node at (-1.5,5.95) {$v_1$};
\draw[dashed,thick] (1,-0.4) .. controls (-1.5,-1.5) and (-1.5,-1.5) .. (-1.5,-1.5);
\draw[dashed,thick] (-1.55,5.5) .. controls (1,-0.4) and (1,-0.4) .. (1,-0.4);
\draw[dashed][red] (-1.55,5.5) .. controls (-1.5,-1.5) and (-1.5,-1.5) .. (-1.5,-1.5);
\node at (-4.05,2.2) {$(e')^f$};
\node at (-1.8,0.4) {$e$};
\draw[thick] (-1.55,5.5) .. controls (-6.05,5.3) and (-6.2,-1.25) .. (-1.5,-1.5);
\draw [black][fill=black,opacity=1]  (1.045,-0.4) ellipse (0.05 and 0.05);
     \end{tikzpicture}
      \caption{\footnotesize{$\Gamma=\Gamma^t\cup\Gamma^b$}}
       \label{fig:admissible graph 1}
   \end{subfigure}
    \begin{subfigure}[b]{0.31\textwidth}
    \begin{tikzpicture}[scale=0.56]

\draw (-1.5,2) ellipse (3.5 and 3.5);
\draw [black][fill=black,opacity=1]  (-1.55,5.5) ellipse (0.05 and 0.05);
\draw [black][fill=black,opacity=1]  (-4.99,1.96) ellipse (0.05 and 0.05);
\draw [black][fill=black,opacity=1]  (2,1.9) ellipse (0.05 and 0.05);
\draw [black][fill=black,opacity=1]  (-1.5,-1.5) ellipse (0.05 and 0.05);
\node at (-1.4,-1.9) {$v_2$};
\draw[blue] (-5,1.95) .. controls (2,1.9) and (2,1.9) .. (2,1.9);
\draw[thick] (-1.55,5.5) .. controls (2,1.9) and (2,1.9) .. (2,1.9);
\draw[thick](-1.5,-1.5) .. controls (2,1.9) and (2,1.9) .. (2,1.9);
\node at (-1.5,5.95) {$v_1$};
\draw[dashed,thick] (1,-0.4) .. controls (-1.5,-1.5) and (-1.5,-1.5) .. (-1.5,-1.5);
\draw[dashed,thick] (-1.55,5.5) .. controls (1,-0.4) and (1,-0.4) .. (1,-0.4);
\node at (-4.05,2.2) {$(e')^f$};
\node at (-2.8,0.75) {$e^f$};
\draw[thick] (-1.55,5.5) .. controls (-6.05,5.3) and (-6.2,-1.25) .. (-1.5,-1.5);
\draw[red][dashed] (-5,1.95) .. controls (1,-0.4) and (1,-0.4) .. (1,-0.4);
\draw [black][fill=black,opacity=1]  (1.045,-0.4) ellipse (0.05 and 0.05);
\end{tikzpicture}
\caption{\footnotesize{$\Gamma^t\cap(\Gamma^b)_e$}}
       \label{fig:admissible graph 2}
   \end{subfigure}
    \begin{subfigure}[b]{0.32\textwidth}
     \begin{tikzpicture}[scale=0.56]
 \draw (-1.5,2) ellipse (3.5 and 3.5);
\draw [black][fill=black,opacity=1]  (-1.55,5.5) ellipse (0.05 and 0.05);
\draw [black][fill=black,opacity=1]  (-4.98,1.95) ellipse (0.05 and 0.05);
\draw [black][fill=black,opacity=1]  (2,1.9) ellipse (0.05 and 0.05);
\draw [black][fill=black,opacity=1]  (-1.5,-1.5) ellipse (0.05 and 0.05);
\node at (-1.4,-1.9) {$v_2$};
\draw[thick] (-1.55,5.5) .. controls (2,1.9) and (2,1.9) .. (2,1.9);
\draw[thick](-1.5,-1.5) .. controls (2,1.9) and (2,1.9) .. (2,1.9);
\node at (-1.5,5.95) {$v_1$};
\draw[dashed,thick] (1,-0.4) .. controls (-1.5,-1.5) and (-1.5,-1.5) .. (-1.5,-1.5);
\draw[dashed,thick] (-1.55,5.5) .. controls (1,-0.4) and (1,-0.4) .. (1,-0.4);
\node at (-1.8,3.3) {$e'$};
\node at (-2.8,0.75) {$e^f$};
\draw[thick] (-1.55,5.5) .. controls (-6.05,5.3) and (-6.2,-1.25) .. (-1.5,-1.5);
\draw[red][dashed] (-4.98,1.95) .. controls (1,-0.4) and (1,-0.4) .. (1,-0.4);
\draw [black][fill=black,opacity=1]  (1.045,-0.4) ellipse (0.05 and 0.05);
\draw[blue] (-1.55,5.5) .. controls (-1.5,-1.5) and (-1.5,-1.5) .. (-1.5,-1.5);
     \end{tikzpicture}
     \caption{\footnotesize{$T\cup(\Gamma^b)_e$}}
       \label{fig:admissible graph 3}
     \end{subfigure}
   \caption{\small{The examples of three relevant triangulations of $\Sigma_{0,N}$ in the proof Observation \ref{obs:admissible and union} (partially drawn, with the other vertices and edges hidden), where the circle represents the equator $\gamma$ of $\Sigma_{0,N}$ and the left semicircle has exactly one vertex between $v_1$ and $v_2$, the solid (resp. dashed) lines represent the top (resp. bottom) edges, the region bounded by the left semicircle and bold solid (resp. dashed) lines represents the quadrilateral $q^t$ (resp. $q^b$).}}
   \label{fig:admissible}
\end{figure}
\end{proof}

The following is a known result (see \cite[Theorem 4.2]{HN}) about the (vertex)-connectivity of the graph of triangulations of a convex polygon with $N$ vertices, which will be used to show the connectedness of $G_T(\Sigma_{0,N})$.

\begin{lemma}\label{lem:connectivity of triangulations of polygons}
Let $p_N$ be a convex polygon with $N$ vertices ($N\geq 4$). Then the graph $G_T(p_N)$ of triangulations of $p_N$ is (N-3)-connected. In particular, $G_T(p_N)$ is connected.
\end{lemma}

We are now ready to show the connectedness of the graph $G_T(\Sigma_{0,N})$.

\begin{proposition}\label{connectivity of graph of triangulation}
For $N\geq 5$, the graph $G_T(\Sigma_{0,N})$ is connected.
\end{proposition}

\begin{proof}
 We prove this result using the following three steps:

 \textbf{Step 1.} We show that for any admissible triangulation $\Gamma\in\cG_N$ and any triangulation $T$ of the closure of the top of $\Sigma_{0,N}$, there is an admissible triangulation $\Gamma'\in\cG_N$ whose top subgraph coincides with $T$ and is connected to $\Gamma$ by a finite number of flips.

 Indeed, by Lemma \ref{lem:connectivity of triangulations of polygons}, the top subgraph $\Gamma^t$ of $\Gamma$ can be connected to $T$ by a finite number of flips, say $\Gamma^t:=T_0, T_1, ..., T_n:=T$. Denote $\Gamma_0:=\Gamma$. We first show that $\Gamma_0$ can be connected to an admissible triangulation, say $\Gamma_1$, whose top subgraph is $T_1$.

 If there is no pair of vertices (on the equator of $\Sigma_{0,N}$) connected by both a bottom edge of $\Gamma$ and an edge of $T_1$, then, by Lemma  \ref{lem:criterion for admissible}, the union $T_1\cup \Gamma^b$ is an admissible triangulation. We let $\Gamma_1=T_1\cup\Gamma^b$. Then $\Gamma_0=\Gamma$ is connected to $\Gamma_1$ by an edge of $G_T(\Sigma_{0,N})$.

 Otherwise, by Observation \ref{obs: admissible and flips}, there is exactly one edge say $e$ of $\Gamma^b$ and one edge $e'$ (which is the new one obtained from $\Gamma^t$ by a flip) of $T_1$ connecting the same vertices. Let $(\Gamma^b)_e$ be the triangulation obtained from $\Gamma^b$ by a flip at $e$.  By Observation \ref{obs:admissible and union}, the triangulations $T_1\cup(\Gamma^b)_e$ and $\Gamma^t\cup(\Gamma^b)_e$ are both admissible. We let $\Gamma_1=T_1\cup(\Gamma^b)_e$. Then $\Gamma_0=\Gamma$ is connected to $\Gamma_1$ by a path of $G_T(\Sigma_{0,N})$ through $\Gamma^t\cup(\Gamma^b)_e$.

 Using the same argument, we can show that there is always an admissible triangulation $\Gamma_{i+1}$ (whose top subgraph is $T_{i+1}$) connected to an admissible triangulation $\Gamma_{i}$ (whose top subgraph is $T_{i}$) by a path of $G_T(\Sigma_{0,N})$ for $i=1,...,n-1$. This concludes Step 1.

 \textbf{Step 2.} Let $T_0$ be a special triangulation (called \emph{fan}) of the closure of the top of $\Sigma_{0,N}$, whose non-equatorial edge set is obtained by connecting one vertex (say $v_0$) to all the others non-adjacent to $v_0$ along the equator. We show that for any two admissible triangulations $\Gamma_1, \Gamma_2\in\cG_N$ with $\Gamma_1^t =T_0 =\Gamma_2^t$, there is a path in $G_T(\Sigma_{0,N})$ connecting $\Gamma_1$ to $\Gamma_2$.

 Indeed, let $v_l, v_r$ denote the two vertices adjacent to $v_0$ on the equator of $\Sigma_{0,N}$. By Lemma \ref{lem:criterion for admissible} and the assumption that $\Gamma_1,\Gamma_2$ are both admissible triangulations of $\Sigma_{0,N}$, the bottom subgraphs $\Gamma_1^b$ and $\Gamma_2^b$ of $\Gamma_1$ and $\Gamma_2$ both contain an edge connecting $v_l$ and $v_r$ (see Figure \ref{fig:fan} for example).

 \begin{figure}
      \begin{tikzpicture}[scale=0.55]
 \draw (-1.5,2) ellipse (3.5 and 3.5);
\draw [black][fill=black,opacity=1]  (-1.55,5.5) ellipse (0.05 and 0.05);
\draw [black][fill=black,opacity=1]  (-4.05,-0.4) ellipse (0.05 and 0.05);
\draw [black][fill=black,opacity=1]  (1.95,2.65) ellipse (0.05 and 0.05);
\draw [black][fill=black,opacity=1]  (-1.5,-1.5) ellipse (0.05 and 0.05);
\draw [black][fill=black,opacity=1]  (-3.85,4.6) ellipse (0.05 and 0.05);
\draw [black][fill=black,opacity=1]  (0.9,4.55) ellipse (0.05 and 0.05);
\draw [black][fill=black,opacity=1]  (-4.92,2.75) ellipse (0.05 and 0.05);
\draw [black][fill=black,opacity=1]  (1.1,-0.35) ellipse (0.05 and 0.05);
\draw (-1.55,5.5) .. controls (-1.5,-1.5) and (-1.5,-1.5) .. (-1.5,-1.5);
\draw (-1.55,5.5) .. controls (-4.9,2.75) and (-4.9,2.75) .. (-4.9,2.75);
\draw (-1.55,5.5) .. controls (1.95,2.65) and (1.95,2.65) .. (1.95,2.65);
\node at (-1.5,5.95) {$v_0$};
\draw (-1.55,5.5) .. controls (-4.05,-0.4) and (-4.05,-0.4) .. (-4.05,-0.4);
\draw (-1.55,5.5) .. controls (1.1,-0.35) and (1.1,-0.35) .. (1.1,-0.35);
\draw[dashed] (-3.85,4.6) .. controls (0.9,4.55) and (0.9,4.55) .. (0.9,4.55);
\node at (-4,5) {$v_l$};
\node at (1,5) {$v_r$};
\end{tikzpicture}
 \caption{\small{An example of $\Gamma_1$ (or $\Gamma_2$) in Step 2 of the proof of Proposition \ref{connectivity of graph of triangulation} (with only one bottom edge drawn), whose top subgraph (shown in solid lines) is a fan and bottom graph (shown in dashed lines) contains an edge connecting $v_l$ and $v_r$.}}
       \label{fig:fan}
       \end{figure}

 Let $(\Gamma_1^b)_{v_0}$ (resp. $(\Gamma_2^b)_{v_0}$) denote the graph obtained from $\Gamma_1^b$ (resp. $\Gamma_2^b$) by removing the vertex $v_0$ and its adjacent edges. Then $(\Gamma_1^b)_{v_0}$ and  $(\Gamma_2^b)_{v_0}$ can be identified with two triangulations of a convex polygon with $N-1$ vertices. Note that $N-1\geq 4$. By Lemma \ref{lem:connectivity of triangulations of polygons}, $(\Gamma_1^b)_{v_0}$ and $(\Gamma_2^b)_{v_0}$ differ by a finite number of flips, say $(\Gamma_1^b)_{v_0}:=G_0, G_1, ..., G_m:=(\Gamma_2^b)_{v_0}$. It is clear that the union of $G_i$ and the two equatorial edges adjacent to $v_0$ becomes again a triangulation, say $G'_i$, of the closure the bottom of $\Sigma_{0,N}$ for $i=0,...,m$. Moreover, $G'_i$ and $G'_{i+1}$ differ by a flip for $i=0,...,m-1$. We consider the triangulations $T_0\cup G'_i$ of $\Sigma_{0,N}$ for $i=1,..., m$. Note that the endpoints of each bottom edge of $G'_i$ are distinct from $v_0$ for $i=0,...,m$. Using Lemma \ref{lem:criterion for admissible} again, $T_0\cup G'_i$ is admissible for $i=1,...,m-1$. This gives a path in $G_T(\Sigma_{0,N})$ connecting $\Gamma_1=T_0\cup G'_0$ and $\Gamma_2=T_0\cup G'_m$. Step 2 is complete.

 \textbf{Step 3.} Given any two admissible triangulations $\Gamma_1,\Gamma_2\in\cG_N$, we first connect $\Gamma_1$ (resp. $\Gamma_2$) to an admissible triangulation $\Gamma'_1$ (resp. $\Gamma'_2$) whose top subgraph is a given fan $T_0$ (i.e. $(\Gamma'_1)^t=(\Gamma'_2)^t=T_0$) by a finite number of flips. By Step 1, this is doable. Then we connect $\Gamma'_1$ and $\Gamma'_2$ by a finite number of flips. By Step 2, this is also doable. Therefore, $\Gamma_1$ can be connected to $\Gamma_2$ by a finite number of flips. This concludes the proof.
\end{proof}

Using the connectedness of $G_T(\Sigma_{0,N})$, we prove the following:

\begin{claim}\label{clm:connectedness of A}
For $N\geq 5$, the space $\cA_N$ is connected.
\end{claim}

\begin{proof}
Note that the space $\cA_{\Gamma}$ is non-empty and contractible for all $\Gamma\in\Graph(\Sigma_{0,N},\gamma)$ with $N\geq 4$ (see Claim \ref{clm:topology of A_Gamma}). Assume that $N\geq 5$. We prove the claim by the following three steps:

\textbf{Step 1.} We show that for any two graphs $\Gamma_1, \Gamma_2\in\cG_N$ which differ by a flip, the space $\cA_{\Gamma_1}$ is connected to $\cA_{\Gamma_2}$ in $\cA_N$. 
Note that $\Gamma_1$ and $\Gamma_2$ are both admissible triangulations of $\Sigma_{0,N}$, the graph $\Gamma_1\cap\Gamma_2$ contains the equator $\gamma$. We claim that $\Gamma_1\cap\Gamma_2\in\Graph(\Sigma_{0,N},\gamma)$. It suffices to show that $\Gamma_1\cap\Gamma_2$ is still 3-connected. To see this, we denote by $q$ the quadrilateral shared by both $\Gamma_1$ and $\Gamma_2$ in which an edge $e_1$ of $\Gamma_1$ (as one diagonal of $q$) is flipped to an edge $e_2$ of $\Gamma_2$ (as the other diagonal of $q$). Let $\Gamma_{v_1v_2}$  denote the graph obtained from $\Gamma$ by removing the vertices $v_1$, $v_2$ and their adjacent edges. We just need to show that $(\Gamma_1\cap\Gamma_2)_{v_1v_2}$ is connected for any two vertices $v_1\not=v_2$ on the equator $\gamma$. We divide the discussion into the following two cases:

\begin{enumerate}
 \item If at least one of the vertices $v_1$, $v_2$ is a vertex of $q$, without loss of generality, we assume that $v_1$ is a vertex of $q$. Then $v_1$ is an endpoint of either the edge $e_1$ of $\Gamma_1$ or the edge $e_2$ of $\Gamma_2$. Without loss of generality again, we assume that $v_1$ is an endpoint of $e_1$. In this case, we have that $(\Gamma_1\cap\Gamma_2)_{v_1v_2}=(\Gamma_1)_{v_1v_2}$ (noting that $\Gamma_1\cap\Gamma_2$ is obtained from $\Gamma_1$ by removing the edge $e_1$). Combined with the assumption that $\Gamma_1$ is 3-connected, $(\Gamma_1)_{v_1v_2}=(\Gamma_1\cap\Gamma_2)_{v_1v_2}$ is therefore connected.

\item If neither $v_1$ nor $v_2$ is a vertex of $q$ (note that this case occurs only when $N\geq 6$), $(\Gamma_1\cap\Gamma_2)_{v_1v_2}=(\Gamma_1)_{v_1v_2}\setminus \{e_1\}$. Note that $e_1$ is a diagonal of $q$, one can connect the endpoints of $e_1$ by two consecutive edges of $q$ on the same side of $e_1$. There are two choices for these two edges and both are contained in $(\Gamma_1)_{v_1v_2}\setminus \{e_1\}$. Combined with the assumption that $\Gamma_1$ is 3-connected,  $(\Gamma_1)_{v_1v_2}$ is connected and $(\Gamma_1)_{v_1v_2}\setminus \{e_1\}=(\Gamma_1\cap\Gamma_2)_{v_1v_2}$ is also connected.
\end{enumerate}

This concludes that $\Gamma_1\cap\Gamma_2\in\Graph(\Sigma_{0,N},\gamma)$.
Step 1 follows by the topology of $\cA_N$ (see subsection \ref{subsec:def}), which says that the spaces $\cA_{\Gamma_1}$ and $\cA_{\Gamma_2}$ are glued together along parts of their boundaries, which is identified with the closure of $\cA_{\Gamma_1\cap\Gamma_2}$ in $\cA_N$, as long as $\Gamma_1$, $\Gamma_2$ and $\Gamma_1\cap\Gamma_2$ belong to $\Graph(\Sigma_{0,N},\gamma)$.

\textbf{Step 2.} We show that for any two graphs in $\cG_N$, one can be obtained from the other by a finite number of flips. This follows directly from Proposition \ref{connectivity of graph of triangulation} that any two graphs $\Gamma_1,\Gamma_2\in\cG_N$ can be connected by a path in $G_T(\Sigma_{0,N})$, of which every two adjacent graphs differ by a flip.

\textbf{Step 3.} It remains to show that for any graph $\Gamma\in \Graph(\Sigma_{0,N},\gamma)\setminus \cG_N$, $\cA_\Gamma$ is connected to $\cA_{\Gamma_0}$ for some triangulation $\Gamma_0\in\cG_N$ of which $\Gamma$ is a subgraph (note that such a triangulation $\Gamma_0$ always exists by triangulating the non-triangular faces of $\Gamma$ in a way such that no top edge and bottom edge share common two endpoints). Similar to Step 1, the topology of $\cA_N$ states that $\cA_{\Gamma}$ is identified to a subset of $\partial\cA_{\Gamma_{0}}$. This concludes the proof of Step 3.

Combining the above three steps and Claim \ref{clm:topology of A_Gamma}, this claim follows.
\end{proof}

\begin{proof}[\textbf{Proof of Proposition \ref{prop:topology of A}}]
Let $\theta\in\cA_N$. Then $\theta\in\cA_{\Gamma}$ for some $\Gamma\in\Graph(\Sigma_{0,N},\gamma)$. Note that $\Gamma$ is a subgraph of some admissible triangulation $\Gamma_0\in\cG_N$. In the case that $\Gamma=\Gamma_0$, $\cA_{\Gamma}$ is equal to $\cA_{\Gamma_0}$. In the case that $\Gamma$ is a proper subgraph of $\Gamma_0$, $\cA_{\Gamma}$ is a subset of $\partial\cA_{\Gamma_0}$ (by the topology of $\cA_N$, see subsection \ref{subsec:def}). By Claim \ref{clm:topology of A_Gamma}, $\cA_{\Gamma_0}$ is a (real) smooth manifold which achieves the highest dimension $3N-6$ among all the subspaces of $\cA_N$ in the form of $\cA_{\Gamma'}$ with $\Gamma'\in\Graph(\Sigma_{0,N},\gamma)$. Therefore, $\theta$ has a neighbourhood in $\cA_N$ of (real) dimension $3N-6$. This concludes that $\cA_N$ has real dimension $3N-6$. Combined with Claim \ref{clm:connectedness of A}, Proposition \ref{prop:topology of A} follows.
\end{proof}

\subsection{The dimensions of $\cP_N$ and $\cP^{(\epsilon_1,\ldots,\epsilon)}$}
\label{sec:dimensions of P}
Recall that $\cP_N$ is the disjoint union of $\cP_{\Gamma}$ over all $\Gamma\in\Graph(\Sigma_{0,N},\gamma)$, where $\cP_{\Gamma}$ is the the space of marked non-degenerate hyperideal AdS polyhedra with $N$ vertices and 1-skeleton $\Gamma$ (up to isometries in $\Isom_0(\AdS)$). Recall that $\cP^{(\epsilon_1,\ldots,\epsilon_N)}$ is the space of marked non-degenerate and degenerate hyperideal AdS polyhedra with $N$ vertices and vertex signature $(\epsilon_1,\ldots,\epsilon_N)$, up to isometries in $\Isom_0(\AdS)$. 
We first prove the following.

\begin{claim}\label{clm:dimension of P_Gamma}
For each $\Gamma\in\Graph(\Sigma_{0,N},\gamma)$, the space $\cP_{\Gamma}$ is a smooth manifoldof (real) dimension $|E(\Gamma)|$. In particular, for each $\Gamma\in\cG_N$, the space $\cP_{\Gamma}$ realizes the highest (real) dimension $3N-6$.
\end{claim}

\begin{proof}
Let $P\in\cP_{\Gamma}$. Set $E:=E(\Gamma)$ and $V=V(\Gamma)$. As before, we fix a representative (denoted by $\tilde{P}$) for $P$ and fix an affine chart $\R^3$.  Let $\tilde{\cP}_{\Gamma}$ denote the subset of $\tilde{\cP}_N$ whose polyhedra have 1-skeleton $\Gamma$.  We consider the map $\phi: \tilde{\cP}_{\Gamma}\rightarrow(\R^3)^V$ which assigns to an element in $\tilde{\cP}_{\Gamma}$ its vertex set in $\R^3$.

We first consider the dimension of $\tilde{\cP}_{\Gamma}$. A neighbourhood $U_{\tilde{P}}$ of $\tilde{P}$ in $\tilde{\cP}_{\Gamma}$ is then identified with the intersection of $\phi(\tilde{\cP}_{\Gamma})$ with a neighbourhood, say $U_{\phi(\tilde{P})}$, in $(\R^3)^V$ of $\phi(\tilde{P})$.

If $\Gamma\in\cG_N$, we claim that $\dim\tilde{\cP}_{\Gamma}=3N=|E|+6$. Indeed, by the convexity of $\tilde{P}$, any two faces of $\tilde{P}$ does not lie on the same plane. Note that each face of $\tilde{P}$ has three vertices and they uniquely determine a plane. For each vertex $v_i\in\R^3$ of $\tilde{P}$, we can take a sufficiently small neighbourhood $U_i$ of $v_i$ in $\R^3$ such that for any vertex $v_i^t\in U_i\setminus \AdS$, the polyhedron with vertices $v^t_1$, $v^t_2$, \dots, $v^t_N$ is hyperideal and has the same combinatorics $\Gamma$ as $\tilde{P}$. Note that the number of the edges of any triangulation in $\cG_N$ is $3N-6$. The claim follows.

If $\Gamma\not\in\cG_N$, we can add some new edges (collected by the set $E'$) to $\Gamma$ (without adding new vertices) so that the new 1-skeleton (say $\Gamma'$) obtained by taking the union of the edges in $E\cup E'$ belongs to $\cG_N$. Assume $\tilde{P}$ has $m$ non-triangular faces. Each non-triangular face $\tilde{f}_i$ of $\tilde{P}$ has $k_i\geq 4$ vertices. Note that $\tilde{f}_i$ inherits a triangulation from $\Gamma'$ and is decomposed into $k_i-2$ triangles, denoted by $\Delta^i_1$, $\Delta^i_2$, \dots, $\Delta^i_{k_i-2}$. Let $A^i_1$, $B^i_1$, $C^i_1$ denote the three vertices of $\Delta^i_1$. Let $D^i_2$, \dots, $D^i_{k_i-2}$ denote the remaining vertices of $\tilde{f}_i$ such that $D^i_j\in\Delta^i_j$. By assumption, for each $2\leq j\leq k_i-2$, $\Delta^i_1$ and $\Delta^i_j$ lie on the same plane, which means that the vectors $\overrightarrow{A^i_1B^i_1}$, $\overrightarrow{A^i_1C^i_1}$ and $\overrightarrow{A^i_1D^i_j}$ are linearly dependent for $j=2,\dots, k_i-2$. For each $1\leq i\leq m$, this is equivalent to the following condition holds for $j=2,\dots, k_i-2$ :
\begin{equation}\label{eq:planarity}
 \left|
   \begin{array}[h]{cccc}
      b^i_{1,1}- a^i_{1,1} &  b^i_{1,2}- a^i_{1,2} & b^i_{1,3}- a^i_{1,3} \\
       c^i_{1,1}- a^i_{1,1} &  c^i_{1,2}- a^i_{1,2} & c^i_{1,3}- a^i_{1,3}  \\
     d^i_{j,1}- a^i_{1,1} &  d^i_{j,2}- a^i_{1,2} & d^i_{j,3}- a^i_{1,3}  \\
   \end{array}
 \right|=0~,
 \end{equation}
with $A^i_1=(a^i_{1,1},a^i_{1,2},a^i_{1,3})$,  $B^i_1=(b^i_{1,1},b^i_{1,2},b^i_{1,3})$, $C^i_1=(c^i_{1,1},c^i_{1,2},c^i_{1,3})$,
$D^i_j=(d^i_{j,1},d^i_{j,2},d^i_{j,3})$. For each $1\leq i\leq m$, the planarity condition for the face $\tilde{f}_i$ of $\tilde{P}$ gives a system of $k_i-3$ equations with $3k_i$ variables. Moreover, in the $j$-th equation, the variables are exactly the coordinates in $\R^3$ of the vertices $A^i_1$, $B^i_1$ and $C^i_1$ and $D^i_j$. As a consequence, the planarity condition for non-triangular face $\tilde{f}_i$ over $i=1,\dots, m$ gives a system of equations with the total number
$$\sum\limits^m\limits_{i=1}(k_i-3)=|E'|~.$$
For each vertex $v_i\in\R^3$ of $\tilde{P}$, we can take a sufficiently small neighbourhood $U_i$ of $v_i$ in $\R^3$ such that for any vertex $v_i^t\in U_i\setminus \AdS$ satisfying the above planarity conditions, the polyhedron with vertices $v^t_1$, $v^t_2$, \dots, $v^t_N$ is hyperideal and has the same combinatorics $\Gamma$ as $\tilde{P}$.  Note that these $|E'|$ conditions are independent, this deformation of $\tilde{P}$ has in total $3N-|E'|=3N-(3N-6-|E|)=|E|+6$ degrees of freedom. Those equations are linearly independent -- this is a purely projective statement, equivalent to the fact that the realization space of 3-dimensional polytope has the expected dimension, and is attributed to Legendre \cite[Note VIII]{legendre} and Steinitz \cite[\S 69]{steinitz-rademacher}.

As a consequence, for each $\tilde{P}\in\tilde{\cP}_{\Gamma}$, a small neighbourhood in $\tilde{\cP}_{\Gamma}$ of $\tilde{P}$ is identified with the set of the points in a small neighbourhood $U_{\phi(\tilde{P})}$ in $(\R^3)^V$ of $\phi(\tilde{P})$ whose coordinates satisfy the above planarity conditions: a system of $3N-6-|E|$ equations (note that it is empty if $\Gamma\in\cG_N$), which is exactly the intersection of $U_{\phi(\tilde{P})}$ with an $(|E|+6)$-dimensional subspace in $(\R^3)^V$. Note that the planarity conditions are the same for the polyhedra with 1-skeleton $\Gamma$, the transition functions are smooth, inherits from the smooth structure of $(\R^3)^{V}$ (here we define the maps from subsets of $(\R^3)^{V}$ to be smooth if they locally admit extensions to smooth functions defined on open domains). As a result, $\tilde{\cP}_{\Gamma}$ is a smooth manifold of dimension $|E|+6$.

Recall that $\cP_{\Gamma}=\tilde{\cP}_{\Gamma}/\Isom_0(\AdS)$. The polyhedra are identified up to elements of $\Isom_0(\AdS)=\rm{PSL}_2(\R)\times\rm{PSL}_2(\R)$ (which has dimension 6). 
Combined with the above result that $\tilde{\cP}_{\Gamma}$ is a smooth manifold of dimension $|E|+6$ and the fact that Lie group action of $\Isom_0(\AdS)$ on $\tilde{\cP}_{\Gamma}$ is smooth, free and proper, it follows from the quotient manifold theorem that $\cP_{\Gamma}$ has dimension $|E|+6-6=|E|$ and moreover, it has a unique smooth structure such that the natural projection $\pi:\tilde{\cP}_{\Gamma}\rightarrow\tilde{\cP}_{\Gamma}/\Isom_0(\AdS)=\cP_{\Gamma}$ is a submersion.
\end{proof}

\begin{proof}[\textbf{Proof of Proposition \ref{prop:dimension of P}}] 
  Let $P\in\cP_N$. Then $P\in\cP_{\Gamma}$ for some $\Gamma\in\Graph(\Sigma_{0,N},\gamma)$. Note that $\Gamma$ is a subgraph of some admissible triangulation $\Gamma_0\in\cG_N$. In the case that $\Gamma=\Gamma_0$, $\cP_{\Gamma}$ is equal to $\cP_{\Gamma_0}$. In the case that $\Gamma$ is a proper subgraph of $\Gamma_0$, $\cP_{\Gamma}$ is a subset of $\partial\cP_{\Gamma_0}$ (by the topology of $\cP_N$, see subsection \ref{subsec:def}). By Claim \ref{clm:dimension of P_Gamma}, $\cP_{\Gamma_0}$ is a (real)   smooth manifold  which achieves the    highest dimension $3N-6$ among all the subspaces of $\cP_N$ in the form of $\cP_{\Gamma'}$ for some $\Gamma'\in\Graph(\Sigma_{0,N},\gamma)$. Therefore, $P$ has a neighbourhood in $\cP_N$ of (real) dimension $3N-6$, in which every polyhedron $P'$ is a smooth deformation of $P$ (determined by a tangent vector of $\phi(\tilde{\cP}_N)\subset(\R^3)^{V}$ at $\phi(\tilde{P})$ for a representative $\tilde{P}$ of $P$) with its 1-skeleton $\Gamma'\supseteq\Gamma$ and some planarity conditions (as discussed in the proof of Claim \ref{clm:dimension of P_Gamma}) removed on the faces of $\Gamma$ which contain edges of $\Gamma'\setminus\Gamma$ (if non-empty). This concludes Proposition \ref{prop:dimension of P} that $\cP_N$ is a smooth manifold of real dimension $3N-6$ for $N\geq 4$.
\end{proof}

\begin{proof}[\textbf{Proof of Proposition \ref{prop:dimension of bP}}]
Let $\tilde{\cP}^{(\epsilon_1,\ldots,\epsilon_N)}$ denote the subspace of $\tilde{\cP}_N\cup\widetilde{\polyg}_N$ whose polyhedra have vertex signature $(\epsilon_1,\ldots,\epsilon_N)$. We fix an affine chart $\R^3$ and identify each marked vertex $v_i$ (with signature $\epsilon_i$) of a polyhedron $\tilde{P}\in\tilde{\cP}^{(\epsilon_1,\ldots,\epsilon_N)}$ with its coordinate in $\R^3$.

We claim that $\dim \tilde{\cP}^{(\epsilon_1,\ldots,\epsilon_N)}=2N+\epsilon_1+\ldots+\epsilon_N$. Indeed, 
for each $\tilde{P}\in\tilde{\cP}^{(\epsilon_1,\ldots,\epsilon_N)}$ and for each vertex $v_i\in \R^3$ of $\tilde{P}$, we can take a sufficiently small neighbourhood $U_i$ of $v_i$ in $\R^3$ such that for any vertex $v_i^t\in U_i\setminus\overline{\AdS}$ (resp. $v_i^t\in U_i\cap\partial\AdS$) if $\epsilon_i=1$ (resp. $\epsilon_i=0$), the convex hull of $v_1^t, \ldots, v_N^t$ is a hyperideal AdS polyhedron $\tilde{P}_t$ with vertex signature $(\epsilon_1,\ldots,\epsilon_N)$ and is still an element of $\tilde{\cP}^{(\epsilon_1,\ldots,\epsilon_N)}$. Therefore, each ideal vertex of $\tilde{P}$ has two degrees of freedom and each strictly hyperideal vertex of $\tilde{P}$ has three degrees of freedom (which are independent of each other) during the deformation in $\tilde{\cP}^{(\epsilon_1,\ldots,\epsilon_N)}$. The claim follows. Moreover, $\tilde{\cP}^{(\epsilon_1,\ldots,\epsilon_N)}$ is a smooth manifold, as discussed for $\tilde{\cP}_{\Gamma}$ in the proof of Claim \ref{clm:dimension of P_Gamma}.

Since $\cP^{(\epsilon_1,\ldots,\epsilon_N)}=\tilde{\cP}^{(\epsilon_1,\ldots,\epsilon_N)}/\Isom_0(\AdS)$, $\cP^{(\epsilon_1,\ldots,\epsilon_N)}$ is a smooth manifold of real dimension $2N-6+\epsilon_1+\ldots+\epsilon_N$ (see the corresponding proof for $\cP_{\Gamma}$ in Claim \ref{clm:dimension of P_Gamma}). We conclude Proposition \ref{prop:dimension of bP}.
\end{proof}

    \section{Rigidity}\label{sec:rigidity}

In this section we prove Lemma \ref{lem:local immersion_angles} and Lemma \ref{lem:local immersion_metrics}. 
We then proceed to give the proofs of Theorem \ref{thm:homeo_angles} and Theorem \ref{thm:homeo_metrics}.

\subsection{The infinitesimal Pogorelov map}

We first recall the definition of the infinitesimal Pogorelov map and its key properties (see \cite[Definition 5.6, Proposition 5.7]{shu} and \cite[Section 3.3]{fillastre3} in particular for the related proofs, and also \cite{izmestiev:projective,iie,cpt} for relevant references). As an adaption of the infinitesimal version of a remarkable map introduced by Pogorelov \cite{Po} to solve rigidity questions in spaces of constant curvature, this turns out to be an important tool that translates infinitesimal rigidity questions for polyhedra (or submanifolds) in constant curvature pseudo-Riemannian space-forms to those in flat spaces (see e.g. \cite{DMS,weakly}).

Choose an affine chart $x_4=1$ and denote by $H_{\infty}$ the projective plane in $\HS$ which contains the totally geodesic space-like hyperplane in $\AdS$ at infinity (with respect to the chosen affine chart). Then the dual of $H_{\infty}$, defined as $x_0:=H_{\infty}^{\perp}=\{[y]\in\RP^3: \langle y,x\rangle_{2,2}=0$ for all $x\in H_{\infty}\}$ is contained in $\AdS$ and it is exactly the origin in this affine chart $\R^3$. Let $\R^{2,1}$ denote the 3-dimensional Minkowski space, which is the vector space $\R^3$ endowed with the metric induced from the bilinear form $\langle x, y\rangle_{2,1}=x_1y_1+x_2y_2-x_3y_3$. Recall that an affine chart $\R^3$ can be equipped with the Euclidean metric and the Minkowski metric.

Let $C(x_0)$ denote the union of all light-like geodesic rays in $\AdS$ starting from $x_0$. We call it the \emph{light cone} at $x_0$. Let $U=\HS\setminus H_{\infty}$. 
Then $U$ is the intersection of $\HS$ with the aforementioned affine chart $\R^3$ of $\RP^3$. Let $\iota: U\rightarrow \R^{2,1}$ be an inclusion of $U$ into $\R^{2,1}$, which sends $x_0$ to the origin $0$ of $\R^{2,1}$.
It is clear that $\iota$ is an isometry at the tangent space to $x_0$ and it sends $C(x_0)$ to the light cone at $0$ of $\mathbb{R}^{2,1}$.
For any $x\in U\setminus C(x_0)$ and any vector $v\in T_{x}U$, write $v=v_r+v_{\perp}$, where $v_r$ is tangent to the radial geodesic passing through $x_0$ and $x$, and $v_{\perp}$ is orthogonal to this radial geodesic. Let $\Upsilon: T(U\setminus C(x_0))\rightarrow T\R^{2,1}$ be the bundle map over the inclusion $\iota$, which is defined in the following way:
$$ \Upsilon(v)=d\iota(v) $$
for all $v\in T_{x_0}U$, and
\begin{equation}\label{eq:def_infPogorelov}
\Upsilon(v)=\sqrt{\frac{\langle\hat{x},\hat{x}\rangle_{HS}}{\langle d\iota (\hat{x}),d\iota(\hat{x})\rangle_{2,1}}}d\iota (v_r)+ d\iota (v_{\perp})~,
\end{equation}
for all $v\in T_{x}U$ with $x\in U\setminus C(x_0)$, where $\hat{x}$ is a non-zero radial vector,
 and  $\langle \cdot, \cdot\rangle_{HS}$ is the HS scalar product.
 It is worth mentioning that $\iota$ sends a radial geodesic of $U$ (passing through $x_0$) to a radial geodesic in $\R^{2,1}$
(passing through the origin $0$) of the same type (space-like, time-like and light-like), with the length measure along the (non-lightlike) geodesic changed.
Moreover, each radial geodesic passing through $x\in U\setminus C(x_0)$ is either spacelike or timelike.
Therefore, the quantity under the square-root in \eqref{eq:def_infPogorelov}  is always positive.

By an adaption of the proof in \cite[Lemma 11]{fillastre3}, we obtain the following property of the bundle map $\Upsilon$.

\begin{lemma}\label{lm:Killing HS-Min}
Let $Z$ be a vector field on $U\setminus C(x_0)$. Then $Z$ is a Killing vector field (for the HS metric) if and only if $\Upsilon(Z)$ is a Killing vector field for the Minkowski metric on $\R^{2,1}$.
\end{lemma}

Indeed, it follows from this lemma that the bundle map $\Upsilon$, which so far is defined over $U\setminus C(x_0)$, has a continuous extension for all of $U$. We call this extended bundle map, still denoted by $\Upsilon$, an \emph{infinitesimal Pogorelov map}.

Now we translate the infinitesimal rigidity questions in the Minkowski 3-space $\R^{2,1}$ to those in the Euclidean 3-space $\E^3$, by considering the bundle map over the identity:
\begin{equation*}
\Xi: T\R^{2,1}\rightarrow T\E^3~,
\end{equation*}
which simply changes the sign of the last coordinate of a given tangent vector. It sends a Killing vector field on $\R^{2,1}$ to a Killing vector field on $\E^3$. Denote $\Pi=\Xi\circ\Upsilon: TU\rightarrow T\E^3$, which is also called an infinitesimal Pogorelov map. Then $\Pi$ is a bundle map over the inclusion $\tau: U\hookrightarrow \E^3$ and it has the following property (see e.g. \cite[Proposition 5.7]{shu}):

\begin{lemma}\label{lm:Killing HS-Eur}
Let $Z$ be a vector field on $U$. Then $Z$ is a Killing vector field (for the HS metric) if and only if $\Pi(Z)$ is a Killing vector field for the Euclidean metric on $\E^{3}$.
\end{lemma}

\subsection{The signed length} \label{subsec:signed length}
We now introduce a geometric quantity which helps to study the infinitesimal rigidity questions. Let $c:[0,1]\rightarrow\HS$ be a geodesic segment contained in an affine chart with endpoints disjoint from $\partial\AdS$. We define the \textit{signed length} of $c$, denoted by $l_{HS}(c)$, in the following way.

If $c$ is contained in $\AdS$ or ${\AdS}^*$, we define
\begin{equation*}
l_{HS}(c):={\sgn}\big(\langle \dot{c}(0),\dot{c}(0)\rangle_{HS}\big)\int_0^1\sqrt{|\langle \dot{c}(t),\dot{c}(t)\rangle_{HS}|}~dt~.
\end{equation*}

If $c$ satisfies that $c(0)\in\AdS$ and $c(1)\in{\AdS}^*$ (note that in this case $c\cap\AdS$ is spacelike), we define
\begin{equation}\label{eq:signed length 1}
l_{HS}(c):=\varepsilon_c \cdot l_{HS}(\,[c(0),p_c]\,)~,
\end{equation}
where $p_c$ is the intersection of the complete geodesic containing $c$ with the dual plane $c(1)^{\perp}$, $[c(0),p_c]$ is the geodesic segment connecting $c(0)$ and $p_c$, and $\varepsilon_c$ is defined to be $1$ (resp. $-1$) if $c(1)^{\perp}\cap c\not=\emptyset$ (resp. $c(1)^{\perp}\cap c=\emptyset$). 

If $c$ passes through $\AdS$ with both endpoints contained in ${\AdS}^*$, we define
\begin{equation}\label{eq:signed length 2}
l_{HS}(c):=l_{HS}([c(0),c(s)])+l_{HS}([c(s),c(1)])~,
\end{equation}
with $c(s)\in\AdS$. By the definition \eqref{eq:signed length 1} and the fact that the dual planes $c(0)^{\perp}$ and $c(1)^{\perp}$ are disjoint from each other in the chosen affine chart (with $c(0)^{\perp}$ lying between $c(0)$ and $c(1)^{\perp}$), we have $l_{HS}(c)=
l_{HS}([c(t_0),c(t_1)])$, where $c(t_i)$ is the intersection of $c$ with the dual plane $c(i)^{\perp}$. This implies that the above definition \eqref{eq:signed length 2} is independent of the choice of the partition of $c$ as long as $c(s)\in\AdS$. For an edge $e$ of a hyperideal AdS polyhedron $P$, the signed length of the dual edge $e^*$ of $P^*$ is equal to the dihedral angle at $e$ of $P$. (Note that the above signed lengths can be defined in a unified formulation, in terms of the Hilbert distance with respect to the boundary $\partial\AdS$ \cite[Section 2]{shu}. We define it in the above way to adapt to the convention of dihedral angles here.)

\subsection{Rigidity with respect to induced metrics} \label{sec:rigidity_metric}

This part is dedicated to proving Lemma \ref{lem:local immersion_metrics}, which states that the restriction to $\cP^{(\epsilon_1, \ldots, \epsilon_N)}$ of the map $\Phi: \cP_N\cup\polyg_N\rightarrow \mathcal{T}_{0,N}$ is a local immersion for each $(\epsilon_1,\dots,\epsilon_N)\in\{0,1\}^N$. (Note that $\Phi$ is differentiable on $\cP^{(\epsilon_1, \ldots, \epsilon_N)}$, since the induced metric on the intersection of $\AdS$ with the boundary of a polyhedron $P\in\cP^{(\epsilon_1, \ldots, \epsilon_N)}$ depends smoothly on the positions of its vertices in $\R^3$ when its deformation polyhedron $P_t$ remains in $\cP^{(\epsilon_1, \ldots, \epsilon_N)}$.)

Fix an affine chart of $\RP^3$ and let $P$ be a polyhedron in $\HS$ contained in that affine chart (identified with $\R^3$). 
Since $P$ is the convex hull in $\R^3$ of its vertices, $P$ is uniquely determined by its vertices in $\R^3$. Let $V$ denote the vertex set of $P$ and let $\Theta(P)$ denote the coordinate in $\R^3$ of its vertex set. An \emph{infinitesimal deformation} of $P$, denoted by $\dot{P}$, is the assignment to $P$ a tangent vector in $T(\R^3)^V$ at $\Theta(P)$.
We say that $\dot{P}$ is a \emph{first-order isometric deformation} of a polyhedron $P$ if it does not change the HS structure induced on $\partial P$ at first order or, equivalently, if there is a triangulation of the polyhedral surface $\partial P$ which is given by a triangulation of each face (without adding new vertices) of $P$ and a Killing vector field (for the HS metric) on each face of this triangulation, such that two Killing vector fields on two adjacent triangles coincide on the common vertices and edges,  and moreover the restriction of the Killing vector field to a vertex is exactly the restriction of $\dot{P}$ to the corresponding vertex.

 Assume that an infinitesimal deformation $\dot{P}$ of $P$ is given by a smooth one-parameter family $(P_t)_{t\in[0,\epsilon)}$ of polyhedra in $\HS$, $\dot{P}$ is trivial if $P_t=g_t(P)$ for a smooth one-parameter family $(g_t)_{t\in[0,\epsilon)}$ of isometries in $\Isom_0(\AdS)$ with $g_0=\rm{id}$. The family $(g_t)$ determines a global Killing field of $\HS$, whose restriction to the vertex set of $P$ is exactly $\dot{P}$. Equivalently, we say $\dot{P}$ is \emph{trivial} if it is the restriction to the vertex set of $P$ of a global Killing vector field of $\HS$.

To show Lemma \ref{lem:local immersion_metrics}, it suffices to consider the infinitesimal rigidity question for a (possibly degenerate) hyperideal AdS polyhedron $P$ with respect to the induced metric, which asks whether any first-order isometric deformation of $P$ is trivial.

The bundle map $\Pi$ is a key tool to solve the infinitesimal rigidity problem in our case, together with the following classical theorem of Alexandrov \cite{alex}, which provides a strong version of the infinitesimal rigidity of convex Euclidean polyhedra (see e.g. \cite{DMS,weakly}).

\begin{theorem}\label{thm:Alexandrov}
Let $P$ be a convex polyhedron in $\E^3$ and let $V$ be an infinitesimal deformation of $P$ which does not change the induced metric on $\partial P$ at first order (possibly changing the combinatorics), then $V$ is the restriction to the vertices of $P$ of a global Euclidean Killing vector field.
\end{theorem}

We first prove the following:

\begin{proposition}\label{prop:infPogorelov}
Let $P$ be a polyhedron in $\HS$ and $\dot{P}$ be an infinitesimal deformation of $P$. If $\dot{P}$ is a first-order isometric deformation with respect to the induced HS metric, then $\dot{P}$ is trivial.
\end{proposition}

\begin{proof}
By assumption, there is a triangulation, say $T$, of the boundary surface $\partial P$, given by a triangulation of each face (without adding new vertices) of $P$ and there is a Killing vector field $\kappa_f$ (with respect to the HS metric) on each face $f$ of this triangulation $T$, such that the restriction of $\kappa_{f}$ to each vertex of $f$ is equal to the restriction of $\dot{P}$ to the corresponding vertex, and for any two faces $f_1$ and $f_2$ of $T$ with a common edge $e$, the Killing vector fields $\kappa_{f_1}$ and $\kappa_{f_2}$ agree on the edge $e$. By Lemma \ref{lm:Killing HS-Eur}, $\Pi(\kappa_{f})$ is the restriction of a Killing vector field of $\E^3$ to the face $\tau(f)$ of the triangulation $\tau(T)$ of the boundary surface $\partial\tau(P)$, and for any two faces $\tau(f_1)$, $\tau(f_2)$ of $\tau(T)$ sharing an edge $\tau(e)$, $\Pi(\kappa_{f_1})$ and $\Pi(\kappa_{f_2})$ agree on the edge $\tau(e)$. Moreover, for each face $\tau(f)$ of $\tau(P)$, the restriction of $\Pi(\kappa_f)$ to the vertices of $\tau (P)$ coincides with the restriction of $\Pi(\dot{P})$ to the corresponding vertices. This implies that the infinitesimal deformation $\Pi(\dot{P})$ of $\tau(P)$ does not change the induced Euclidean metric on $\partial \tau(P)$ at first order. Combined with Theorem \ref{thm:Alexandrov}, $\Pi(\dot{P})$ is the restriction to the vertices of $\tau(P)$ of a global Euclidean Killing vector field $Y$. Using Lemma \ref{lm:Killing HS-Eur} again, $\dot{P}$ is the restriction to the vertices of $P$ of a global Killing vector field $\Pi^{-1}(Y)$ of $\HS$. This implies that $\dot{P}$ is trivial. The lemma follows.
\end{proof}

\begin{proof}[\textbf{Proof of Lemma \ref{lem:local immersion_metrics}}]
  Let $P\in \cP^{(\epsilon_1, \ldots, \epsilon_N)}$ be a hyperideal AdS polyhedron with vertex signature $(\epsilon_1, \ldots, \epsilon_N)$, then the induced HS metric restricted to $\partial P\cap\AdS$ is 
a complete hyperbolic metric on $\Sigma_{0,N}$ of signature $(\epsilon_1, \ldots, \epsilon_N)$,
which determines a point in $\cT^{(\epsilon_1, \ldots, \epsilon_N)}_{0,N}$.

We first claim that the HS structure on the boundary of a hyperideal polyhedron $P$ is uniquely determined by the hyperbolic metric on $\partial P\cap \AdS$. This is clear for the case that $(\epsilon_1, \ldots, \epsilon_N)=(0,\dots,0)$. 
It suffices to show the remaining cases, where the polyhedron $P$ has at least one end corresponding to a strictly hyperideal vertex.

Consider such an end, bounded by a closed geodesic $\gamma$ of the induced metric on $\partial P\cap \AdS$. Then $\gamma=\partial P\cap v^{\perp}$ where $v$ is a strictly hyperideal vertex of $P$. As a consequence, for each face $f$ of $P$ adjacent to $v$, $f$ is isometric to a polygon $\bar f$ in $\HSS$, and $\gamma\cap f$ then corresponds to the intersection of $\bar f$ with 
$\bar{v}^{\perp}$, where $\bar v$ is the point in $\HSS$ corresponding to $v$ (as shown in Figure \ref{fig:metric_dual plane}). As a consequence, the HS structure on the connected component of $\partial P\setminus \gamma$ containing $v$ is isometric to a ``standard'' model: the quotient by the translation of length $L(\gamma)$ along 
$\bar{v}^{\perp}$ of the triangle in $\HSS$ with vertex $\bar v$ and opposite edge $\bar{v}^{\perp}\cap \HH^2$ (see e.g. Figure \ref{fig:metric_dual plane}).

\begin{figure}
\begin{subfigure}[b]{0.46\textwidth}
\centering
\begin{tikzpicture}[scale=1]
\draw  (-0.5,2.55) ellipse (2.5 and 0.5);
\draw[densely dashed] (-2.95,-2.5) .. controls (-2.95,-1.85) and (2,-1.9) .. (2,-2.55);
\draw (-2.95,-2.5) .. controls (-2.95,-3.15) and (2,-3.2) .. (2,-2.55);
\draw (-3.5,3) .. controls (-0.5,0.05) and (-0.5,0.05) .. (-3.5,-3);
\draw (2.5,3) .. controls (-0.5,0.05) and (-0.5,0.05) .. (2.5,-3);
\draw(-1.25,0) .. controls (-1.2,-0.35) and (0.25,-0.35) .. (0.25,0);
\draw[densely dashed](-1.25,0) .. controls (-1.2,0.35) and (0.25,0.35) .. (0.25,0);
\draw [fill=black](1,0) ellipse (0.03 and 0.03);
\draw(-0.4,2.05) .. controls (-0.1,0.6) and (-0.1,-1) .. (-0.28,-3);
\draw[densely dashed](0.4,3) .. controls (0,1.8) and (0,-0.9) .. (0.46,-2.08);
\node at (1,3.4) {$v^{\perp}\cap\mathbb{A}{\rm{d}}\mathbb{S}^3$};
\draw [fill=black](0,0) ellipse (0.03 and 0.03);
\draw[white][fill=gray,opacity=0.3](-0.4,2.05) .. controls (-0.1,0.6) and (-0.1,-1) .. (-0.28,-3) .. controls (0.46,-2.08) and (0.46,-2.08) .. (0.46,-2.08) .. controls (0,-0.9) and (0,1.8) .. (0.4,3) .. controls (-0.4,2.05) and (-0.4,2.05) .. (-0.4,2.05);
\draw[thick](1,0) .. controls (-0.5,0.35) and (-0.5,0.35) .. (-0.5,0.35);
\draw[thick](1,0) .. controls (-0.6,-0.08) and (-0.6,-0.08) .. (-0.6,-0.08);
\draw[thick](1,0) .. controls (-0.6,-0.4) and (-0.6,-0.4) .. (-0.6,-0.4);
\draw[red,thick](0.08,0.24) .. controls (-0.08,-0.04) and (-0.1,-0.04) .. (-0.1,-0.04);
\draw[red, thick](-0.1,-0.04) .. controls (0.06,-0.25) and (0.06,-0.25) .. (0.06,-0.25);
\draw[red,densely dashed, thick](0.08,0.24) .. controls (0.06,-0.25) and (0.06,-0.25) .. (0.06,-0.25);
\node at (1.3,0) {$v$};
\node at (1.7,-0.5) {$\gamma=\partial P\cap v^{\perp}$};
\draw(0.05,-0.32) .. controls (0.3,-0.6) and (0.6,-0.5) .. (0.6,-0.5);
\node at (-0.38,0.12) {$f$};
\end{tikzpicture}
\end{subfigure}
    \begin{subfigure}[b]{0.46\textwidth}
     \centering
\begin{tikzpicture}
\usetikzlibrary{decorations.pathreplacing}
\draw [thick] (-0.5,1) node (v1) {} ellipse (2 and 2);
\draw(3.5,1) .. controls (0.4,2.8) and (0.4,2.8) .. (0.4,2.8) .. controls (0.4,-0.8) and (0.4,-0.8) .. (0.4,-0.8) .. controls (3.5,1) and (3.5,1) .. (3.5,1);
\node at (3.8,1) {$\bar v$};
\node at (-0.35,2.3) {$\bar v^{\perp}\cap\mathbb{H}^2$};
\draw[thick](3.5,1) .. controls (0.1,1.9) and (0.1,1.9) .. (0.1,1.9);
\draw[thick](3.5,1) .. controls (0.2,1.2) and (0.2,1.2) .. (0.2,1.2);
\draw[thick](3.5,1) .. controls (0.2,0.6) and (0.2,0.6) .. (0.2,0.6);
\draw[thick](3.5,1) .. controls (0.1,0) and (0.1,0) .. (0.1,0);
\draw[red, thick](0.4,1.8) .. controls (0.4,0.1) and (0.4,0.1) .. (0.4,0.1);
\node at (-0.5,1) {$L(\gamma)$};
\draw[decorate,decoration={brace,mirror}] (0.1,1.8) -- (0.1,0.1);
\node at (2,3) {$\mathbb{HS}^2$};
\node at (0.8,0.9) {$\bar f$};
\end{tikzpicture}
    \end{subfigure}
    \caption{\small{An example of an end of $\partial P\cap\AdS$ corresponding to a strictly hyperideal vertex $v$ of $P$. The left picture is drawn in an affine chart of $\bHS^3_1$, while the right picture is drawn in an affine chart of $\HSS$. A face $f$ (partially drawn) of $P$ is isometric to a polygon $\bar{f}$ (partially drawn) in $\HSS$, with the vertex $v$ of $f$ corresponding to the vertex $\bar{v}$ of $\bar{f}$, and the geodesic segment $f\cap \gamma$ corresponding to the segment $\bar{f}\cap \bar{v}^{\perp}$.}}
    \label{fig:metric_dual plane}
    \end{figure}

This simple description of the HS structure on $\partial P$ from the hyperbolic metric on $\partial P\cap \AdS$ also implies that any first-order deformation of $P$ which does not change the induced hyperbolic structure on $\partial P\cap \AdS$ also does not change the induced HS structure on $\partial P$.

Consider now a first-order deformation $\dot P$ of $P$ which does not change the induced hyperbolic structure on $\partial P\cap \AdS$. Then $\dot P$ does not change the induced HS structure, and by Proposition \ref{prop:infPogorelov} 
this proves that $\dot P$ is trivial.
\end{proof}

\subsection{Rigidity with respect to dihedral angles}\label{sec:rigidity_angle}

In this part we will prove Lemma \ref{lem:local immersion_angles}, which states that the map $\Psi: \cP_N\rightarrow \R^E$ is a local immersion near a non-degenerate hyperideal AdS polyhedron $P$ whose 1-skeleton is a subgraph of a triangulation $\Gamma$ of $\Sigma_{0,N}$, where $E=E(\Gamma)$. To show this, it suffices to consider the infinitesimal rigidity question of $P$ with respect to dihedral angles, which asks whether any infinitesimal deformation of $P$ that preserves the dihedral angle (at each edge of $P$) at first order is trivial.

\subsubsection{An alternative version of the rigidity question}\label{sect:alternative version}

Instead of directly considering the infinitesimal rigidity question with respect to dihedral angles, we translate it into an alternative version.

For each non-degenerate hyperideal AdS polyhedron $P$, we construct a polyhedron, say $P_0$, in the following way: if each vertex of $P$ is ideal, we let $P_0=P$; otherwise, for each strictly hyperideal vertex $v$, we denote by $v^{\perp}$ the dual plane of $v$ and by $H_v$ the half-space which is delimited by $v^{\perp}$ and does not contain $v$. Then we define $P_0$ to be the intersection of $P$ with $H_v$ over all strictly hyperideal vertices $v$.  We call such $P_0$ the \emph{truncated polyhedron} of $P$.

Indeed, the dihedral angle at each new edge of $P_0$ obtained from the truncations is always orthogonal (which is defined to be zero by our convention), since the dual plane (serving as a truncated plane) of a strictly hyperideal vertex $v$ of $P$ is orthogonal to all the faces adjacent to $v$. Therefore, the infinitesimal rigidity question of $P$ with respect to dihedral angles is equivalent to the infinitesimal rigidity question of $P_0$ with respect to dihedral angles.

Now we consider the dual polyhedron, say $P^*_0$, of $P_0$. Recall that an edge $e$ of $P_0$ connecting two vertices $v$, $v'$ is dual to an edge $e^*$ of $P^*_0$ between the two faces $v^*$, $(v')^*$ dual to $v$, $v'$ respectively. Therefore the (exterior) dihedral angle at each edge $e$ of $P_0$ is equal to the signed length of the dual edge $e^*$ of $P^*_0$ (see Section \ref{subsec:signed length}). Therefore, the infinitesimal rigidity of $P_0$ with respect to dihedral angles is equivalent to the infinitesimal rigidity question of $P^*_0$ with respect to the (signed) edge lengths.

The following is a key property of a first-order deformation (with respect to dihedral angles) $\dot{P}$ of $P$ with at least one ideal vertex.

\begin{proposition}\label{Prop:deformation at ideal vertex}
Let $P$ be a hyperideal AdS polyhedron with at least one ideal vertex $v$ and let $\dot{P}$ be an infinitesimal deformation of $P$. If $\dot{P}$ preserves the dihedral angle at each edge of $P$ at first order, then the restriction of $\dot{P}$ to the vertex $v$ is tangent to the boundary $\partial \AdS$.
\end{proposition}

\begin{proof}
  Suppose by contradiction that the restriction of $\dot{P}$ to the vertex $v$ is not tangent to the boundary $\partial \AdS$, for instance that  the vector $\dot{P}|_{v}$ is towards the exterior of $\partial \AdS$, and we will show that the sum of angles at the edges adjacent to $v$ then strictly increases (at first order). The result will clearly follow.

  Since this sum of angles remains constant when $\dot{P}$ at $v$ is tangent to the boundary, we can assume that $\dot{P}|_{v}$ is orthogonal to the boundary and of unit norm in the Euclidean metric, and that $\dot P$ vanishes at all the other vertices. We can choose an orthonormal coordinate system $(x,y,z)$ in $\R^3$ such that $v=(0,1,0)$, $\dot P_v=(0,1,0)$, while $\partial\AdS$ corresponds to the quadric of equation $x^2+y^2-z^2=1$ (see Figure \ref{fig:deform vertex}).

  Let $(P_t)_{t\in[0,\epsilon)}$ be a family of convex hyperideal AdS polyhedra that realize the infinitesimal deformation $\dot{P}$ of $P$ at $t=0$. More precisely, $P_0=P$, the vertex $v_t$ of $P_t$ (corresponding to $v$ of $P$) has coordinates $v_t=(0,1+t,0)$, and $v'_t=v'$ for any vertex $v'\not=v$ of $P$. Let $e_1$, $e_2$, ..., $e_k$ denote the edges of $P$ adjacent to $v$, let $v_i$ be the vertex of $e_i$ different from $v$, and let $e^t_i$ denote the edges of $P_t$ corresponding to $e_i$.

\begin{figure}
\begin{subfigure}[b]{0.46\textwidth}
\centering
\begin{tikzpicture}[scale=1.06]
\draw  (-0.5,2.55) ellipse (2.5 and 0.5);
\draw[densely dashed] (-2.95,-2.5) .. controls (-2.95,-1.85) and (2,-1.9) .. (2,-2.55);
\draw (-2.95,-2.5) .. controls (-2.95,-3.15) and (2,-3.2) .. (2,-2.55);
\draw (-3.5,3) .. controls (-0.5,0.05) and (-0.5,0.05) .. (-3.5,-3);
\draw (2.5,3) .. controls (-0.5,0.05) and (-0.5,0.05) .. (2.5,-3);
\draw(-1.25,0) .. controls (-1.2,-0.35) and (0.25,-0.35) .. (0.25,0);
\draw[densely dashed](-1.25,0) .. controls (-1.2,0.35) and (0.25,0.35) .. (0.25,0);
\draw[-latex](-0.5,0)--(3.6,0);
\draw[-latex](-0.5,0)--(-3,-1);
\draw[-latex](-0.5,0)--(-0.5,3.5);
\draw [fill=black](1,0) ellipse (0.03 and 0.03);
\node at (-3.2,-1.2) {$x$};
\node at (3.9,0) {$y$};
\node at (-0.7,3.3) {$z$};
\draw(-0.58,2.05) .. controls (-0.3,0.6) and (-0.3,-1) .. (-0.58,-3.03);
\draw[densely dashed](0.4,3) .. controls (0,1.8) and (0,-0.9) .. (0.46,-2.08);
\node at (2.1,-0.37) {$v_t=(0,1+t,0)$};
\node at (1.2,3.25) {$\mathbb{A}{\rm{d}}\mathbb{S}^3 \cap\{y=f(t)\}$};
\draw [fill=black](-0.1,0) ellipse (0.03 and 0.03);
\node at (-0.54,-0.14) {$o$};
\draw[white][fill=gray,opacity=0.3](-0.58,2.05) .. controls (-0.3,0.6) and (-0.3,-1) .. (-0.58,-3.03) .. controls (0.46,-2.08) and (0.46,-2.08) .. (0.46,-2.08) .. controls (0,-0.9) and (0,1.8) .. (0.4,3) .. controls (-0.58,2.05) and (-0.58,2.05) .. (-0.58,2.05);
\draw[blue,line width=0.33mm](1,0) .. controls (-0.6,0.27) and (-0.6,0.27) .. (-0.6,0.27);
\draw[blue,line width=0.33mm](1,0) .. controls (-0.6,-0.4) and (-0.6,-0.4) .. (-0.6,-0.4);
\draw[blue,line width=0.33mm,densely dashed](1,0) .. controls (-0.56,0.08) and (-0.56,0.08) .. (-0.56,0.08);
\draw[blue,line width=0.33mm](1,0) .. controls (-0.45,-0.1) and (-0.45,-0.1) .. (-0.45,-0.1);
\draw[blue,line width=0.33mm,densely dashed](1,0) .. controls (-0.5,-0.26) and (-0.5,-0.26) .. (-0.5,-0.26);
\draw[red,line width=0.35mm](0.08,0.05) .. controls (-0.18,0.2) and (-0.18,0.2) .. (-0.18,0.2);
\draw[red, line width=0.33mm](0.08,0.05) .. controls (0.01,-0.16) and (0.01,-0.16) .. (0.01,-0.16);
\draw[red, line width=0.33mm](0.01,-0.16) .. controls (-0.2,-0.3) and (-0.2,-0.3) .. (-0.2,-0.3);
\draw[red, line width=0.33mm](-0.2,-0.3) .. controls (-0.32,-0.1) and (-0.32,-0.1) .. (-0.32,-0.1);
\draw[red, line width=0.33mm](-0.32,-0.1) .. controls (-0.18,0.2) and (-0.18,0.2) .. (-0.18, 0.2);
\node at (0.6,0.24) {{\color{blue}$e_{i,t}$}};
\end{tikzpicture}
\end{subfigure}
    \begin{subfigure}[b]{0.46\textwidth}
     \centering
\begin{tikzpicture}[scale=1]
\draw[white][fill=gray,opacity=0.3](-4.2,3) .. controls (-1.5,0.05) and (-1.5,0.05) .. (-4.2,-3) .. controls (2.45,-3) and (2.45,-3) .. (2.45,-3) .. controls (-0.7,0.05) and (-0.7,0.05) .. (2.45,3) .. controls (2.45,3) and (-4.2,3) .. (-4.2,3);
\draw (-4.2,3) .. controls (-1.5,0.05) and (-1.5,0.05) .. (-4.2,-3);
\draw (2.45,3) .. controls (-0.7,0.05) and (-0.7,0.05) .. (2.45,-3);
\draw[densely dashed](-2.18,0) .. controls (0.218,0) and (0.218,0) .. (0.085,0);
\node at (-1,-0.19) {{\footnotesize $O$}};
\draw [fill=black](-1,0) ellipse (0.03 and 0.03);
\node at (-1,2) {$v_t^{\perp}\cap \mathbb{A}\rm{d}\mathbb{S}^3$};
\draw(-1,-0.49) .. controls (-1,-1) and (-1,-1.5) .. (-1,-1.5);
\node at (-1,-1.8) {$v_t^{\perp}\cap P_t$};
\draw[blue,densely dashed](0.67,1.25) -- (-2.17,0.15);
\draw[red,thick](-1,0.6) -- (-2.02,0.2) -- (-1.3,-0.45) -- (-0.55,-0.46) -- (-0.1,-0.06) -- (-1,0.6);
\node at (-1,0.85) {$\bar{v}_{1,t}$};
\node at (-1.9,0.5) {$\bar{v}_{2,t}$};
\node at (-1.5,-0.8) {$\bar{v}_{3,t}$};
\node at (-0.3,-0.8) {$\bar{v}_{4,t}$};
\node at (-0.17,0.3) {$\bar{v}_{5,t}$};
\node at (1.1,1.2) {$a_{ij,t}$};
\node at (-2.6,0.2) {$b_{ij,t}$};
\end{tikzpicture}
    \end{subfigure}
    \caption{\small{An example of a deformation of an ideal vertex $v$ of $P$, with the deformation vector at $v$ outwards orthogonal to $\partial \AdS$ and of unit norm in the Euclidean metric. The left picture is drawn in an affine chart, where  $e_{i,t}$ represents an edge of $P_t$ adjacent to $v_t$. 
    The right picture zooms in the intersection part of $v_t^{\perp}$ with $P_t$, which is a convex polygon (shown in the bold lines).}} 
    \label{fig:deform vertex}
    \end{figure}

  Since $v\in\partial\AdS$,
\begin{equation}\label{sum}
\sum_{i=1}^{k}\theta(e_i)=0~.
\end{equation}
We claim that there exists a constant $c>0$ (depending on $P$) such that for $t>0$ small enough,
\begin{equation}\label{sum_t}
\sum_{i=1}^{k}\theta(e^t_i)\geq ct~.
\end{equation}

Indeed, the position of the plane $v_t^\perp$ dual to $v_t$ is uniquely determined by its intersection with the line $Oy$, denoted by $f(t)$. Since any two pairs $(p, p^{\perp})$ and $(q, q^{\perp})$ with $p, q\in{\AdS}^*$ differ by the action of a projective transformation in $\PO(2,2)$, the cross ratio (as a projective invariant) $[f(t), v_t; -1, 1]$ is equal to $[0,\infty; -1, 1]$ for all $t$. A direct computation shows that $f(t)=1/(1+t)$ and the plane $v_t^\perp$ has the equation
$$ v_t^\perp = \left\{ y=\frac 1{1+t}\right\}~. $$
As a consequence, the intersection point $\bar v_{i,t}$ of $v_t^\perp$ with $e_{i,t}$ has coordinates
$$ \bar v_{i,t} = \left(\frac{t^2+2t}{(1+t-y_i)(1+t)} x_i, \,
\frac 1{1+t},\, \frac{t^2+2t}{(1+t-y_i)(1+t)} z_i\right)~,$$
where $v_i=(x_i, y_i,z_i)$.

However the intersection ${\partial\AdS}\cap v_t^{\perp}$ has the equation
$$ x^2-z^2= 1-\left(\frac 1{1+t}\right)^2~, $$
so
$$ x^2-z^2= \frac {t^2+2t}{(1+t)^2}~. $$
As a consequence, the line through $\bar v_{i,t}$ and $\bar v_{j,t}$ (for $i\not=j\in \{ 1,\cdots, k\}$) intersects the quadric $\partial\AdS$ at two points $a_{ij,t}, b_{ij,t}$ with coordinates (which are complex numbers if the line through $\bar v_{i,t}$ and $\bar v_{j,t}$ is timelike) behaving as a constant times $\sqrt{t}$ as $t\to 0$. This implies that the AdS distance between $\bar v_{i_,t}$ and $\bar v_{j,t}$, which can be computed as (see e.g. \cite[Lemma 2.10]{FS19} and \cite[Section 2]{shu})
$$ d(\bar v_{i_,t},\bar v_{j_,t})=|\,\frac 12\log [ \bar v_{i_,t},\bar v_{j_,t}; a_{ij,t},b_{ij,t}]\,|~, $$
behaves as a constant times $\sqrt{t}$ as  $t\to 0$.

As a consequence, the area of the polygon  $v_t^{\perp}\cap P_t$ with vertices the $\bar v_{i,t}$ behaves as a constant times $t$ as $t\to 0$. The estimate \eqref{sum_t} therefore follows from the Gauss-Bonnet formula, applied to this polygon.

The proof of the proposition then follows from this lower bound.
\end{proof}

As a consequence, we obtain the following property of an infinitesimal deformation $\dot{P}^*_0$ of $P^*_0$ with at least one face contained in a light-like plane.

\begin{corollary}\label{cor:dual plane}
Let $P_0$ be a hyperideal AdS polyhedron with at least one ideal vertex $v$ and let $\dot{P}^*_0$ be an infinitesimal deformation of ${P}^*_0$. If $\dot{P}^*_0$ preserves all the edge lengths of $P^*_0$ at first order, then for any family $\{(P^*_0)_t\}_{t\in[0,\epsilon)}$ of polyhedra with derivative $\dot{P}^*_0$ at $t=0$, the face $f_v^t$ of $(P^*_0)_t$ corresponding to the face $f_v$ (dual to $v$) of $P^*_0$ remains a light-like plane tangent to $v$ at first order at $t=0$.
\end{corollary}

We now state a simple but crucial relation between the variation of edge lengths of a polyhedron $P$ in  $\RP^3$ (with all vertices disjoint from $\partial\AdS$) and the deformation of the induced HS metric,  which will be used to show Lemma \ref{lem:local immersion_angles}.

\begin{lemma}\label{lm:edge lengths}
Let $P$ be a polyhedron in $\HS$ whose vertices are all disjoint from $\partial\AdS$. Assume that $P$ admits a triangulation $T$ given by subdivising (without adding new vertices) each face. Then an infinitesimal deformation $\dot{P}$ of $P$ that preserves all the edge lengths of $T$ at first order (with respect to the induced HS metric)
is a first-order isometric deformation of $P$.
\end{lemma}

\begin{proof}
  Note that all the vertices of $P$ are disjoint from $\partial\AdS$, therefore, for each edge $e$ of the triangulation $T$, the signed length of $e$ with respect to the induced HS metric   is finite (see Section \ref{subsec:signed length}).

For each edge $e$ of $T$, the restriction of an infinitesimal deformation of $P$ to the two endpoints of $e$ completely determines the infinitesimal deformation of $e$. By assumption, $\dot{P}$ does not change the edge lengths of $T$ at first order, so that, for each face $f$ of the triangulation $T$, there is a unique Killing vector field $\kappa_f$ on $f$ such that the restriction of $\kappa_f$ to each vertex of $f$ is equal to the restriction of $\dot{P}$ to the corresponding vertex. Moreover, for any two faces $f_1$ and $f_2$ of $T$ with a common edge $e$, the Killing vector fields $\kappa_{f_1}$ and $\kappa_{f_2}$ agree on the endpoints of $e$ and hence on $e$. Therefore, $\dot{P}$ is a first-order isometric deformation of $P$.
\end{proof}

\subsubsection{Degenerate metrics induced on light-like planes}\label{subsec:degnerate metric}
In the case that $P$ has at least one ideal vertex $v$, to prove the infinitesimal rigidity of $P^*_0$ with respect to edge lengths, we need some basic facts on the degenerate metric induced on the plane $v^{\perp}$ dual to $v$.

Note that 
$v^{\perp}$ is a totally geodesic light-like plane tangent to $\partial\AdS$ that intersects the boundary $\partial \AdS$ along two light-like lines meeting at $v$. The HS metric induced on $v^{\perp}\setminus\partial \AdS$ is degenerate. Indeed, for each geodesic line $\ell$ contained in $v^{\perp}$ passing through $v$, the induced (pseudo) distance, say $d_{HS}$, between any two points $p$, $q$ in $\ell\setminus\partial \AdS$ is zero (see Section \ref{sect:AdS}). 

It remains to determine the pseudo-distance $d_{HS}$ for any two points $p_1$ and $p_2$ distinct from $v$ and lying on two distinct geodesic rays $\overline{vp_1}$ and $\overline{vp_2}$ in $v^{\perp}\setminus\partial\AdS$ with the initial point $v$. It can be seen later that in this case it is more natural to first extend $d_{HS}$ to a signed directed measure, say $m_{HS}$, that is, $m_{HS}(p_1, p_2)$ is allowed to be negative and different from $m_{HS}(p_2,p_1)$, and then define $d_{HS}$ by determining the choice of the direction.

To clarify this, we identify the plane $v^{\perp}$ with the real 2-dimensional vector space $\R^2$ (where $v$ is identified with the origin $0$) equipped with the push-forward degenerate HS metric. Let $\ell_L$ and $\ell_R$ be two lines in $\R^2$ through $0$ with slopes $1$ and $-1$ respectively, which are identified with the left and right leaves at the intersection of $v^{\perp}$ with $\partial \AdS$. These two lines divides $\R^2$ into four regions, say $I$, $II$, $III$ and $IV$, consecutively in the anticlockwise direction with $I:=\{(x,y)\in\R^2: x>|y|>0\}$. 
A geodesic ray in $v^{\perp}$ starting from $v$ but not contained in $\partial \AdS$ is thus identified with a ray in $\R^2$ starting from $0$ with slope neither $1$ nor $-1$ (called a \emph{non-singular} ray, which is exactly contained in one of the four regions).

For any two rays $\ell_1,\ell_2$ in $\R^2$, we denote by $ \ell_10\ell_2$ a \emph{directed angle} with the initial direction $\ell_1$, the terminal direction $\ell_2$, and a specified region bounded by $\ell_1$ and $\ell_2$ with Euclidean angle at $0$ not greater than $\pi$. It is clear that the choice of this region is unique in the case that $\ell_1$ and $\ell_2$ are not opposite to each other. Otherwise, we arbitrarily fix one of the two regions.

We say the directed angle $\ell_10\ell_2$ is \emph{non-singular} if both $\ell_1$ and $\ell_2$ are non-singular. A non-singular directed angle $\ell_10\ell_2$ is called a \emph{fundamental} angle if either

(a) $\ell_1$ and $\ell_2$ are contained in the same region of the four, or

(b) the specified region of $ \ell_10\ell_2$ intersects exactly one of the two lines $\ell_L$, $\ell_R$ and $\ell_1$ is orthogonal to $\ell_2$ with respect to the Minkowski metric $\langle \cdot,\cdot\rangle_{1,1}$ on $\R^2$.

We first define the \emph{non-directed sectorial measure} with respect to $\langle \cdot,\cdot\rangle_{1,1}$ on $\R^2$ of the non-directed angle between the above two rays $\ell_1$, $\ell_2$, denoted by $\angle (\ell_1,\ell_2)$, in the following (see \cite[Definition 2]{DJJ} for an alternative version).

\begin{definition}\label{def:angles}
Let $\ell_1$ and $\ell_2$ be two non-singular rays lying in either Case (a) or Case (b).
\begin{itemize}
\item
in Case (a), if $\ell_1$ and $\ell_2$ lie in the same region $I$ or $III$ (resp. $II$ or $IV$), then the measure $\angle (\ell_1,\ell_2)$ is defined to be non-positive (resp. non-negative) and satisfies that
\begin{equation*}
\cosh \angle (\ell_1,\ell_2)= |\langle w_1, w_2\rangle_{1,1}|~,
\end{equation*}
where $w_i$ is the unit vector in the direction $\ell_i$ with respect to $\langle \cdot,\cdot\rangle_{1,1}$ for $i=1,2$.
\item
in Case (b), the measure $\angle (\ell_1,\ell_2)$ is defined to be zero.
\end{itemize}
\end{definition}

We then define the \emph{directed sectorial measure}, denoted by $\measuredangle \ell_10\ell_2$, of the non-singular directed fundamental angle $\ell_10\ell_2$ to be $\measuredangle \ell_10\ell_2:=\angle (\ell_1,\ell_2)$ if the specified region of $\ell_10\ell_2$ is obtained by an anticlockwise rotation from $\ell_1$ to $\ell_2$, otherwise, we define it as $\measuredangle \ell_10\ell_2:=-\angle (\ell_1,\ell_2)$. The following gives a natural way to extend the directed sectorial measure to a general non-singular directed angle $\ell_10\ell_2$ (see e.g. Definition 7 in \cite{DJJ}).

\begin{definition}\label{def:directed angle}
Let $\ell_10\ell_2$ be a general non-singular directed angle. The directed sectorial measure $\measuredangle\ell_10\ell_2$ of $\ell_10\ell_2$ is defined by splitting the angle $\ell_10\ell_2$ into successive non-overlapping non-singular directed fundamental angles and then summing the directed sectorial measures of these fundamental angles. One can directly check that this definition is independent of the choices of splittings.
\end{definition}

It is clear that $\measuredangle \ell_10\ell_2$ is zero whenever $\ell_1$ and $\ell_2$ are the same, opposite or orthogonal to each other. This also explains why the choice of the specified region for the directed angle $\ell_10\ell_2$ can be arbitrary in the case that $\ell_1$ and $\ell_2$ are opposite.

Now we are ready to extend $d_{HS}$ to a signed directed measure, say $m_{HS}$, for any ordered pairs of points $(p_i,p_j)$ with $p_i$ and $p_j$ lying on two non-singular rays $\ell_i$ and $\ell_j$ in $v^{\perp}$ respectively.

\begin{definition}\label{def:directed measure}
Let $\ell_1$, $\ell_2$ be two non-singular rays in $\R^2$ and let $p_1$, $p_2$ be any two points lying on $\ell_1$, $\ell_2$. The (signed) length of the directed segment $\overrightarrow{p_1p_2}$ in $\R^2$ connecting $p_1$ to $p_2$ is defined to be $m_{HS}(p_1,p_2):=\measuredangle \ell_10\ell_2$.
\end{definition}

By definition, $m_{HS}(p_1,p_2)=-m_{HS}(p_2,p_1)$. Recall that in this subsection $P$ has at least one ideal vertex say $v$ and let $f_v$ be the face of $P^*_0$ dual to $v$ of the truncated polyhedron $P_0$. By duality, for each edge $e$ of $P_0$ adjacent to $v$, the (exterior) dihedral angle at $e$ is equal to the signed length (with respect to the 
HS metric induced on $\partial P^*_0$) of $e^*$ (which is an edge of $f_v$).

Indeed, from our convention, the dihedral angle at the edge $e$ adjacent to two faces $f_1$ and $f_2$ of $P_0$ can be viewed as the directed sectorial measure of the directed angle $\ell_10\ell_2$ (resp. $\ell_20\ell_1$) with $\ell_1$ and $\ell_2$ the outwards-pointing rays (based on one interior point of $e$) orthogonal to $f_1$ and $f_2$ respectively if $\ell_2$ (resp. $\ell_1$) is obtained from $\ell_1$ (resp. $\ell_2$) by an anticlockwise rotation  with Euclidean angle less than $\pi$. This directed angle induces a natural orientation of the dual edge $e^*$ in the following way: $e^*$ is directed from the endpoints $f_1^*$ to $f_2^*$ (resp. $f_2^*$ to $f_1^*$) if the aforementioned directed angle at $e$ is $\ell_10\ell_2$ (resp. $\ell_20\ell_1$). Therefore, each edge of the face $f_v\subset v^{\perp}$ is implicitly equipped with a direction. One can check that the union of the directions over the edges of $f_v$ indeed determine an orientation of the face $f_v$, which is exactly a clockwise rotation around $\partial f_v$. The signed length of $e^*$ with respect to the induced degenerate HS metric on $\partial P^*_0$ is exactly the signed length with respect to $m_{HS}$ of the directed edge $e^*$ (see Definition \ref{def:directed measure}).

In general, the choice of the direction for a line segment in $v^{\perp}$ with endpoints (say $p_1$, $p_2$) disjoint from $\partial \AdS$ corresponds to the choice of the sign (or definition) of the dihedral angle between the two dual planes $p^{\perp}_1$, $p^{\perp}_2$ with prescribed orientations (described by prescribed normal vectors). 

The directed measure $m_{HS}$ has a convenient property, stated here for completeness.

\begin{claim}\label{clm:sum}
Let $Q$ be an oriented convex polygon in the light-like plane $v^{\perp}$ with all vertices disjoint from $\partial \AdS$. Then the sum of the signed lengths (with respect to $m_{HS}$) over all the directed edges of $Q$ is $0$.
\end{claim}

\begin{proof}
  Without   loss of generality, we assume that the boundary of $Q$ is anticlockwise oriented. By assumption, the ray starting from $v$ through any vertex of $Q$ is non-singular. Therefore, the signed length of each directed edge of $Q$ is well-defined (see Definition \ref{def:directed measure}). We discuss $Q$ in the following two cases:
\begin{itemize}
\item $Q$ contains the vertex $v$ in its boundary or in its interior.
\item The closure of $Q$ is disjoint from $v$.
\end{itemize}

In the first case, let $v_1$, $v_2$, ..., $v_k$ be the consecutive vertices of $Q$ in the anticlockwise order. Denote the ray starting from $v$ through $v_i$ by $\ell_i$ for $i=1,...,k$. Since the vertices of $Q$ are all disjoint from $\partial \AdS$, the rays $\ell_1$, $\ell_2$, ..., $\ell_k$ then split the whole plane $v^{\perp}$ into $k$ successive non-overlapping non-singular angles $\ell_10\ell_2$, $\ell_20\ell_3$, ..., $\ell_k0\ell_1$. By Definition \ref{def:directed measure}, the signed length of the directed edge with endpoints $v_i$ and $v_{i+1}$ is $\measuredangle \ell_i0\ell_{i+1}$ for $i=1, ..., k$ (where $v_{k+1}=v_1$) . Let $\ell^*_1$ deonte the ray opposite to $\ell_1$. It follows from Definition \ref{def:directed angle} that the sum of $\measuredangle \ell_i0\ell_{i+1}$ over $i=1, ..., k$ (where $\ell_{k+1}=\ell_{1}$) is equal to the sum of $\measuredangle \ell_10\ell^*_1$ and $\measuredangle \ell^*_10\ell_1$, since they share a common finer splitting of the plane $v^{\perp}$. Note that $\measuredangle \ell_10\ell^*_1+\measuredangle\ell^*_10\ell_1=\measuredangle\ell_10\ell^*_1-\measuredangle \ell_10\ell^*_1=0$. We obtain the desired result.

In the second case, we order the vertices of $Q$ and thus the corresponding rays successively in the anticlockwise order, say $\ell_1$, $\ell_2$, ..., $\ell_k$, such that the anticlockwise rotation from $\ell_1$ to $\ell_k$ crosses all the other rays $\ell_2$, $\ell_3$, ..., $\ell_{k-1}$. It is easy to check that $\measuredangle \ell_10\ell_2+\measuredangle \ell_20\ell_3+...+\measuredangle \ell_{k-1}0\ell_{k}=\measuredangle \ell_10\ell_k=-\measuredangle\ell_k0\ell_1$. By definition, the sum of the  signed lengths over all the directed edges of $Q$ is equal to $\measuredangle \ell_10\ell_2+\measuredangle \ell_20\ell_3+...+\measuredangle \ell_{k-1}0\ell_{k}+\measuredangle\ell_k0\ell_1$ and is thus zero.
\end{proof}

\begin{claim} \label{clm:edge lengths}
Let $v$ be an ideal vertex of $P$ and let $f_v$ be the face of $P^*_0$ contained in the dual plane $v^{\perp}$.  Let $T=\{T_j\}_{j=1}^m$ be a triangulation of $f_v$ (without adding new vertices). 
If an infinitesimal deformation $\dot{P}^*_0$ of $P^*_0$ preserves all the edge lengths of $P^*_0$ at first order, then $\dot{P}^*_0$ preserves the lengths of all the directed edges of 
$T_j$ at first order for $1\leq j\leq m$.
\end{claim}

\begin{proof}
Let $v_1$, ..., $v_k$ be the consecutive vertices of the face $f_v$ in the clockwise order and let $\ell_i$ be the ray starting from $v$ through $v_i$ for $i=1, ..., k$. Since $v_i$ is dual to 
a space-like face in $\AdS$ of $P_0$, $v_i$ is contained in $\AdS$ and thus contained in $v^{\perp}\setminus\partial \AdS$. This implies that $\ell_i$ is a non-singular ray for all $i=1, ..., k$. Moreover, each edge of $f_v$ connecting $v_i$ and $v_{i+1}$ is equipped with an orientation from $v_i$ to $v_{i+1}$ (where $v_{k+1}=v_1$), compatible with the sign of the dihedral angle at its dual edge. For each $T_j$, we orient $\partial T_j$
clockwise. This is compatible with the orientation of $\partial f_v$ at the common edges.

\textbf{Step 0.} First we claim that if $\dot{P}^*_0$ preserves the (signed) lengths of two directed edges, say $e_{j_1}$, $e_{j_2}$ of a triangle say $T_j$ at first order, then $\dot{P}^*_0$ also preserves the (signed) length of the third directed edge, say $e_{j_3}$, of $T_j$ at first order. Let $v_{j_1}$, $v_{j_2}$, $v_{j_3}$ denote the three vertices of $T_j$. By definition and our convention, the sum of the (signed) lengths of directed edges of $T_j$ is either equal or opposite to the sum of dihedral angles at the edges of the (unbounded) polyhedron (say $P^*_j$) bounded by the three dual (space-like) planes $v^{\perp}_{j_1}$, $v^{\perp}_{j_2}$, $v^{\perp}_{j_3}$ (intersecting at a common point which is exactly the ideal vertex $v$) of $v_{j_1}, v_{j_2}, v_{j_3}$, where $P^*_j$ is oriented with outward-pointing time-like normal vectors. Note that $\dot{P^*_0}$ preserves all the edge lengths of  $P^*_0$ at first order, then the corresponding infinitesimal deformation, denoted by $\dot{P}$, of $P$ preserves the dihedral angle at each edge of $P$ at first order. Combined with Proposition \ref{Prop:deformation at ideal vertex}, the restriction of $\dot{P}$ to the ideal vertex $v$ is tangent to $\partial \AdS$, hence $\dot{P}$ preserves the sum of dihedral angles at the edges of the (unbounded) polyhedron bounded by at least three planes (each of which contains a face of $P$ adjacent to $v$) at first order. This implies that $\dot{P}$ preserves the sum of dihedral angles at the edges adjacent to the vertex $v$ of $P^*_j$ at first order. As a consequence, $\dot{P^*_0}$ preserves the sum of the (signed) lengths of directed edges of $T_j$ at first order and thus preserves the (signed) length of $e_{j_3}$ at first order. The claim follows.

If $f_v$ is a triangle, Claim \ref{clm:edge lengths} follows. Otherwise, we do the following procedure:

\textbf{Step 1.} We consider the triangles in $T$ with exactly two edges contained in $\partial f_v$, denoted by $T_{n_1}, ..., T_{n_{k_1}}$. 
For each $T_j$ with $j\in\{n_1,..., n_{k_1}\}$, note that $\dot{P}^*_0$ preserves the (signed) lengths of the two directed edges of $T_j$ contained in $\partial f_v$ at first order. By Step 0,  $\dot{P}^*_0$ also preserves the (signed) length of the third directed edge of $T_j$ at first order.
 On the other hand, for each edge $e$ in the triangulation $T$, if $\dot{P}^*_0$ preserves the (signed) length of a directed edge underlying $e$ at first order, then $\dot{P}^*_0$ also preserves the (signed) length of the reversely directed edge underlying $e$ at first order.

\textbf{Step 2.} If $\{T_j\}_{j=1}^m\setminus\{T_{n_1}, ..., T_{n_{k_1}}\}$ is empty or the remaining triangles are those with edges either contained in $\partial f_v$ or belonging to some triangles in $\{T_{n_1}, ..., T_{n_{k_1}}\}$, we are done. Otherwise, we consider the triangles, collected by $\{T_{n_{k_1+1}}, ..., T_{n_{k_2}}\}$, with  exactly one edge neither contained in $\partial f_v$ nor belonging to any triangle in $\{T_{n_1}, ..., T_{n_{k_1}}\}$. By the assumption of $\dot{P}^*_0$ and applying  Step 0 again as above to the triangles $T_{n_{k_1+1}}, ..., T_{n_{k_2}}$, then $\dot{P}^*_0$ preserves the lengths of the third (directed) edges of those triangles at first order.

\textbf{Step 3.} Repeat the same procedure as Step 2, by replacing $\{T_{n_1}, ..., T_{n_{k_1}}\}$ with $\{T_{n_1}, ..., T_{n_{k_1}}, T_{n_{k_1+1}}, ..., T_{n_{k_2}}\}$. After finitely many steps, we have that $\{T_j\}_{j=1}^m\setminus\{T_{n_1}, ..., T_{n_{k_1}}, ..., T_{n_{k_N}}\}$ is either empty or the remaining triangles are those with edges either contained in $\partial f_v$ or belonging to some triangles in $\{T_{n_1}, ..., T_{n_{k_1}}, ..., T_{n_{k_N}}\}$ for some $N\in\mathbb{N}^+$. This concludes the proof that $\dot{P}^*_0$ preserves the (signed) lengths of all the directed edges of $T_j$ at first order for $1\leq j\leq m$.
\end{proof}

\begin{proposition}\label{prop:angle rigidity}
Let $P\in\cP$ and let $P^*_0$ be the dual polyhedron of the truncated polyhedron $P_0$ of $P$.  Let $\dot{P}^*_0$ be an infinitesimal deformation of $P^*_0$. If $\dot{P}^*_0$ preserves all the edge lengths of $P^*_0$ at first order, then $\dot{P}^*_0$ is trivial.
\end{proposition}

\begin{proof}
We discuss $P$ according to the positions of its vertices with respect to $\partial \AdS$.

\textbf{Case 1.} All the vertices of $P$ are strictly hyperideal.

In this case, each vertex of $P_0$ has degree three. Therefore, each face of the polyhedron $P^*_0$ dual to $P_0$ is a triangle. Moreover, all the vertices of $P^*_0$ are disjoint from $\partial \AdS$. Since $\dot{P_0^*}$ does not change the edge lengths of $P^*_0$ at first order and all the faces of $P^*_0$ constitute a triangulation of the boundary surface $\partial P^*_0$, by Lemma  \ref{lm:edge lengths}, $\dot{P}^*_0$ is a first-order isometric deformation of $P^*_0$. Combined with Proposition \ref{prop:infPogorelov}, $\dot{P}_0$ is trivial.

\textbf{Case 2.} At least one vertex of $P$ is ideal.

In this case, we still consider the polyhedron $P^*_0$ dual to the truncated polyhedron $P_0$ (which coincides with $P$ if and only if all the vertices of $P$ are ideal). For each ideal vertex $v$, we denote by $f_v$ the face of $P^*_0$ dual to $v$, which is contained in a light-like plane $v^{\perp}$. In particular, all the vertices of $f_v$ are contained in $\AdS$ and thus disjoint from $\partial\AdS$.

Note that $f_v$ is not necessarily a triangle. Moreover, if $P^*_0$ has a non-triangular face, then it must be the dual face of an ideal vertex of $P_0$. After taking a triangulation (without adding new vertices) of each non-triangular face of $P^*_0$, we obtain a triangulation of $\partial P^*_0$, say $T$. By the assumption of $\dot{P}^*_0$ and Claim \ref{clm:edge lengths}, $\dot{P}^*_0$ preserves all the edge lengths of the triangulation $T$ at first order. Applying Lemma \ref{lm:edge lengths} and Proposition \ref{prop:infPogorelov} again, we conclude that $\dot{P}^*_0$ is trivial.
\end{proof}

\begin{proof}[\textbf{Proof of Lemma \ref{lem:local immersion_angles}}]
This follows immediately from Proposition \ref{prop:angle rigidity} and the equivalence between the infinitesimal rigidity of $P$ with respect to dihedral angles and that of $P^*_0$ with respect to edge lengths.
\end{proof}

This also provides an alternative method to prove the infinitesimal rigidity with respect to dihedral angles at the edges of hyperideal hyperbolic polyhedra (see \cite{bao-bonahon}) and the infinitesimal rigidity with respect to dihedral angles at the edges of ideal AdS polyhedra (see \cite{DMS}).

\subsection{The proof of Theorem \ref{thm:homeo_angles}}\label{sec:proof_angles}

 We first recall a known result of \cite{DMS} for the ideal case,  which concerns the parameterization of the space $\cP'_N$ of all marked non-degenerate convex ideal polyhedra in $\AdS$ with $N$ vertices (up to isometries) in terms of dihedral angles.

Let $\cA'_N$ denote the disjoint union of the spaces $\cA'_{\Gamma}$ of weight functions (satisfying certain angle conditions) on the edges of $\Gamma$ over $\Gamma\in\Graph(\Sigma_{0,N},\gamma)$, glued together along faces corresponding to common subgraphs. In this case, the angle conditions are a modification of Conditions (i)-(iv) in Definition \ref{def:admissible angles} obtained by replacing the inequalities $\theta(e_1)+...+\theta(e_k)\geq 0$ in Condition (ii) for each vertex by the equality $\theta(e_1)+...+\theta(e_k)= 0$ (see also \cite[Definition 1.3]{DMS}).

It was proved in \cite[Theorem 1.4]{DMS} that the map $\Psi':\cP'_N\rightarrow \cA'_N$ which assigns to a polyhedron $P'\in\cP'_N$ its dihedral angle at each edge is a homeomorphism.

\subsubsection{The case $N\geq5$}

It follows from Lemma \ref{lem:local immersion_angles} that for each triangulation $\Gamma\in\cG_N$, the angle-assignation map $\Psi: \cP_N\rightarrow \R^{E(\Gamma)}$ is a local immersion near each $P\in\cP_N$ whose 1-skeleton is a subgraph of the $\Gamma$. By Claim \ref{clm:topology of A_Gamma} and Claim \ref{clm:dimension of P_Gamma}, the spaces $\cA_{\Gamma}$ and $\cP_{\Gamma}$ have the same dimension $3N-6$ for a triangulation $\Gamma\in\cG_N$. Therefore, the map $\Psi: \cP_N\rightarrow \R^{E(\Gamma)}$ is a local homeomorphism near each $P\in\cP_N$ whose 1-skeleton is a subgraph of the $\Gamma$. Note that the map $\Psi$, pieced together from $\Psi_{\Gamma}$ over all $\Gamma\in\Graph(\Sigma_{0,N},\gamma)$, is an open map by the definition of the topology of the 
space $\cA_N$ (see subsection \ref{subsec:def}). As a consequence the map $\Psi: \cP_N\rightarrow \cA_N$ is a local homeomorphism in a neighbourhood of any point $P$ in the closure of the stratum $\cP_{\Gamma}$ of $\cP_N$ for each $\Gamma\in\cG_N$. Therefore, $\Psi$ is a local homeomorphism. Combined with the properness (see Lemma \ref{lem:properness of psi}), $\Psi$ is a covering. Moreover, $\Psi$ is an $m$-sheeted covering for a positive integer $m$, since $\cA_N$ is connected (see Proposition \ref{prop:topology of A}).

 It suffices to show that $m=1$. Indeed, let $\theta\in\cA'_N$ be an angle assignment. We claim that $\theta$ has a unique preimage in $\cP_N$. Recall the known result (see \cite[Theorem 1.4]{DMS}) that the parameterization map $\cP'_N\rightarrow\cA'_N$, which is exactly the restriction of $\Psi$ to the space $\cP'_N$, is a homeomorphism. Therefore, $\theta$ has a unique preimage in $\cP'_N$. It remains to show that $\theta$ has no preimage in $\cP_N\setminus\cP'_N$. Otherwise, there is a polyhedron say $P\in\cP_N\setminus\cP'_N$ whose dihedral angle-assignment is $\theta$. However, $P$ has at least one strictly hyperideal vertex say $v$, by the necessity (see Proposition \ref{prop: necessary_angle} for the proof of Condition (ii)), the sum of the dihedral angles at the edges adjacent to $v$ is greater than 0. This implies that $\theta\not\in\cA'_N$, which leads to contradiction. As a consequence, $\Psi$ is a one-sheeted covering and is therefore a homeomorphism.

\subsubsection{The case $N=4$}

In this case, $\Graph(\Sigma_{0,N},\gamma)=\cG_N$ and it consists of exactly two graphs, say $\Gamma_1$, $\Gamma_2$, which are both triangulations of $\Sigma_{0,N}$. According to the gluing construction of $\cP_{\Gamma}$ (resp. $\cA_{\Gamma}$) for $\cP_N$ (resp. $\cA_N$), $\cP_N$ (resp. $\cA_N$) has exactly two connected components $\cP_{\Gamma_1}$ and $\cP_{\Gamma_2}$ (resp. $\cA_{\Gamma_1}$ and $\cA_{\Gamma_2}$). Note that the space $\cP'_N$ (resp. $\cA'_N$) also has two connected components corresponding to $\Gamma_1$ and $\Gamma_2$ (see e.g. \cite[Section 7.3]{DMS}). Applying the same argument as the case $N\geq 5$ to the angle-assignation map $\Psi_{\Gamma_i}: \cP_{\Gamma_i}\rightarrow \cA_{\Gamma_i}$ for $i=1,2$, we have that $\Psi_{\Gamma_i}$ is a homeomorphism, which implies that $\Psi$ is a homeomorphism.

\subsection{The proof of Theorem \ref{thm:homeo_metrics}}\label{sec:proof_metric}

 We recall another known result of \cite{DMS} for the ideal case, which provides the parameterization of the space $\cP'_N\cup\polyg'_N$ of all marked non-degenerate and degenerate convex ideal polyhedra in $\AdS$ with $N$ vertices (up to isometries) in terms of induced metric on the boundary of polyhedra. Let $\cT'_{0,N}$ denote the space of complete hyperbolic metrics on $\Sigma_{0,N}$ with finite area, considered up to isotopy fixing each marked point. It was shown (see \cite[Theorem 1.5]{DMS}) that the map $\Phi':\cP'_N\cup\polyg'_N\rightarrow \cT'_{0,N}$ which takes a polyhedron $P'\in\cP'_N\cup\polyg'_N$ to the induced metric on $\partial P'$ is also a homeomorphism.

 In this proof we discuss directly the case of $N\geq 3$, since the space $\cT_{0,N}$ is connected for $N\geq 3$. The argument is almost the same as above for the map $\Psi$ in the case $N\geq 5$. We include the proof for completeness. Indeed, the restriction to $\cP^{(\epsilon_1, \ldots, \epsilon_N)}$ of the map $\Phi: \cP_N\cup\polyg_N\rightarrow \cT_{0,N}$ is a local immersion to $\cT^{(\epsilon_1, \ldots, \epsilon_N)}_{0,N}$ by Lemma \ref{lem:local immersion_metrics}.
 Note that the spaces $\cP^{(\epsilon_1, \ldots, \epsilon_N)}$ and $\cT^{(\epsilon_1, \ldots, \epsilon_N)}_{0,N}$ are smooth manifolds of the same dimension $2N-6+\epsilon_1+\ldots+ \epsilon_N$ for all $N\geq 3$ (see subsection \ref{subsec:metric} and Proposition \ref{prop:dimension of bP}),
so the restriction to $\cP^{(\epsilon_1, \ldots, \epsilon_N)}$ of $\Phi$  is a local homeomorphism to $\cT^{(\epsilon_1, \ldots, \epsilon_N)}_{0,N}$ for each $(\epsilon_1, \ldots, \epsilon_N)\in\{0,1\}^N$. Combined with the fact that $\cP_N\cup\polyg_N$ (resp. $\cT_{0,N}$) is the disjoint union of $\cP^{(\epsilon_1, \ldots, \epsilon_N)}$ (resp. $\cT^{(\epsilon_1, \ldots, \epsilon_N)}_{0,N}$) over all $(\epsilon_1, \ldots, \epsilon_N)\in\{0,1\}^N$ and the gluing structure (see subsection \ref{subsec:metric}) of the subspaces $\cP^{(\epsilon_1, \ldots, \epsilon_N)}$ (resp. $\cT^{(\epsilon_1, \ldots, \epsilon_N)}_{0,N}$ ) in $\cP_N\cup\polyg_N$ (resp. $\cT_{0,N}$), $\Phi$ is a local homeomorphism.

Combined with the properness of $\Phi$ (see Lemma \ref{lem:proper_metrics}) and the fact that $\cT_{0,N}$ is connected, $\Phi$ is an $m$-sheeted covering for some positive integer $m$. Using the known result (see \cite[Theorem 1.5]{DMS}) that the map $\Phi':\cP'_N\cup\polyg'_N\rightarrow\cT'_{0,N}$, which is exactly the restriction of $\Phi$ to $\cP'_N\cup\polyg'_N$, is a homeomorphism and the fact that the complete metric induced near a strictly hyperideal vertex on the boundary of a hyperideal polyhedron has infinite area, we can show that each metric $h\in\cT'_{0,N}$ has a unique preimage in $\cP_N\cup\polyg_N$, which implies that $m=1$. This concludes that $\Phi$ is a homeomorphism.

\bibliographystyle{amsalpha}
\bibliography{/home/jean-marc/Dropbox/papiers/outils/biblio}
\end{document}